\documentclass[a4paper,12pt]{amsart}
\usepackage[utf8]{inputenc}
\usepackage[english]{babel}
\usepackage[a4paper,twoside,top=1in, bottom=0.9in, left=0.7in, right=0.7in]{geometry} 
\usepackage{graphicx}
\usepackage{hyperref}
\usepackage{amsmath}
\usepackage{amsthm}
\usepackage{amssymb}
\usepackage{amsfonts}
\usepackage{esint} 
\usepackage{mathrsfs} 
\usepackage{xcolor}
\usepackage{array}
\usepackage{hhline}
\usepackage{enumitem} 
\usepackage{imakeidx} 
\usepackage{xparse} 
\usepackage{mathtools}
\usepackage{float}
\usepackage{comment}
\usepackage{nicefrac}
\usepackage{graphicx}
\usepackage{enumerate}

\definecolor{newblue}{rgb}{0.2, 0.3, 0.85}
\hypersetup{colorlinks=true, linkcolor=newblue, citecolor=newblue, urlcolor  = newblue} 

\newcommand{\X}{{\rm X}}
\newcommand{\mm}{\mathfrak m}
\newcommand{\sfd}{{\sf d}}
\newcommand{\ppi}{{\mbox{\boldmath\(\pi\)}}}

\newcommand{\nchi}{{\raise.3ex\hbox{\(\chi\)}}}

\usepackage{fancyhdr} 
\usepackage{ifthen} 
\usepackage{forloop} 
\usepackage{xstring} 
\usepackage{blindtext} 
\usepackage{nomencl}
\usepackage{dsfont} 
\usepackage{faktor} 
\usepackage{tikz}
\usetikzlibrary{arrows}

\numberwithin{equation}{section}

\usepackage[english,capitalize]{cleveref}

\usepackage[maxbibnames=99,backend=biber, doi=false, url=false]{biblatex}
\usepackage[colorinlistoftodos,textwidth=2.3cm]{todonotes}

\definecolor{dgreen}{rgb}{0.0, 0.56, 0.0}

\newcommand{\N}{\ensuremath{\mathbb N}}
\newcommand{\R}{\ensuremath{\mathbb R}}
\newcommand{\meas}{\mathfrak{m}}
\newcommand{\lip}{{\rm lip \,}}

\newcommand{\st}{\ensuremath{\ :\ }} 
\newcommand{\eqdef}{\ensuremath{\vcentcolon=}}
\newcommand \eps{\ensuremath{\varepsilon}} 
\renewcommand{\epsilon}{\varepsilon}

\newcommand{\de}{\ensuremath{\,\mathrm d}} 
\renewcommand{\d}{\ensuremath{\mathrm d}} 

\DeclareMathOperator{\sn}{sn}

\newcommand{\CD}{\mathsf{CD}}
\newcommand{\RCD}{\mathsf{RCD}}
\newcommand{\MCP}{\mathsf{MCP}}

\newcommand{\dist}{\mathsf{d}} 
\newcommand{\haus}{\mathcal{H}} 
\newcommand{\fr}{\penalty-20\null\hfill\(\blacksquare\)}

\let\div\undefined
\DeclareMathOperator{\div}{div}

\newcommand{\lapl}{\ensuremath{\Delta}}

\newcommand{\ric}{\ensuremath{\mathrm{Ric}}} 
\DeclareMathOperator{\vol}{vol}

\theoremstyle{plain}
\newtheorem{theorem}{Theorem}[section] 
\theoremstyle{plain}

\theoremstyle{plain}
\newtheorem{proposition}[theorem]{Proposition}
\theoremstyle{plain}
\newtheorem{lemma}[theorem]{Lemma}
\theoremstyle{plain}
\newtheorem{corollary}[theorem]{Corollary}
\theoremstyle{definition}
\newtheorem{definition}[theorem]{Definition} 
\theoremstyle{definition}
\newtheorem{remark}[theorem]{Remark}
\theoremstyle{definition}

\theoremstyle{definition}

\title{Isoperimetric sets in spaces with lower bounds on the Ricci curvature}

\author[Gioacchino Antonelli]{Gioacchino Antonelli}\address{Scuola Normale Superiore, Piazza dei Cavalieri, 7, 56126 Pisa, Italy.}\email{gioacchino.antonelli@sns.it}

\author[Enrico Pasqualetto]{Enrico Pasqualetto}\address{Scuola Normale Superiore, Piazza dei Cavalieri, 7, 56126 Pisa, Italy.}\email{enrico.pasqualetto@sns.it}

\author[Marco Pozzetta]{Marco Pozzetta}\address{Dipartimento di Matematica e Applicazioni, Universit\`a di Napoli Federico II, Via Cintia, Monte S. Angelo, 80126 Napoli, Italy.}\email{marco.pozzetta@unina.it}

\date{\today}

\begin{document}

\maketitle

\begin{abstract}
In this paper we study regularity and topological properties of volume constrained minimizers of quasi-perimeters in $\RCD$ spaces where the reference measure is the Hausdorff measure. A quasi-perimeter is a functional given by the sum of the usual perimeter and of a suitable continuous term. In particular, isoperimetric sets are a particular case of our study. 

We prove that on an $\RCD(K,N)$ space $(\X,\dist,\mathcal{H}^N)$, with $K\in\mathbb R$, $N\geq 2$, and a uniform bound from below on the volume of unit balls, volume constrained minimizers of quasi-perimeters are open bounded sets with $(N-1)$-Ahlfors regular topological boundary coinciding with the essential boundary. 

The proof is based on a new Deformation Lemma for sets of finite perimeter in $\RCD(K,N)$ spaces $(\X,\dist,\meas)$ and on the study of interior and exterior points of volume constrained minimizers of quasi-perimeters.

The theory applies to volume constrained minimizers in smooth Riemannian manifolds, possibly with boundary, providing a general regularity result for such minimizers in the smooth setting.
\end{abstract}

\tableofcontents


\medskip


\section{Introduction}

On a space having notions of \emph{volume} and \emph{perimeter} on a subfamily of its subsets it makes sense to formulate the \emph{isoperimetric problem}. The most classical formulation of the problem aims at minimizing the perimeter among sets having a fixed volume. The development of the theory of sets of finite perimeter on metric measure spaces $(\X,\dist,\meas)$, see \cite{Ambrosio02, AmbrosioAhlfors, Miranda03, AmbrosioDiMarino14}, 
makes such ambient spaces a general framework where to set up the isoperimetric problem, naturally including the setting of smooth Riemannian manifolds. As for the problem formulated on the Euclidean spaces or in smooth Riemannian manifolds, it is useful to develop tools and a basic regularity theory able to treat \emph{isoperimetric sets}, i.e., minimizers of the isoperimetric problem, which therefore solve a minimization problem under a \emph{volume constraint}. In this paper, we want to develop part of this fundamental theory in the nonsmooth setting of $\RCD(K,N)$ metric measure spaces, considering the regularity theory for volume constrained minimizers of functionals called \emph{quasi-perimeters}, which include the usual perimeter as a particular case. Already in the classical smooth setting, the basic regularity theory for sets minimizing a functional under a volume constraint is much more involved than the theory for minimizers without constraints  \cite{GonzalezMassariTamanini, Xia05}. Indeed, roughly speaking, comparison and deformation arguments on a minimizer must preserve its volume.

The first achievement of the paper is a Deformation Lemma for sets of finite perimeter in the realm of $\RCD(K,N)$ spaces with $K\in\mathbb R$ and $N<+\infty$. We denote by $P(E)$ the perimeter of a set $E$, see \cref{def:BVperimetro}.

\begin{theorem}[Deformation Lemma]\label{thm:MainEst}
Let \((\X,\sfd,\mm)\) be an \({\sf RCD}(K,N)\) space with \(N<\infty\) and let $R>0$. Then there exists a constant
\(C_{K,N,R}>0\) such that the following holds. If $\bar x \in \X$ and \(E\subset\X\) is a set of finite perimeter, then
\begin{equation}\label{eq:main_claim2}
	\,P\big(E\setminus B_r(\bar x)\big)\leq \frac{C_{K,N,R}}{r}\,\mm\big(E\cap B_r(\bar x)\big)+\,
	P(E),\quad\text{for every}\,\,r\in(0,R),
\end{equation}
\begin{equation}\label{eq:main_claim2NEW}
	\,P\big(E\cup B_r(\bar x)\big)\leq \frac{C_{K,N,R}}{r}\,\mm\big( B_r(\bar x)\setminus E\big)+\,
	P(E),\quad\text{for every}\,\,r\in(0,R).
\end{equation}
\end{theorem}

We stress the key role of the lower curvature and finite upper dimension bound in the previous statement. The proof of the previous statement will follow by a careful use of the Gauss--Green formula, and, in order to perform the estimate, the Laplacian comparison for the distance function will be of key importance. The last tool is classically known to be deeply linked to lower curvature bounds. Moreover, in order to perform the proof of the last result we will crucially need a second order differential calculus, which is currently available on $\RCD$ spaces and not on $\CD$ or $\MCP$ spaces, for example.

The previous result is quietly well-known in the Euclidean spaces, see e.g., \cite[Equation (8)]{GonzalezMassariTamanini} and references therein. It roughly says that, by adding (or subtracting) a ball to a set of finite perimeter, the perimeter increases in a controlled way with respect to the volume. In particular, for $r\ge r_0>0$ in \cref{thm:MainEst}, the perimeter of the deformed set given by adding (or subtracting) a ball is bounded \emph{linearly} with respect to the variation of the measure; this fact is the foundation of several classical arguments in different settings \cite{GonzalezMassariTamanini, Xia05, Nar14, CintiPratelli, PratelliSaracco}. Such a deformation with bound is also easily obtained in the smooth setting since one has a representation of the weak gradient of the characteristic function of a set as a vector valued measure, see \cite[Lemma 17.21]{MaggiBook}.

On the contrary, the proof of \cref{thm:MainEst} is nontrivial in the nonsmooth context of $\RCD$ spaces. The main idea to obtain such a result is to apply the recently obtained Gauss--Green formula for sets of finite perimeter in $\RCD$ spaces, see \cite[Theorem 2.2]{BPS19}, to the vector field $\nabla\dist_{\bar x}^2$, where $\dist_{\bar x}(\cdot):=\dist(\bar x,\cdot)$ is the distance function from $\bar x$. This is done in \cref{thm:key_estimate}. On the other hand, $\div(\nabla\dist_{\bar x}^2)=\lapl\dist_{\bar x}^2$ is a signed Radon measure \cite{Gigli12}, and thus $\nabla\dist_{\bar x}^2$ does not enjoy the regularity properties to directly apply the Gauss--Green formula in \cite[Theorem 2.2]{BPS19}. Hence one has to perform a careful smoothing argument,
based on the mollified heat flow for vector fields (cf.\ \cref{Mollified}). For a comparison of this approach with another possible one using the results in \cite{BCM20}, see \cref{rem:BCM}.
In order to conclude the proof of \cref{thm:key_estimate}, and in turn to obtain the result in \cref{thm:MainEst}, one also needs to couple the previous Gauss--Green formula with a proper integral version of the equality $\nabla\dist_{\bar x}^2(\cdot)=2\dist_{\bar x}(\cdot )\nu_{B_{\dist_{\bar x}(\cdot)}(\bar x)}$, where $\nu$ denotes the outer unit normal to balls centered at $\bar x$. The latter result is contained in \cref{prop:nu_sphere}. Let us mention that the analogue of \cref{thm:key_estimate}, that is the key step to obtain \cref{thm:MainEst}, in the setting of Riemannian manifolds with bound below on the Ricci tensor is contained in \cite[Lemma 4.8]{NardulliOsorio}.

The result in \cref{thm:MainEst}, when applied at an interior (or an exterior) point $\bar x$ of $E$, if any, yields the possibility of taking $r$ bounded below away from zero, realizing the aforementioned deformation of $E$ controlling the perimeter linearly as a function of the variation of the volume. Such an observation, itself sometimes called Deformation Lemma in the literature, is of crucial importance to obtain regularity results for the set $E$. A Deformation Lemma is already contained in Almgren's paper \cite[VI.2(3)]{AlmgrenBook}, and its importance is clear also in Morgan's book \cite[Lemma 13.5]{MorganBook}. Moreover, the statement of a Deformation Lemma for Euclidean spaces is in \cite[Lemma 17.21]{MaggiBook}; for contact sub-Riemannian manifolds in \cite[Lemma 4.5]{GalliRitore}; and for sub-Finsler nilpotent Lie groups in \cite[Lemma 3.6]{Pozuelo}. The analogous version of such Deformation Lemmas in $\RCD$ spaces, which is a direct consequence of \cref{thm:MainEst}, is contained in \cref{thm:VariazioniMaggi}.

From the previous discussion it seems clear that in order to apply the Deformation Lemma to obtain regularity results of volume constrained minimizers, one should first show that they have interior and exterior points. This is exactly what the authors prove in \cite[Theorem 1]{GonzalezMassariTamanini} in the setting of volume constrained minimizers in open sets of $\mathbb R^n$, and in \cite[Theorem 4.3]{Xia05} in the same setting but for minimizers of functionals of the form $P+G$, called \emph{quasi-perimeters}, where $P$ is the perimeter and $G$ is a suitable continuous term.  For the next result we adapt the strategy of \cite{Xia05} to show that, in the setting of $\RCD(K,N)$ spaces $(\X,\dist,\mathcal{H}^N)$, with $2\leq N<+\infty$ and $K\in\mathbb R$, volume constrained minimizers for suitable quasi-perimeters in open sets have interior and exterior points.  

For the exposition of the next results of the paper, we need to introduce some terminology. Let $(\X,\dist,\haus^N)$ be an $\RCD(K,N)$ space with $N\ge 2$ natural number, and $K\in\mathbb R$. Let $\Omega\subset \X$ be an open set. Let us consider a functional
\[
G:\left\{\text{$\haus^N$-measurable sets in $\Omega$} \right\}/\sim \,\,\to (-\infty,+\infty],
\]
where $E\sim F$ if and only if $\haus^N(E\Delta F)=0$, where $E\Delta F$ denotes the symmetric difference between $E$ and $F$, such that
\begin{equation}\label{eq:ConditionGIntro}
    \begin{array}{lll}
    &G(\emptyset)<+\infty,&\\
    &\forall\widetilde\Omega\Subset \Omega \text{ bounded open }\,\,&\exists C_G>0, \sigma>1-\frac1N \st \\ 
    &   &G(E)\le G(F) + C_G\haus^N(E\Delta F)^\sigma ,
    \end{array}
\end{equation}
for any Borel sets $E,F \subset \Omega$ such that $E\Delta F\subset\widetilde\Omega$. Observe that both $C_G$ and $\sigma$ may depend on $\widetilde\Omega$.

For such a function $G$, we define the \emph{quasi-perimeter $\mathscr P$ restricted to $\Omega$} by
\begin{equation*}
    {\mathscr P}(E,\Omega) \eqdef P(E,\Omega) + G(E\cap \Omega),
\end{equation*}
for any $\haus^N$-measurable set $E$ in $\Omega$. If $\Omega=\X$, we simply write $\mathscr P(E)\eqdef \mathscr P(E, \X) $. The class of quasi-perimeters clearly comprises many examples of energy functionals; a classical case consists in functionals appearing in the problem of the prescription of the mean curvature of a set, while a further application of the regularity theory developed in this work is contained in \cite{ABFP21}.

\begin{definition}\label{def:MinimizersIntro}
Let $(\X,\dist,\haus^N)$, $\Omega\subset \X$, $G$, $\mathscr P$ be as above.

We say that a set of locally finite perimeter $E \subset \X$ is a \emph{volume constrained minimizer of $\mathscr P$ in $\Omega$} if for any $\haus^N$-measurable set $F$ such that there is a compact set $K\subset \Omega$ with $\haus^N((E\Delta F)\setminus K)=0$, and such that $\haus^N(F \cap K)=\haus^N(E \cap K)$, it occurs that
\[
\mathscr P(E,\Omega) \le \mathscr P(F,\Omega).
\]

If $\Omega=\X$, $\haus^N(E)<+\infty$, and $E$ satisfies $\mathscr P(E) \le \mathscr P(F)$ for any $F$ with $\haus^N(F)=\haus^N(E)$, we say that $E$ is a \emph{volume constrained minimizer of $\mathscr P$}.
\end{definition}

We are then ready to state the second main achievement of this paper. We prove, as a first fundamental step, that on every $(\X,\dist,\mathcal{H}^N)$ that is an $\RCD(K,N)$ space with $2\leq N<+\infty$ and $K\in\mathbb R$, every volume constrained minimizer of a quasi-perimeter $\mathscr{P}$ in an open set $\Omega$ has interior (and exterior) points, i.e., points $x$ (resp.\! $y$) such that $B_r(x)\setminus E$ is negligible (resp.\! $B_s(y)\setminus E$ has full measure in \(B_s(y)\)) for some $r>0$ (resp.\! some $s>0$). Then we use the latter information, coupled with the Deformation Lemma discussed above, to obtain stronger regularity results.

We recall that $E^{(t)}$ is the set of points with density $t$ in $E$ with respect to $\mathcal{H}^N$, $\partial^e E\eqdef \X\setminus (E^{(0)}\cup E^{(1)})$ is the essential boundary of $E$, and $\partial F$ is the topological boundary of any set $F$. For the notion of Ahlfors regular set we refer the reader to \cref{def:Ahlfors}.

\begin{theorem}\label{thm:MainIntro2}
    Let $(\X,\dist,\haus^N)$ be an $\RCD(K,N)$ space with $2\leq N<+\infty$ natural number, and $K\in\mathbb R$ . Let $\Omega\subset \X$ be open, and let $\mathscr P=P+G$ be the quasi-perimeter restricted to $\Omega$ associated to $G$ as in \eqref{eq:ConditionGIntro}. Let $E\subset \X$ be a volume constrained minimizer of $\mathscr P$ in $\Omega$, and assume $P(E,\Omega)>0$.
    
    Hence $E\cap\Omega$ has both interior and exterior points, $E^{(1)}\cap \Omega$ is open, $\partial^e E \cap \Omega = \partial E^{(1)} \cap \Omega$, and $\partial E^{(1)}$ is locally $(N-1)$-Ahlfors regular in $\Omega$.
    
    If we further assume that the constants $C_G,\sigma$ in \eqref{eq:ConditionGIntro} are uniform on the choice of $\widetilde\Omega$, and that there exists $v_0>0$ such that $\mathcal{H}^N(B_1(x))\geq v_0$ for every $x\in \X$, then $\partial E^{(1)}$ is (globally) $(N-1)$-Ahlfors regular in $\Omega$. 
\end{theorem}

From the last part of the previous Theorem it follows that, under the hypotheses described in there, a volume constrained minimizer of $\mathscr{P}$ in $\Omega$ has an open representative with $(N-1)$-Ahlfors regular topological boundary coinciding with its essential boundary in $\Omega$, which is precisely $E^{(1)}\cap\Omega$.

Let us give a hint of the proof of \cref{thm:MainIntro2}. We adapt the strategy of \cite{Xia05}, which is, in turn, based on \cite{GonzalezMassariTamanini}, not without some difficulty, as the technical part has to be necessarily different. A major obstacle is the fact that on $\RCD$ spaces one does not have at disposal an isoperimetric inequality with the Euclidean constant. Nevertheless, by exploiting the results in \cite{CavallettiMondinoAlmostEuclidean}, we notice that locally around points of density one with respect to $\mathcal{H}^N$ in $(\X,\dist,\mathcal{H}^N)$ there holds an almost Euclidean isoperimetric inequality, see \cref{prop:AlmostEuclIsop}. The latter, together with an adaption of a Lemma by Morgan--Johnson in the nonsmooth case, see \cref{prop:MorganJohnson}, and a volume decay estimate \cref{lem:VolumeDecayEstimate}, which is ultimately linked to the relative isoperimetric inequality \cref{prop:RelativeIsoperimetricInequality}, allows us to import in our setting the machinery of \cite{Xia05}, see \cref{prop:KeyAlternative}, \cref{lem:ExistenceSequenceBalls}, and finally \cref{thm:InteriorAndExteriorPoints}.

A key tool for proving the second part of \cref{thm:MainIntro2} is the fact that, whenever one has a uniform bound from below on the volumes of unit balls, an isoperimetric inequality for small volumes holds, see \cref{prop:IsopVolumiPiccoli}, and \cref{rem:IsoperimetricaSiPuoApplicareARCD}. This is known in the setting of Riemannian manifolds with bound below on the Ricci curvature, see \cite[Lemma 3.2]{Heb00}. We adapt here the proof of \cite[Lemma 3.2]{Heb00} in a large class of PI spaces. Such a class contains $\CD(K,N)$, and thus $\RCD(K,N)$, spaces with $N<+\infty$, $K\in\mathbb R$, and with a uniform lower bound on the volumes of unit balls. We stress that the class of metric measure spaces for which \cref{prop:IsopVolumiPiccoli} holds contains also all the examples discussed in \cite{AmbrosioAhlfors}, among which also Carnot--Carathéodory spaces.

By using the existence of interior and exterior points, the latter isoperimetric inequality for small volumes, and general facts about the theory of $\Lambda$-minimizers, see \cref{sec:LambdaMin}, we get that any volume constrained minimizer is actually a $(k,r)$-quasi minimal set, see \cref{def:QuasiMinimi}. Then the proof is concluded by a classical argument according to which being quasi minimal implies having density estimates at points of the topological boundary, see \cref{prop:FromQuasiMinToOpen}. We stress that the very final part of this argument, i.e., the proof of \cref{prop:FromQuasiMinToOpen}, could be obtained as a consequence of the more general result in \cite[Theorem 4.2]{KKLS13}. Nevertheless, for the readers' convenience, we give in this paper a more direct and short proof inspired by \cite[Theorem 21.11]{MaggiBook}.

In case we are dealing with $G$ defined on the whole $\X$ and with volume constrained minimizers, we can add a piece of information to \cref{thm:MainIntro2}, which is the boundedness of the representative. More precisely, for the next statement we consider
\[
G:\left\{\text{$\haus^N$-measurable sets in $\X$} \right\}/\sim \,\,\to (-\infty,+\infty],
\]
where $E\sim F$ if and only if $\haus^N(E\Delta F)=0$, where $E\Delta F$ denotes the symmetric difference between $E$ and $F$,
such that
\begin{equation}\label{eq:ConditionG2Intro}
    \begin{split}
    G(\emptyset)&<+\infty,\\
    G(E)\le G(F) &+ C_G\haus^N(E\Delta F)^\sigma ,
    \end{split}
\end{equation}
for any Borel sets $E,F \subset \X$. Notice that now the hypothesis on $G$ is required on every couple of sets $E,F$ no matter whether their symmetric difference $E\Delta F$ is compactly contained in $\X$ or unbounded.

\begin{theorem}\label{cor:FINALEIntro}
    Let $(\X,\dist,\haus^N)$ be an $\RCD(K,N)$ space with $2\leq N<+\infty$ natural number, $K\in\mathbb R$, and assume there exists $v_0>0$ such that $\mathcal{H}^N(B_1(x))\geq v_0$ for every $x\in \X$. Let $\mathscr P=P+G$ be the quasi-perimeter associated to a function $G$ defined on the whole $\X$, and let us assume that $G$ satisfies \eqref{eq:ConditionG2Intro}. Let $E\subset \X$ be a volume constrained minimizer of $\mathscr P$, and assume $P(E)>0$.
    
    Then $E^{(1)}$ is open and bounded, $\partial^eE=\partial E^{(1)}$, and $\partial E^{(1)}$ is $(N-1)$-Ahlfors regular in $\X$. 
\end{theorem}

We stress that for the boundedness in the previous statement the lower volume bound on the volume of unit balls is necessary. Indeed, classical examples of collapsed smooth manifolds without boundary having unbounded isoperimetric regions consist in manifolds having ``cuspidal ends'' with finite volume; in this way, for suitable volumes, isoperimetric regions are given by part of such unbounded ends.

Notice that if in the previous statement we take $G\equiv 0$, we obtain as a corollary that isoperimetric sets in $\RCD(K,N)$ spaces with reference measure $\mathcal{H}^N$ and a uniform bound from below on the volumes of unit balls have a bounded open representative with $(N-1)$-Ahlfors regular topological boundary coinciding with its essential boundary. 

The proof of the latter Theorem is based on a rather classical argument involving the previously discussed Deformation Lemma together with an argument that is in turn reduced to an ODE comparison. Such an argument already appeared in \cite[Proposition 3.7]{RitRosales04} and in \cite[Lemma 13.6]{MorganBook} in the Euclidean setting, and in \cite[Theorem 3]{Nar14} on Riemannian manifolds. See also \cite[Appendix B]{AFP21}. 

Similar statements as the one in \cref{cor:FINALEIntro}, but only for isoperimetric sets, already appeared in the literature in other settings as well: in the setting of Carnot groups see \cite[Section 5]{LeonardiRigot}; in the setting of contact sub-Riemannian manifolds, cf.\! \cite[Lemma 4.6]{GalliRitore}; in the setting of sub-Finsler nilpotent Lie groups, see \cite[Section 3]{Pozuelo}. We stress that the latter cases do not fall into the class of $\RCD$ spaces and require different techniques to be treated.

We stress that the result in \cref{cor:FINALEIntro} will be of key importance in order to fully extend the generalized existence result of \cite{Nar14} in the setting of $\RCD(K,N)$ spaces $(\X,\dist,\mathcal{H}^N)$ with $2\leq N<+\infty$, $K\in\mathbb R$, and a uniform bound from below on the volumes of unit balls. A first step in order to get such a generalization has been made in \cite{AFP21} by the first and the last named author together with Fogagnolo. We will complete such a generalization in a forthcoming paper. 

Let us also remark that the investigation in \cite{AFP21} shows how the nonsmooth theory naturally arises also from the study of the isoperimetric problem stated in a perfectly smooth ambient space. We refer also to the forthcoming \cite{ABFP21} in which the authors show new existence results for the isoperimetric problem on nonnegatively Ricci curved manifolds with Euclidean volume growth by exploiting the nonsmooth theory.

A comprehension of the basic properties of isoperimetric sets in $\RCD$ spaces seems, in fact, \emph{necessary} for improving the study of the isoperimetric problem in the smooth realm. Among further significant achievements in the study of the existence and regularity theory for the isoperimetric problem in the smooth setting, let us also mention \cite{LeonardiRitore, MorganRitore02}.

\begin{remark}
Contrary to \cref{thm:MainEst}, the main regularity results in \cref{thm:MainIntro2} and \cref{cor:FINALEIntro} hold on $\RCD(K,N)$ spaces with reference measure $\haus^N$, i.e., noncollapsed $\RCD$ spaces in the sense of \cite{DePhilippisGigli18}. It is unknown whether such regularity results hold on spaces $(\X,\dist,\meas)$ with generic reference measure $\meas$. The arguments developed in the present paper seem not to be easily generalized to such a general setting since they ultimately rely on the coincidence between the analytic dimension of the reference Hausdorff measure $\haus^N$ and the geometric dimension of the $\RCD(K,N)$ condition. Moreover, also the fact that the $N$-dimensional density of $\haus^N$ is bounded above by $1$ plays a key role, understanding that a comprehension of the regularity properties of the density of the measure of a generic space $(\X,\dist,\meas)$ might be necessary for generalizing the theory developed here to such spaces. This will be clear, for instance, in the several applications of the Bishop--Gromov monotonicity (see \cref{rem:PerimeterMMS2}) in the proofs of the results in \cref{sec:IsoperimetricAndComparison}.

We mention that it is today an open problem to understand whether an $\RCD(K,M)$ space $(\X,\dist,\haus^N)$ with $M>N$ is, in fact, an $\RCD(K,N)$ space, cf. \cite[Conjecture 4.2]{Honda20}.
\fr
\end{remark}

We are finally committed to discuss how our analysis applies to Riemannian manifolds, possibly with boundary. Indeed, it is known that the class of $\RCD$ spaces is rich enough to contain also some examples of Riemannian manifolds with boundary, see \cite{Han20}. We refer also the reader to the recent works \cite{DePhilippisGigli18,KM21,BNS20} in which remarkable fine properties of the boundary of $\RCD(K,N)$ spaces with Hausdorff reference measures $\mathcal{H}^N$ are studied. 

We recall that in case $M$ has nonempty boundary the perimeter measure $P(E,\cdot)$ of a set $E$ does not charge the boundary $\partial M$, i.e. $P(E,\partial M)=0$. In other words, $P(E,\cdot)$ is automatically the {\em relative perimeter} in the interior of $M$. We say that the boundary $\partial M$ of a Riemannian manifold is \emph{convex} if the second fundamental form of $\partial M$ with respect to the inner normal direction in $M$ is nonnegatively definite.

\begin{corollary}[Interior regularity of volume constrained minimizers on smooth Riemannian manifolds]\label{cor:Manifold}
Let $2\leq N<+\infty$ be a natural number and let $(M^N,g)$ be a complete Riemannian manifold, possibly with boundary $\partial M$.
Let $\Omega\subset M$ be open, and let $\mathscr P=P+G$ be the quasi-perimeter restricted to $\Omega$ associated to $G$ as in \eqref{eq:ConditionGIntro}. Let $E\subset M$ be a volume constrained minimizer of $\mathscr P$ in $\Omega$, and assume $P(E,\Omega)>0$. 

Hence $E^{(1)}\cap \Omega \setminus \partial M$ is open, $\partial^e E \cap \Omega  \setminus \partial M = \partial E^{(1)} \cap \Omega  \setminus \partial M$, and $\partial E^{(1)}$ is locally $(N-1)$-Ahlfors regular in $\Omega  \setminus \partial M$, see \cref{def:Ahlfors}.
Moreover for every point $x\in \Omega  \setminus \partial M$ there exists $r>0$ such that the reduced boundary $\partial^*E \cap B_r(x)$ is a $C^{1,\alpha}$ open hypersurface, where $\alpha=\alpha(x)>0$, and $\dim_{\rm H}((\partial E\setminus  \partial^*E)\cap\Omega  \setminus \partial M)\leq N-8$.

Assume also one of the following:
\begin{itemize}
    \item[i)]\label{it:Manifold1} $\partial M=\emptyset$ and $\ric\geq K$ on $M$;
    \item[ii)]\label{it:Manifold2} $\partial M$ is orientable and convex, and $\ric\geq K$ on $M\setminus\partial M$.
\end{itemize}
Then, in case ii) holds, we can further say that $E^{(1)}\cap \Omega $ is open, $\partial^e E \cap \Omega   = \partial E^{(1)} \cap \Omega$, and $\partial E^{(1)}$ is locally $(N-1)$-Ahlfors regular in $\Omega$.

Moreover, if we further assume that the constants $C_G, \sigma$ are uniform on the choice of $\widetilde\Omega$ and there exists $v_0>0$ such that $\vol(B_1(x))\geq v_0$ for every $x\in M$, then $\partial E^{(1)}$ is (globally) $(N-1)$-Ahlfors regular in $\Omega$, and the parameter $\alpha$ above is uniform on $\Omega$. 

Finally, if we further assume $\Omega=M$, \eqref{eq:ConditionG2Intro}, and $E$ is a volume constrained minimizer of $\mathscr P$ in $M$, then $E^{(1)}$ is bounded.
\end{corollary}

We mention that a regularity result for (relative) isoperimetric regions in open sets in Riemannian manifolds is well-known in the literature, see \cite[Proposition 2.4]{RitRosales04}, \cite{morgan2003regularity}, and references therein.

\begin{remark}
In the notation and setting of \cref{cor:Manifold}, higher regularity on volume constrained minimizers can be clearly achieved in case $G$ is explicit and suitably smooth. Indeed, in such a case, higher regularity on the $C^{1,\alpha}$ part of $\partial^*E$ can be classically deduced by the minimality properties of the set $E$.
\fr\end{remark}

\bigskip
\noindent\textbf{Organization of the paper.} In \cref{sec:Deformation} we introduce some preliminary notions and facts that we shall use throughout the paper, and we prove the Deformation Lemma stated in \cref{thm:MainEst}. In particular in \cref{sec:Preliminaries} we introduce the notions of perimeter and BV function in metric measure spaces, and we discuss some of their properties with a particular attention to the case when the space is PI. In \cref{sec:DifferentialCalculusRCD} we shall recall the basic notions of Differential Calculus in metric measure spaces, the definition of $\RCD$ space, and we shall recall some geometric and analytic properties of such spaces. In \cref{sec:Deformation2} we study the notions of $(p)$-divergence and mollified heat flow for measures and vector fields. Finally we prove the main results in \cref{thm:MainEst} and \cref{thm:VariazioniMaggi}.

In \cref{sec:VolumeConstrained} we prove the main theorems \cref{thm:MainIntro2}, \cref{cor:FINALEIntro}, and \cref{cor:Manifold}. In \cref{sec:IsoperimetricAndComparison}, building on \cite{CavallettiMondinoAlmostEuclidean}, we prove an almost Euclidean isoperimetric inequality for $\RCD$ spaces; in \cref{sec:Preparatory} we prove a volume decay estimate and several preparatory estimates that will be used to reach the main results; in \cref{sec:InteriorExterior} we give the main definitions of volume constrained minimizers and quasi-perimeters and we prove that volume constrained minimizers have interior and exterior points, see \cref{thm:InteriorAndExteriorPoints}; in \cref{sec:SmallPI} we prove that in PI spaces that are uniformly lower Ahlfors regular an isoperimetric inequality for small volumes holds, see \cref{prop:IsopVolumiPiccoli}; in \cref{sec:LambdaMin} we study the properties of quasi-perimeter minimizers and quasi minimal sets in the setting of $\RCD$ spaces, and then we use them to provide the proof of \cref{thm:MainIntro2}; finally in \cref{sec:LambdaMin} we prove that volume constrained minimizers for quasi-perimeters have bounded representatives, and then we give the proofs of \cref{cor:FINALEIntro} and \cref{cor:Manifold}.

\bigskip
\noindent\textbf{Acknowledgments.} The authors are grateful to Mattia Fogagnolo, Stefano Nardulli, and Daniele Semola, whose suggestions on a preliminary draft of this paper led to an improvement of the presentation. They are grateful to Elia Bru{\`{e}} for inspiring discussions around the topic of the paper. They also thank Francesco Nobili and Ivan Violo for having noticed an inaccuracy in a preliminary version of the paper. The first author is partially supported by the European Research Council (ERC Starting Grant 713998 GeoMeG `\emph{Geometry of Metric Groups}').
The second author acknowledges the support by the Balzan project led by Luigi Ambrosio.

\section{Deformation lemma}\label{sec:Deformation}

\subsection{Preliminaries and auxiliary results}\label{sec:Preliminaries}

In this section we shall recall and prove some preliminary facts concerning the perimeter functional in metric measure spaces, and in particular in PI spaces.

\subsubsection{$BV$ functions and sets of finite perimeter in metric measure spaces}

In this paper, by a \emph{metric measure space} (briefly, m.m.s.)
we mean a triple \((\X,\sfd,\mm)\), where \((\X,\sfd)\) is a complete and separable
metric space, while \(\mm\geq 0\) is a boundedly-finite Borel measure
on \(\X\). Given a locally Lipschitz function $u:\X\to\mathbb R$, \[
\lip u (x) \eqdef \limsup_{y\to x} \frac{|u(y)-u(x)|}{\dist(x,y)}
\]
is the \emph{slope} of $u$ at $x$, for any accumulation point $x\in\X$, and $\lip u(x):=0$ if $x\in\X$ is isolated.

\begin{definition}[$\rm BV$ functions and perimeter on m.m.s.]\label{def:BVperimetro}
Let $(\X,\dist,\meas)$ be a metric measure space.  Given $f\in L^1_{\mathrm{loc}}(\X,\meas)$ we define
\[
|Df|(A) \eqdef \inf\left\{\liminf_i \int_A \lip f_i \de\meas \st \text{$f_i \in {\rm Lip}_{\rm loc}(A),\,f_i \to f $ in $L^1_{\mathrm{loc}}(A,\meas)$} \right\},
\]
for any open set $A\subset \X$.
We declare that a function \(f\in L^1_{\rm loc}(\X,\mm)\) is of
\emph{local bounded variation}, briefly \(f\in{\rm BV}_{\rm loc}(\X)\),
if \(|Df|(A)<+\infty\) for every \(A\subset\X\) open bounded.
A function $f \in L^1(\X,\meas)$ is said to belong to the space of \emph{bounded variation functions} ${\rm BV}(\X)={\rm BV}(\X,\dist,\meas)$ if $|Df|(\X)<+\infty$. 

If $E\subset\X$ is a Borel set and $A\subset\X$ is open, we  define the \emph{perimeter $P(E,A)$  of $E$ in $A$} by
\[
P(E,A) \eqdef \inf\left\{\liminf_i \int_A \lip u_i \de\meas \st \text{$u_i \in {\rm Lip}_{\rm loc}(A),\,u_i \to \nchi_E $ in $L^1_{\rm loc}(A,\meas)$} \right\},
\]
in other words \(P(E,A)\coloneqq|D\nchi_E|(A)\).
We say that $E$ has \emph{locally finite perimeter} if $P(E,A)<+\infty$ for every open bounded set $A$. We say that $E$ has \emph{finite perimeter} if $P(E,\X)<+\infty$, and we denote $P(E)\eqdef P(E,\X)$.
\end{definition}

Let us remark that when $f\in{\rm BV}_{\rm loc}(\X,\dist,\meas)$ or $E$ is a set with locally finite perimeter, the set functions $|Df|, P(E,\cdot)$ above are restrictions to open sets of Borel measures that we still denote by $|Df|, P(E,\cdot)$, see \cite{AmbrosioDiMarino14}, and \cite{Miranda03}.

\medskip

Given any two functions \(f,g\in{\rm BV}_{\rm loc}(\X)\), the following
important properties are verified:
\begin{itemize}
    \item[\(\rm i)\)] \textsc{Locality.} If \(f=g\) on some open set
    \(A\subset\X\), then \(|Df|(B)=|Dg|(B)\) for every Borel set
    \(B\subset A\) with \( \dist(B,A^c)>0\).
    \item[\(\rm ii)\)] \textsc{Subadditivity.} It holds \(|D(f+g)|(B)\leq|Df|(B)+|Dg|(B)\) for every \(B\subset\X\) Borel,
    thus in particular \(f+g\in{\rm BV}_{\rm loc}(\X)\). Moreover, if
    \(f,g\in{\rm BV}(\X)\), then \(f+g\in{\rm BV}(\X)\).
\end{itemize}

\medskip

In the sequel, we shall frequently make use of the following
\emph{coarea formula}, proved in \cite{Miranda03}:
\begin{theorem}[Coarea formula]\label{thm:coarea}
Let \((\X,\sfd,\mm)\) be a metric measure space.
Let \(f\in L^1_{\rm loc}(\X)\) be given. Then for any open set
\(\Omega\subset\X\) it holds that \(\R\ni t\mapsto P(\{f>t\},\Omega)\in[0,+\infty]\) is Borel measurable and satisfies
\[
|Df|(\Omega)=\int_\R P(\{f>t\},\Omega)\,\d t.
\]
In particular, if \(f\in{\rm BV}(\X)\), then \(\{f>t\}\) has finite perimeter
for a.e.\ \(t\in\R\).
\end{theorem}

\begin{remark}[Semicontinuity of the total variation under
$L^1_{\mathrm{loc}}$-convergence]\label{rem:SemicontPerimeter}
Let $(\X,\dist,\meas)$ be a metric measure space.
We recall (cf.\ \cite[Proposition 3.6]{Miranda03}) that whenever $g_i,g\in L^1_{\mathrm{loc}}(\X,\meas)$ are such that $g_i\to g$ in $L^1_{\mathrm{loc}}(\X,\meas)$, for every open set $\Omega$ we have 
$$
|Dg|(\Omega)\leq \liminf_{i\to +\infty}|Dg_i|(\Omega).
$$
On PI spaces
(in the sense of Definition \ref{def:PI} below), we can actually generalize the above lower semicontinuity property. See \cref{prop:Lahti}.
\fr\end{remark}

\begin{definition}[PI space]\label{def:PI}
Let $(\X,\dist,\meas)$ be a metric measure space.
 We say that $\meas$ is {\em uniformly locally doubling} if for every $R>0$ there exists $C>0$ such that the following holds 
 $$
 \meas(B_{2r}(x))\leq C\meas(B_r(x)), \qquad \forall x\in\X\;\forall r\leq R.
 $$

We say that a {\em weak local $(1,1)$-Poincar\'{e} inequality} holds on $(\X,\dist,\meas)$ if there exists $\lambda$ such that for every $R>0$ there exists $C_P$ such that for every pair of functions $(f,g)$ where $f\in L^1_{\mathrm{loc}}(\X,\meas)$, and $g$ is an upper gradient (cf. \cite[Section 10.2]{HajlaszKoskela}) of
$f$,
the following inequality holds
$$
\fint_{B_r(x)} |f-\overline f(x)|\de\meas \leq C_Pr\fint_{B_{\lambda r}(x)} g\de\meas, 
$$
for every $x\in\X$ and $r\leq R$, where $\overline f(x):=\fint_{B_r(x)}f\de\meas$.

We say that $(\X,\dist,\meas)$ is a {\em PI space} when $\meas$ is uniformly locally doubling and a weak local $(1,1)$-Poincar\'{e} inequality holds on $(\X,\dist,\meas)$.
\end{definition}

Let \((\X,\sfd,\mm)\) be a PI space. Given a Borel set \(E\subset\X\)
and \(x\in\X\), we define the \emph{upper density} and the
\emph{lower density} of \(E\) at \(x\) as
\[
\overline D(E,x)\coloneqq\varlimsup_{r\searrow 0}\frac{\mm(E\cap B_r(x))}
{\mm(B_r(x))},\quad\underline D(E,x)\coloneqq\varliminf_{r\searrow 0}
\frac{\mm(E\cap B_r(x))}{\mm(B_r(x))},
\]
respectively. Whenever upper and lower densities coincide, their common
value is denoted by \(D(E,x)\) and called just the \emph{density} of \(E\)
at \(x\). The \emph{essential boundary}, the \emph{essential interior}, and
the \emph{essential exterior} of \(E\) are defined as
\[\begin{split}
\partial^e E&\coloneqq\big\{x\in\X\;\big|\;\overline D(E,x)>0,
\,\overline D(E^c,x)>0\big\},\\
E^{(1)}&\coloneqq\big\{x\in\X\;\big|\;D(E,x)=1\big\},\\
E^{(0)}&\coloneqq\big\{x\in\X\;\big|\;D(E,x)=0\big\},
\end{split}\]
respectively. It readily follows from the definitions that
\(\partial^e E\), \(E^{(1)}\), \(E^{(0)}\) are Borel sets and
\begin{equation}\label{eq:density+bdry}
\X=E^{(1)}\sqcup\partial^e E\sqcup E^{(0)}.
\end{equation}

Now suppose \(E\) is a set of finite perimeter. Then its perimeter
measure can be written as
\begin{equation}\label{eq:repr_formula_per}
P(E,\cdot)=\theta_E\mathcal H^{\rm cod\text{-}1}|_{\partial^e E},
\end{equation}
where \(\theta_E\colon\X\to(0,+\infty)\) is a Borel function,
while \(\mathcal H^{\rm cod\text{-}1}\) stands for the
\emph{codimension-one Hausdorff measure} on \((\X,\sfd,\mm)\),
namely the Borel regular outer measure on \(\X\) obtained via
Carath\'{e}odory construction starting from the gauge function
\(h(B_r(x))\coloneqq\mm(B_r(x))/(2r)\). Namely, we set
\[
\mathcal H^{\rm cod\text{-}1}(E)\coloneqq\sup_{\delta>0}
\inf\bigg\{\sum_{i=1}^\infty\frac{\mm(B_{r_i}(x_i))}{2r_i}\;\bigg|\;
(x_i)_i\subset\X,\,(r_i)_i\subset(0,\delta),\,E\subset\bigcup_{i=1}^\infty B_{r_i}(x_i)\bigg\}
\]
for every set \(E\subset\X\).

The representation formula
\eqref{eq:repr_formula_per} was proved in \cite[Theorem 5.3]{A02}.
\medskip

A PI space \((\X,\sfd,\mm)\) is said to be \emph{isotropic} provided
the density function \(\theta_E\) is `universal', in the following sense:
given two sets \(E,F\subset\X\) of finite perimeter, it holds that
\[
\theta_E(x)=\theta_F(x),\quad\text{ for }\mathcal H^{\rm cod\text{-}1}
\text{-a.e.\ }x\in\partial^e E\cap\partial^e F.
\]
The notion of isotropicity was introduced in
\cite[Definition 6.1]{AMP04}. We remark that all
\({\sf RCD}(K,N)\) spaces with \(N<\infty\), see \cref{sec:RCD} for the definition, are isotropic PI spaces;
cf.\ \cite[Example 1.31(iii)]{BPR20}.
\begin{lemma}\label{lem:per_inters_general}
Let \((\X,\sfd,\mm)\) be an isotropic PI space. Let \(E,F\subset\X\)
be sets of finite perimeter satisfying 
\(\mathcal H^{\rm cod\text{-}1}(\partial^e E\cap\partial^e F)=0\).
Then it holds that
\begin{equation}\label{eq:per_inters_general_claim}
P(E\cap F,\cdot)=P(E,\cdot)|_{F^{(1)}}+P(F,\cdot)|_{E^{(1)}}.
\end{equation}
\end{lemma}
\begin{proof}
First of all, we claim that
\begin{equation}\label{eq:per_inters_general_aux}
\partial^e(E\cap F)=(\partial^e E\cap F^{(1)})\sqcup
(\partial^e F\cap E^{(1)}),\quad\text{ up to }\mathcal H^{\rm cod\text{-}1}
\text{-negligible sets.}
\end{equation}
To prove the inclusion \(\supset\), fix \(x\in\partial^e E\cap F^{(1)}\).
On the one hand, \(\overline D\big((E\cap F)^c,x\big)\geq\overline D(E^c,x)>0\).
On the other hand, we may estimate
\[
\overline D(E\cap F,x)\geq\overline D(E,x)-\overline D(E\setminus F,x)
\geq\overline D(E,x)-D(F^c,x)=\overline D(E,x)>0.
\]
This shows that \(x\in\partial^e(E\cap F)\). Hence, we have proved that
\(\partial^e E\cap F^{(1)}\subset\partial^e(E\cap F)\) and similarly
\(\partial^e F\cap E^{(1)}\subset\partial^e(E\cap F)\), thus the inclusion
\(\supset\) in \eqref{eq:per_inters_general_aux} is achieved.
In order to get the converse inclusion \(\subset\) up to
\(\mathcal H^{\rm cod\text{-}1}\)-null sets, we can argue in the following
way. Recall that \(\partial^e(E\cap F)\subset\partial^e E\cup\partial^e F\),
see e.g.\ \cite[Proposition 1.16(ii)]{BPR20}. Therefore, we have that
\[\begin{split}
\partial^e(E\cap F)&\overset{\phantom{\eqref{eq:density+bdry}}}=
\big(\partial^e(E\cap F)\cap(\partial^e E\setminus\partial^e F)\big)\sqcup
\big(\partial^e(E\cap F)\cap(\partial^e F\setminus\partial^e E)\big)\\
&\overset{\eqref{eq:density+bdry}}=
\big(\partial^e(E\cap F)\cap\partial^e E\cap(F^{(1)}\cup F^{(0)})\big)\sqcup
\big(\partial^e(E\cap F)\cap\partial^e F\cap(E^{(1)}\cup E^{(0)})\big)\\
&\overset{\phantom{\eqref{eq:density+bdry}}}=
\partial^e(E\cap F)\cap\big((\partial^e E\cap F^{(1)})\sqcup
(\partial^e F\cap E^{(1)})\big),
\end{split}\]
where the identities have to be intended up to
\(\mathcal H^{\rm cod\text{-}1}\)-negligible
sets. The first identity follows from \(\mathcal H^{\rm cod\text{-}1}
(\partial^e E\cap\partial^e F)=0\), the last one from the fact that
\(\partial^e(E\cap F)\cap(E^{(0)}\cup F^{(0)})=\emptyset\); indeed,
if \(x\in\partial^e(E\cap F)\), then \(\overline D(E,x),\overline D(F,x)
\geq\overline D(E\cap F,x)>0\). The claim
\eqref{eq:per_inters_general_aux} follows.

To conclude, observe that by exploiting \eqref{eq:per_inters_general_aux}
and the isotropicity of \((\X,\sfd,\mm)\) we obtain
\[\begin{split}
P(E\cap F,\cdot)&=\theta_{E\cap F}\mathcal H^{\rm cod\text{-}1}
|_{\partial^e(E\cap F)}=\theta_E\mathcal H^{\rm cod\text{-}1}
|_{\partial^e E\cap F^{(1)}}+\theta_F\mathcal H^{\rm cod\text{-}1}
|_{\partial^e F\cap E^{(1)}}\\
&=P(E,\cdot)|_{F^{(1)}}+P(F,\cdot)|_{E^{(1)}},
\end{split}\]
which yields the sought conclusion.
\end{proof}

Observe that, as a byproduct of the proof of
\cref{lem:per_inters_general}, we also have that
\begin{equation}\label{eq:ineq_per_inters}
P(E,F^{(1)})+P(F,E^{(1)})\leq P(E\cap F),\quad\text{ whenever }
E,F\subset\X\text{ are of finite perimeter}.
\end{equation}
\begin{corollary}\label{cor:per_inters_ball}
Let \((\X,\sfd,\mm)\) be an isotropic PI space. Let \(E\subset\X\)
be a set of finite perimeter. Let \(\bar x\in\X\) be given. Then it
holds that
\[
P(E\cap B_r(\bar x),\cdot)=P(E,\cdot)|_{B_r(\bar x)}+
P(B_r(\bar x),\cdot)|_{E^{(1)}},\quad\text{ for a.e.\ }r>0.
\]
\end{corollary}
\begin{proof}
In view of \cref{lem:per_inters_general}, it is sufficient to show that
\(\mathcal H^{\rm cod\text{-}1}(\partial^e E\cap\partial^e B_r(\bar x))=0\)
holds for a.e.\ \(r>0\). Given that the topological boundaries
\(\{\partial B_r(\bar x)\}_{r>0}\) (thus a fortiori also the essential 
boundaries \(\{\partial^e B_r(\bar x)\}_{r>0}\)) are pairwise disjoint
and \(P(E,\cdot)\) is finite, we have that
\[
(\theta_E\mathcal H^{\rm cod\text{-}1})\big(\partial^e E\cap\partial^e
B_r(\bar x)\big)=P(E,\partial^e B_r(\bar x))=0,\quad\text{ for a.e.\ }r>0.
\]
Since \(\theta_E>0\), we can conclude that
\(\mathcal H^{\rm cod\text{-}1}(\partial^e E\cap\partial^e B_r(\bar x))=0\)
for a.e.\ \(r>0\), as desired.
\end{proof}

\begin{proposition}[Improved lower semicontinuity of the total variation]\label{prop:Lahti}
Let $(\X,\dist,\meas)$ be a PI space.
Let $f_i\in L^1_{\rm loc}(\X,\mm)$ be a sequence of functions converging to some function $f$ in $L^1_{\rm loc}$. If $f \in{\rm BV}_{\rm loc}(\X)$ and $|Df|(F^{(1)})<+\infty$, then
\[
\liminf_i |Df_i|(F^{(1)}) \ge |Df|(F^{(1)}).
\]
In particular, if $\{E_i\}_{i\in\N}$ is a sequence of $\meas$-measurable sets converging in $L^1_{\rm loc}$ to some set $E$ of locally finite perimeter and $P(E,F^{(1)})<+\infty$, then
\[
\liminf_i P(E_i,F^{(1)}) \ge P(E,F^{(1)}).
\]
\end{proposition}

\begin{proof}
We can assume without loss of generality that $F$ has a bounded representative. Indeed, once the claim is proved for bounded sets, applying the thesis on intersections $F\cap B_r(o)$ and letting $r\to+\infty$ eventually implies the full statement. 

So assume that $F$ is bounded, and then we can assume that the space is doubling and that a weak $(1,1)$-Poincar\'{e} inequality holds; in particular, we are in the setting of \cite{LahtiLeibniz, LahtiQuasiOpen}. Let $u\eqdef \nchi_F$ be the characteristic function of $F$. We define $u^\wedge$ pointwise at every $x \in\X$ as in \cite[p. 802]{LahtiLeibniz} by
\[
u^\wedge(x) \eqdef 
\sup \left\{t \in \R \st \lim_{r\to0} \frac{\meas(B_r(x) \cap \{u<t\})}{\meas(B_r(x))} = 0 \right\}
=
\begin{cases}
1 & x \in F^{(1)},\\
0 & x \in \X\setminus F^{(1)}
\end{cases}
= \nchi_{F^{(1)}}(x).
\]
Indeed
\[
 \frac{\meas(B_r(x) \cap \{u<t\})}{\meas(B_r(x))} =
 \begin{cases}
 1 & t>1,\\
  \frac{\meas(B_r(x) \cap (\X\setminus F))}{\meas(B_r(x))}  & t \in (0,1], \\
  0 & t\le0.
 \end{cases}
\]
Hence clearly $u^\wedge(x)=1$ (resp.\! $0$) if $x \in F^{(1)}$ (resp.\! $x \in F^{(0)}$). And if $x \in\X\setminus(F^{(1)}\cup F^{(0)}) = \partial^e F$, then $\overline{D}(\X\setminus F,x)>0$, thus for every $t \in(0,1]$ it holds
\[
\limsup_{r\to0} \frac{\meas(B_r(x) \cap \{u<t\})}{\meas(B_r(x))} >0,
\]
that implies $u^\wedge(x)=0$.
Then by \cite[Proposition 2.3]{LahtiLeibniz} we have that $u^\wedge$ is $1$-quasi lower semicontinuous in the sense of \cite[Definition 2.1]{LahtiLeibniz}, i.e., for any $\eps>0$ there is an open set $G$ such that ${\rm Cap}_1(G)<\eps$ and $u^\wedge|_{\X\setminus G}$ is lower semicontinuous.
It follows by the very definition \cite[Definition 2.1]{LahtiLeibniz} that $F^{(1)}=\{u^\wedge>0\}$ is $1$-quasi open, i.e., such that for any $\eps>0$ there is an open set $G$ with ${\rm Cap}_1(G)<\eps$ such that $F^{(1)}\cup G$ is open. Indeed, given $\eps>0$, by $1$-quasi lower semicontinuity of $u^\wedge$ we find an open set $G$ with ${\rm Cap}_1(G)<\eps$ and $u^\wedge|_{\X\setminus G}$ lower semicontinuous. Then $F^{(1)}\setminus G = \left\{u^\wedge|_{\X\setminus G} >0 \right\}$ is open in $\X\setminus G$, that is, there is an open set $U$ such that $F^{(1)}\setminus G = (\X\setminus G) \cap U$. Hence
\[
F^{(1)}\cup G = (F^{(1)}\setminus G) \cup G = ( (\X\setminus G) \cap U ) \cup G = U \cup G
\]
is open.
Therefore, as $|Df|(F^{(1)})<+\infty$ by assumption, we can apply \cite[Theorem 3.1]{LahtiLeibniz}, that is proved in \cite[Theorem 4.5]{LahtiQuasiOpen}, and we obtain
\[
\liminf_i |Df_i|(F^{(1)}) \ge |Df|(F^{(1)}).
\]
\end{proof}

\subsection{Differential calculus on \texorpdfstring{\(\sf RCD\)}{RCD} spaces}\label{sec:DifferentialCalculusRCD}
In this section we shall introduce basic tools about the first-order and second-order differential calculus on metric measure spaces. Then, we introduce the notion of $\RCD$ space and discuss some basic geometric properties.
\subsubsection{First-order calculus on metric measure spaces}
Let \((\X,\sfd,\mm)\) be a given metric measure space.
The space \(C([0,1],\X)\) of continuous curves in \(\X\)
is a complete and separable metric space if endowed with the supremum
distance \(\sfd_{\rm sup}(\gamma,\sigma)\coloneqq\max_{t\in[0,1]}\sfd(\gamma_t,\sigma_t)\). We recall the notion of test plan, introduced in
\cite{AmbrosioGigliSavare11,AmbrosioGigliSavare11_3}. A Borel
probability measure \(\ppi\) on \(C([0,1],\X)\) is called a
\emph{test plan} on \(\X\) provided:
\begin{itemize}
\item[\(\rm i)\)] \(\ppi\) has \emph{bounded compression},
i.e., there exists a constant \(C>0\) such that
\(({\rm e}_t)_*\ppi\leq C\mm\) for every \(t\in[0,1]\), where
\({\rm e}_t\colon C([0,1],\X)\to\X\) denotes the evaluation
map \({\rm e}_t(\gamma)\coloneqq\gamma_t\).
\item[\(\rm ii)\)] \(\ppi\) is concentrated on absolutely
continuous curves and has \emph{finite kinetic energy}, i.e.,
\[
\int\!\!\!\int_0^1|\dot\gamma_t|^2\,\d t\,\d\ppi(\gamma)<+\infty.
\]
\end{itemize}

Following \cite{AmbrosioGigliSavare11,AmbrosioGigliSavare11_3},
we can use the concept of test plan to introduce the notion of
Sobolev function. We declare that a Borel function \(f\colon\X\to\R\)
(considered up to \(\mm\)-a.e.\ equality) belongs to the
\emph{local Sobolev class} \({\rm S}^2_{\rm loc}(\X)\) provided
there exists \(G\in L^2_{\rm loc}(\X,\mm)\) such that
\[
\int\big|f(\gamma_1)-f(\gamma_0)\big|\,\d\ppi(\gamma)\leq
\int\!\!\!\int_0^1 G(\gamma_t)|\dot\gamma_t|\,\d t\,\d\ppi(\gamma),
\quad\text{ for every test plan }\ppi\text{ on }\X.
\]
The minimal such function \(G\) (where minimality is intended in
the \(\mm\)-a.e.\ sense) is denoted by \(|Df|\in L^2_{\rm loc}(\X,\mm)\)
and called the \emph{minimal weak upper gradient} of \(f\).
The \emph{Sobolev class} and the \emph{Sobolev space} over \(\X\)
are then defined as
\[
{\rm S}^2(\X)\coloneqq\big\{f\in{\rm S}^2_{\rm loc}(\X)\,:\,
|Df|\in L^2(\mm)\big\},\qquad W^{1,2}(\X)\coloneqq
L^2(\mm)\cap{\rm S}^2(\X),
\]
respectively. The Sobolev space \(W^{1,2}(\X)\) is a Banach space
if endowed with the norm
\[
\|f\|_{W^{1,2}(\X)}\coloneqq\Big(\|f\|_{L^2(\mm)}^2+
\big\||Df|\big\|_{L^2(\mm)}^2\Big)^{1/2},\quad
\text{ for every }f\in W^{1,2}(\X).
\]

We assume the reader is familiar
with the notion of \emph{\(L^2(\mm)\)-normed \(L^\infty(\mm)\)-module}
introduced in \cite{Gigli14}
(see also \cite{Gigli17} or \cite{GPbook} for an account of this topic).
If \(\mathscr M\) is an \(L^2(\mm)\)-normed \(L^\infty(\mm)\)-module,
then we denote by \(\mathscr M^*\) its \emph{dual module}. Given two
Hilbert \(L^2(\mm)\)-normed \(L^\infty(\mm)\)-modules \(\mathscr H\)
and \(\mathscr K\), we denote by \(\mathscr H\otimes\mathscr K\) their
\emph{tensor product}.
On a metric measure space \((\X,\sfd,\mm)\), a fundamental
\(L^2(\mm)\)-normed \(L^\infty(\mm)\)-module is the so-called
\emph{cotangent module} \(L^2(T^*\X)\), introduced in
\cite[Definition 2.2.1]{Gigli14} (see also
\cite[Theorem/Definition 2.8]{Gigli17}), which is characterised as follows.

The module \(L^2(T^*\X)\), together with the \emph{differential} operator
\(\d\colon{\rm S}^2(\X)\to L^2(T^*\X)\), are uniquely determined by these
conditions: \(\d\) is a linear map, \(|\d f|=|Df|\) holds
\(\mm\)-a.e.\ on \(\X\) for all \(f\in{\rm S}^2(\X)\), and \(L^2(T^*\X)\)
is generated by \(\big\{\d f\,:\,f\in{\rm S}^2(\X)\big\}\).
The differential enjoys several calculus rules:
\begin{itemize}
\item[\(\rm i)\)] \textsc{Closure.} Let \((f_n)_n\subset{\rm S}^2(\X)\)
be such that \(f_n\to f\) in the \(\mm\)-a.e.\ sense, for some Borel
function \(f\colon\X\to\R\). Suppose \(\d f_n\rightharpoonup\omega\)
weakly in \(L^2(T^*\X)\), for some \(\omega\in L^2(T^*\X)\). Then
\(f\in{\rm S}^2(\X)\) and \(\d f=\omega\).
\item[\(\rm ii)\)] \textsc{Locality.} Let \(f,g\in{\rm S}^2(\X)\) be
given. Then \(\nchi_{\{f=g\}}\cdot\d f=\nchi_{\{f=g\}}\cdot\d g\).
\item[\(\rm iii)\)] \textsc{Chain rule.} Let \(f\in{\rm S}^2(\X)\)
and \(\varphi\in C^1(\R)\cap{\rm Lip}(\R)\) be given. Then
\(\varphi\circ f\in{\rm S}^2(\X)\) and
\(\d(\varphi\circ f)=\varphi'\circ f\cdot\d f\).
\item[\(\rm iv)\)] \textsc{Leibniz rule.} Let \(f,g\in{\rm S}^2(\X)
\cap L^\infty(\mm)\). Then \(fg\in{\rm S}^2(\X)\) and
\(\d(fg)=f\cdot\d g+g\cdot\d f\).
\end{itemize}
We refer to \cite[Propositions 1.11 and 1.12]{Gigli17} for a proof
of these properties. Moreover, the \emph{tangent module} \(L^2(T\X)\)
is defined as the dual of the cotangent module, namely,
\(L^2(T\X)\coloneqq L^2(T^*\X)^*\); see \cite[Definition 2.18]{Gigli17}.
The elements of \(L^2(T^*\X)\) and \(L^2(T\X)\) are called
\emph{\(1\)-forms} and \emph{vector fields} on \(\X\), respectively.
For any \(p\in[1,\infty]\), we denote by \(L^p(T\X)\) the space
of all \emph{\(p\)-integrable vector fields}, namely, the completion of \(\big\{v\in L^2(T\X)\,:\,
|v|\in L^p(\mm)\big\}\) with respect to
\(\|v\|_{L^p(T\X)}\coloneqq\big\||v|\big\|_{L^p(\mm)}\).

The \emph{divergence} is defined as follows (cf.\ \cite[Definition 2.21]{Gigli17}): we define \(D({\rm div})\subset L^2(T\X)\) as the
space of all vector fields \(v\in L^2(T\X)\) for which there exists
(a uniquely determined) function \({\rm div}(v)\in L^2(\mm)\) such that
\[
\int\d f(v)\,\d\mm=-\int f\,{\rm div}(v)\,\d\mm,
\quad\text{ for every }f\in W^{1,2}(\X).
\]
As shown in \cite[eq.\ (2.3.13)]{Gigli14}, the divergence operator
satisfies the following \emph{Leibniz rule}: if \(v\in D({\rm div})\)
and \(f\in{\rm Lip}_{bs}(\X)\), then \(f\cdot v\in D({\rm div})\) and
\begin{equation}\label{eq:Leibniz_div}
{\rm div}(f\cdot v)=\d f(v)+f\,{\rm div}(v).
\end{equation}
Also, \(D({\rm div})\) is a vector subspace of \(L^2(T\X)\) and
\(D({\rm div})\ni v\mapsto{\rm div}(v)\in L^2(\mm)\) is linear.
\subsubsection{Infinitesimal Hilbertianity and Laplacian}
According to \cite{Gigli12}, we say that a metric measure space \((\X,\sfd,\mm)\) is
\emph{infinitesimally Hilbertian} provided \(W^{1,2}(\X)\) is a Hilbert
space. Under this assumption, it holds that \({\rm Lip}_{bs}(\X)\) is
dense in \(W^{1,2}(\X)\), cf.\ \cite{AmbrosioGigliSavare11_3}.
It holds (cf.\ \cite[Proposition 2.3.17]{Gigli14}) that a metric measure
space \((\X,\sfd,\mm)\) is infinitesimally Hilbertian if and only if
\(L^2(T^*\X)\) and \(L^2(T\X)\) are Hilbert \(L^2(\mm)\)-normed
\(L^\infty(\mm)\)-modules. In this case, there is a canonical notion of
\emph{gradient} operator \(\nabla\colon{\rm S}^2(\X)\to L^2(T\X)\),
defined as \(\nabla\coloneqq{\sf R}\circ\d\), where
\({\sf R}\colon L^2(T^*\X)\to L^2(T\X)\)
stands for the Riesz isomorphism. Consequently, \(\nabla\) naturally
inherits its calculus rules from \(\d\).

\medskip

Moreover, one can define the \emph{Laplacian}
as follows: we define \(D(\Delta)\subset W^{1,2}(\X)\) as the
space of all functions \(f\in W^{1,2}(\X)\) for which there exists
(a uniquely determined) \(\Delta f\in L^2(\mm)\) such that
\[
\int\langle\nabla f,\nabla g\rangle\,\d\mm=-\int g\,\Delta f\,\d\mm,
\quad\text{ for every }g\in W^{1,2}(\X),
\]
where we set \(\langle\nabla f,\nabla g\rangle\coloneqq\frac{1}{2}
\big(|D(f+g)|^2-|Df|^2-|Dg|^2\big)\);
cf.\ \cite[Definition 2.42]{Gigli17}. For any given \(f\in W^{1,2}(\X)\),
it clearly holds that \(f\in D(\Delta)\) if and only if
\(\nabla f\in D({\rm div})\), and in this case \(\Delta f={\rm div}(\nabla f)\).
We will also consider a more general notion of Laplacian:
\begin{definition}[Measure-valued Laplacian {\cite[Definition 4.4]{Gigli12}}]
Let \((\X,\sfd,\mm)\) be infinitesimally Hilbertian. Then we define
\(D({\boldsymbol\Delta})\subset{\rm S}^2_{\rm loc}(\X)\) as the space of all
\(f\in{\rm S}^2_{\rm loc}(\X)\) for which there exists (a uniquely determined)
boundedly-finite, signed Radon measure \({\boldsymbol\Delta}f\)
on \(\X\) such that
\[
\int\langle\nabla f,\nabla g\rangle\,\d\mm=-\int g\,\d{\boldsymbol\Delta}f,
\quad\text{ for every }g\in{\rm Lip}_{bs}(\X).
\]
\end{definition}

The two notions of Laplacian are consistent, in the following sense:
a given function \(f\in W^{1,2}(\X)\) belongs to \(D(\Delta)\) if and
only if it belongs to \(D({\boldsymbol\Delta})\) and it satisfies
\({\boldsymbol\Delta}f\ll\mm\) with \(\frac{\d{\boldsymbol\Delta}f}{\d\mm}
\in L^2(\mm)\). In this case, it holds that \({\boldsymbol\Delta}f=
\Delta f\,\mm\).
\begin{proposition}[Chain rule for \(\boldsymbol\Delta\)
{\cite[Proposition 4.28]{Gigli12}}]\label{prop:chain_rule_mv_Lapl}
Let \((\X,\sfd,\mm)\) be infinitesimally Hilbertian. Fix any
\(f\in D({\boldsymbol\Delta})\) locally Lipschitz and \(\varphi\in C^2(\R)\).
Then \(\varphi\circ f\in D({\boldsymbol\Delta})\) and
\[
{\boldsymbol\Delta}(\varphi\circ f)=\varphi'\circ f\,{\boldsymbol\Delta}f
+\varphi''\circ f\,|Df|^2\,\mm.
\]
\end{proposition}

For any \(t>0\) and \(p\in[1,\infty]\), we denote by
\({\sf h}_t\colon L^p(\mm)\to L^p(\mm)\) the \emph{heat flow} at time \(t\).
The operator \({\sf h}_t\colon L^p(\mm)\to L^p(\mm)\) is linear and continuous.
Recall that \(\|{\sf h}_t f\|_{L^p(\mm)}\leq\|f\|_{L^p(\mm)}\) for every
\(f\in L^p(\mm)\) and that \({\sf h}_t\) is \emph{mass-preserving},
meaning that
\begin{equation}\label{eq:h_t_mass-pres}
\int{\sf h}_t f\,\d\mm=\int f\,\d\mm,\quad
\text{ for every }f\in L^1(\mm)\text{ and }t>0.
\end{equation}
More generally, the heat flow is \emph{self-adjoint}, namely if
\(p,q\in[1,\infty]\) satisfy \(\frac{1}{p}+\frac{1}{q}=1\), then
\begin{equation}\label{eq:h_t_self-adj}
\int g\,{\sf h}_t f\,\d\mm=\int f\,{\sf h}_t g\,\d\mm,
\quad\text{ for every }f\in L^p(\mm)\text{ and }g\in L^q(\mm).
\end{equation}
Another important property of the heat flow is the \emph{weak maximum
principle}, which states that if \(f\in L^p(\mm)\) satisfies \(f\leq C\)
in the \(\mm\)-a.e.\ sense for some constant \(C\in\R\), then
\({\sf h}_t f\leq C\) holds \(\mm\)-a.e.\ for every \(t>0\).
\subsubsection{Second-order calculus on \(\sf RCD\) spaces}\label{sec:RCD}
With the terminology we discussed so far at our disposal, we can introduce the definition of the so-called $\RCD$ condition for m.m.s., and discuss some basic and useful properties of it.
For more on the topic, we refer the interested reader to the survey of Ambrosio \cite{AmbrosioSurvey} and the references therein. 

After the introduction, in the independent works \cite{Sturm1,Sturm2} and \cite{LottVillani}, of the curvature dimension condition $\CD(K,N)$ encoding in a synthetic way the notion of Ricci curvature bounded from below by $K$ and dimension bounded above by $N$, the definition of $\RCD(K,N)$ m.m.s.\ was first proposed in \cite{Gigli12} and then studied in \cite{Gigli13, ErbarKuwadaSturm15,AmbrosioMondinoSavare15}, see also \cite{CavallettiMilman16} for the equivalence between the $\RCD^*(K,N)$ and the $\RCD(K,N)$ condition in the case the reference measure is finite. The infinite dimensional counterpart of this notion had been previously investigated in \cite{AmbrosioGigliSavare14}, see also \cite{AmbrosioGigliMondinoRajala15} for the case of $\sigma$-finite reference measures. We shall recall here for the sake of brevity only the definition of the $\RCD$ condition.

For the definition of the weaker $\CD$ condition, which is given by means of convexity properties of appropriate entropies along Wasserstein geodesics, we refer the reader to the original works \cite{LottVillani, Sturm1, Sturm2}, to the book \cite{VillaniBook}, and to the survey \cite{AmbrosioSurvey} and the references therein. 

\begin{definition}[\(\sf RCD\) space]
Let \((\X,\sfd,\mm)\) be a metric measure space. Then
\((\X,\sfd,\mm)\) is a \emph{\({\sf RCD}(K,N)\) space}, for some
\(K\in\R\) and \(N\in[1,\infty]\), provided the following conditions hold:
\begin{itemize}
\item[\(\rm i)\)] There exist \(C>0\) and \(\bar x\in\X\) such that
\(\mm(B_r(\bar x))\leq e^{Cr^2}\) for every \(r>0\).
\item[\(\rm ii)\)] \textsc{Sobolev-to-Lipschitz.} If \(f\in W^{1,2}(\X)\)
satisfies \(|Df|\in L^\infty(\mm)\), then \(f\) admits a Lipschitz
representative \(\bar f\colon\X\to\R\) such that
\({\rm Lip}(\bar f)=\big\||Df|\big\|_{L^\infty(\mm)}\).
\item[\(\rm iii)\)] \((\X,\sfd,\mm)\) is infinitesimally Hilbertian.
\item[\(\rm iv)\)] \textsc{Bochner inequality.} It holds that
\[
\frac{1}{2}\int|Df|^2\Delta g\,\d\mm\geq
\int g\bigg(\frac{(\Delta f)^2}{N}+
\langle\nabla f,\nabla\Delta f\rangle+K|Df|^2\bigg)\,\d\mm,
\]
for every \(f\in D(\Delta)\) with \(\Delta f\in W^{1,2}(\X)\)
and \(g\in D(\Delta)\cap L^\infty(\mm)^+\) with
\(\Delta g\in L^\infty(\mm)\), where we adopt the convention
that \(\frac{(\Delta f)^2}{N}\equiv 0\) when \(N=\infty\).
\end{itemize}
\end{definition}

On \(\RCD(K,\infty)\) spaces, the following \emph{Bakry--\'{E}mery estimate}
is verified: it \(\mm\)-a.e.\ holds that
\begin{equation}\label{eq:BE_fcts}
|D{\sf h}_t f|^2\leq e^{-2Kt}\,{\sf h}_t(|Df|^2),
\quad\text{ for every }f\in W^{1,2}(\X)\text{ and }t>0.
\end{equation}

We shall recall the Bishop--Gromov comparison theorem for spaces satisfying the Curvature-Dimension condition.

\begin{remark}[Models of constant sectional curvature and Bishop--Gromov comparison Theorem]\label{rem:PerimeterMMS2}
{\rm
We need to introduce some notation about the so-called radial simply connected models of constant sectional curvature (cf.
\cite[Example 1.4.6]{Petersen2016}). Let us define
\[
\sn_K(r) := \begin{cases}
(-K)^{-\frac12} \sinh((-K)^{\frac12} r) & K<0,\\
r & K=0,\\
K^{-\frac12} \sin(K^{\frac12} r) & K>0.
\end{cases}
\]

If $K>0$, then $((0,\pi/\sqrt{K}]\times \mathbb S^{N-1},\d r^2+\mathrm{sn}_K^2(r)g_1)$, where $g_1$ is the canonical metric on $\mathbb S^{N-1}$, is the radial model of dimension $N$ and constant sectional curvature $K$. The metric can be smoothly extended at $r=0$, and thus we shall write that the metric is defined on the ball $ B_{\pi/\sqrt{K}} \subset \R^N$. The Riemannian manifold $(B_{\pi/\sqrt{K}}, g_K\eqdef \d r^2+\mathrm{sn}_K^2(r)g_1)$ is the unique (up to isometry) simply connected Riemannian manifold of dimension $N$ and constant sectional curvature $K>0$.

If instead $K\leq 0$, then $((0,+\infty)\times\mathbb S^{N-1},\d r^2+\mathrm{sn}_K^2(r)g_1)$ is the radial model of dimension $N$ and constant sectional curvature $K$. Extending the metric at $r=0$ analogously yields the unique (up to isometry) simply connected Riemannian manifold of dimension $N$ and constant sectional curvature $K\leq 0$, in this case denoted by $(\R^N,g_K)$.

We denote by $v(N,K,r)$ the volume of the ball of radius $r$ in the (unique) simply connected Riemannian manifold of sectional curvature $K$ and dimension $N$, and by $s(N, K, r)$ the volume of the boundary of such a ball. In particular $s(N,K,r)=N\omega_N\mathrm{sn}_K^{N-1}(r)$ and $v(N,K,r)=\int_0^rN\omega_N\mathrm{sn}_K^{N-1}(r)\de r$, where $\omega_N$ is the Euclidean volume of the Euclidean unit ball in $\mathbb R^N$.

Let us now recall the celebrated Bishop--Gromov comparison theorem.
For an arbitrary $\CD((N-1)K,N)$ space $(\X,\dist,\meas)$ the classical Bishop--Gromov volume comparison holds. More precisely, for a fixed $x\in\X$, the function $\meas(B_r(x))/v(N,K,r)$ is nonincreasing in $r$ and the function $P(B_r(x))/s(N,K,r)$ is essentially nonincreasing in $r$, i.e., the inequality
\[
P(B_R(x))/s(N,K,R) \le P(B_r(x))/s(N,K,r),
\]
holds for almost every radii $R\ge r$, see \cite[Theorem 18.8, Equation (18.8), Proof of Theorem 30.11]{VillaniBook}. Moreover, it holds that 
\[
P(B_r(x))/s(N, K, r)\leq \vol(B_r(x))/v(N,K,r),
\]
for any $r>0$, indeed the last inequality follows from the monotonicity of the volume and perimeter ratios together with the coarea formula on balls.

In addition, if $(\X,\dist,\mathcal{H}^N)$ is an $\RCD((N-1)K,N)$ space, one can conclude that $\mathcal{H}^N$-almost every point has a unique measured Gromov--Hausdorff tangent isometric to $\mathbb R^N$ (\cite[Theorem 1.12]{DePhilippisGigli18}), and thus, from the volume convergence in \cite{DePhilippisGigli18}, we get
 \begin{equation}\label{eqn:VolumeConv}
 \lim_{r\to 0}\frac{\mathcal{H}^N(B_r(x))}{v(N,K,r)}=\lim_{r\to 0}\frac{\mathcal{H}^N(B_r(x))}{\omega_Nr^N}=1, \qquad \text{for $\mathcal{H}^N$-almost every $x$},
 \end{equation}
 where $\omega_N$ is the volume of the unit ball in $\mathbb R^N$. Moreover, since the density function $x\mapsto \lim_{r\to 0}\mathcal{H}^N(B_r(x))/\omega_Nr^N$ is lower semicontinuous (\cite[Lemma 2.2]{DePhilippisGigli18}), it is bounded above by the constant $1$. Hence, from the monotonicity at the beginning of the remark we deduce that, if $(\X,\dist,\mathcal{H}^N)$ is an $\RCD((N-1)K,N)$ space, then for every $x\in \X$ we have $\mathcal{H}^N(B_r(x))\leq v(N,K,r)$ for every $r>0$.
 
 Let us now recall the definition of density of a point in an $\RCD$ space. Let $(X,\dist,\meas)$ be an $\RCD((N-1)K,N)$ space with $K\in\mathbb R$ and $N<+\infty$. By the previous Bishop--Gromov monotonicity result we have that $\meas(B_r(x))/v(N,K,r)$ is non increasing in $r$ for every $x\in \X$, and hence the following limit exists
 \begin{equation}\label{eqn:Density}
     \vartheta[\X,\dist,\meas](x):=\lim_{r\to 0}\frac{\meas(B_r(x))}{v(N,K,r)}=\lim_{r\to 0}\frac{\meas(B_r(x))}{\omega_Nr^N} \in (0,+\infty],
 \end{equation}
 and will be called the {\em density of the point $x$ in $(\X,\dist,\meas)$}.\fr}
\end{remark}
\begin{remark}[$\CD(K,N)$ spaces with $N<+\infty$ are PI spaces]
{\rm
Given $K\in\mathbb R$ and $N<+\infty$, a $\CD(K,N)$ space,
and thus also an $\RCD(K,N)$ space, is a PI space according to \cref{def:PI}. 

Indeed on a $\CD(K,N)$ space a weak local $(1,1)$-Poincar\'{e} inequality holds, see \cite{Rajala12}, and the uniform locally doubling property holds as a direct consequence of the Bishop--Gromov comparison theorem, cf.\ \cref{rem:PerimeterMMS2}.\fr}\end{remark}

We shall now discuss some features of the second-order differential calculus on $\RCD$ spaces. 
Let \((\X,\sfd,\mm)\) be an \({\sf RCD}(K,\infty)\) space.
Following the axiomatization in \cite[Definition 6.1.7]{GPbook},
the class of \emph{test functions} on \(\X\) is defined as
\[
{\rm Test}^\infty(\X)\coloneqq\Big\{f\in{\rm Lip}(\X)\cap L^\infty(\mm)
\cap D(\Delta)\;\Big|\;\Delta f\in W^{1,2}(\X)\cap L^\infty(\mm)\Big\}.
\]
Thanks to the results in \cite{Savare13,Gigli14}, the family
\({\rm Test}^\infty(\X)\) is strongly dense in \(W^{1,2}(\X)\). As in
\cite[eq.\ (6.60)]{GPbook}, the class of \emph{test vector fields}
on \(\X\) is defined as
\[
{\rm TestVF}(\X)\coloneqq\bigg\{\sum_{i=1}^n g_i\nabla f_i\;\bigg|\;
n\in\N,\,(f_i)_{i=1}^n,(g_i)_{i=1}^n\subset{\rm Test}^\infty(\X)\bigg\}
\subset L^2(T\X).
\]
\begin{remark}\label{rmk:TestVF_div}{\rm
Observe that \({\rm TestVF}(\X)\subset D({\rm div})\) and
\[
{\rm div}(v)\in L^\infty(\mm),\quad\text{ for every }v\in{\rm TestVF}(\X).
\]
Indeed, if \(v=\sum_{i=1}^n g_i\nabla f_i\), then we can readily
deduce from \eqref{eq:Leibniz_div} that \(v\in D({\rm div})\) and
\[
{\rm div}(v)=\sum_{i=1}^n\big(\d g_i(\nabla f_i)+g_i\,{\rm div}(\nabla f_i)
\big)=\sum_{i=1}^n\big(\d g_i(\nabla f_i)+g_i\,\Delta f_i\big),
\]
whence it also follows that \({\rm div}(v)\in L^\infty(\mm)\),
yielding the sought conclusion.
\fr}\end{remark}
%
%


Let \((\X,\sfd,\mm)\) be an \({\sf RCD}(K,\infty)\) space.
Given any \(f\in{\rm Test}^\infty(\X)\), we denote by
\({\rm Hess}(f)\) its \emph{Hessian}, namely, the unique element of
the tensor product \(L^2(T^*\X)\otimes L^2(T^*\X)\) satisfying
\[\begin{split}
-2\int h\,{\rm Hess}(f)&(\nabla g_1\otimes\nabla g_2)\,\d\mm\\
=\,&\int\langle\nabla f,\nabla g_1\rangle{\rm div}(h\nabla g_2)
+\langle\nabla f,\nabla g_2\rangle{\rm div}(h\nabla g_1)
+h\big\langle\nabla f,\nabla\langle\nabla g_1,\nabla g_2\rangle
\big\rangle\,\d\mm,
\end{split}\]
for every choice of \(g_1,g_2,h\in{\rm Test}^\infty(\X)\); see
\cite[Definition 3.5]{Gigli17}. Following \cite[Definition 3.15]{Gigli17},
we define the space \(W^{1,2}_C(T\X)\) of \emph{Sobolev vector fields}
on \(\X\) as follows: given any \(v\in L^2(T\X)\), we declare that
\(v\in W^{1,2}_C(T\X)\) if there exists a tensor
\(\nabla v\in L^2(T\X)\otimes L^2(T\X)\) such that
\[
\int h\nabla v:(\nabla f\otimes\nabla g)\,\d\mm=
-\int\langle v,\nabla g\rangle{\rm div}(h\nabla f)
+h\,{\rm Hess}(g)(v\otimes\nabla f)\,\d\mm,
\]
for every choice of \(f,g\in{\rm Test}^\infty(\X)\) and
\(h\in{\rm Lip}_b(\X)\). The element \(\nabla v\) is called
the \emph{covariant derivative} of \(v\). It holds that
\({\rm TestVF}(\X)\subset W^{1,2}_C(T\X)\) and
\[
\nabla\bigg(\sum_{i=1}^n g_i\nabla f_i\bigg)=
\sum_{i=1}^n\nabla g_i\otimes\nabla f_i+
g_i\,{\rm Hess}(f_i)^\sharp,\quad\text{ for every }
\sum_{i=1}^n g_i\nabla f_i\in{\rm TestVF}(\X),
\]
where \(L^2(T^*\X)\otimes L^2(T^*\X)\ni A\mapsto A^\sharp
\in L^2(T\X)\otimes L^2(T\X)\) stands for the Riesz (musical)
isomorphism. The space \(H^{1,2}_C(T\X)\) is then defined as
the closure of \({\rm TestVF}(\X)\) in \(W^{1,2}_C(T\X)\).

\medskip

We denote by
\(\{{\sf h}_{{\rm H},t}\}_{t\geq 0}\) the gradient flow in \(L^2(T\X)\)
of the \emph{augmented Hodge energy functional}, defined as in
\cite[eq.\ (3.5.16)]{Gigli14} (see also the discussion at
the beginning of \cite[Section 1.4]{BPS19}).
Given that \(H^{1,2}_{\rm H}(T\X)\subset H^{1,2}_C(T\X)\)
by \cite[Corollary 3.6.4]{Gigli14}, we have that
\({\sf h}_{{\rm H},t}(v)\in H^{1,2}_C(T\X)\) for every \(v\in L^2(T\X)\)
and \(t>0\). The consistency of \({\sf h}_{{\rm H},t}\) with the heat flow
for functions is evident from the fact that
\begin{equation}\label{eq:consist_with_heat_flow}
{\sf h}_{{\rm H},t}(\nabla f)=\nabla{\sf h}_t f,
\quad\text{ for every }f\in W^{1,2}(\X)\text{ and }t>0.
\end{equation}
As proved in \cite[Lemma 1.37]{BPS19}, it holds
that \({\sf h}_{{\rm H},t}\) is \emph{self-adjoint}, meaning that
\begin{equation}\label{eq:h_Ht_self-adj}
\int\langle{\sf h}_{{\rm H},t}(v),w\rangle\,\d\mm=
\int\langle v,{\sf h}_{{\rm H},t}(w)\rangle\,\d\mm,
\quad\text{ for every }v,w\in L^2(T\X)\text{ and }t>0.
\end{equation}
Moreover, it was shown in \cite[Proposition 3.6.10]{Gigli14}
that the \emph{Bakry--\'{E}mery estimate} holds:
\begin{equation}\label{eq:BE_vf}
|{\sf h}_{{\rm H},t}(v)|^2\leq e^{-2Kt}\,{\sf h}_t(|v|^2)\;\;\;
\mm\text{-a.e.},\quad\text{ for every }v\in L^2(T\X)\text{ and }t>0.
\end{equation}
Following \cite{Gigli12}, for \(K\in\R\) and \(N\in(1,\infty)\)
we define the function \(\tilde\tau_{K,N}\colon[0,+\infty)\to\R\) as
\[
\tilde\tau_{K,N}(\theta)\coloneqq\left\{\begin{array}{ll}
\frac{1}{N}\Big(1+\theta\sqrt{K(N-1)}\,{\rm cotan}\Big(\theta
\sqrt{\frac{K}{N-1}}\,\Big)\Big),\\
1,\\
\frac{1}{N}\Big(1+\theta\sqrt{-K(N-1)}\,{\rm cotanh}\Big(\theta
\sqrt{\frac{-K}{N-1}}\,\Big)\Big),
\end{array}\quad\begin{array}{ll}
\text{ if }K>0,\\
\text{ if }K=0,\\
\text{ if }K<0.
\end{array}\right.
\]
Observe that if \(K\leq 0\), then \(\tilde\tau_{K,N}\) is bounded on
all bounded subsets of \([0,+\infty)\).
\begin{theorem}[Laplacian comparison {\cite[Corollary 5.15]{Gigli12}}]
\label{thm:Lapl_comp_dist}
Let \((\X,\sfd,\mm)\) be an \({\sf RCD}(K,N)\) space with \(N<\infty\).
Let \(\bar x\in\X\) be given. Then it holds that \(\sfd_{\bar x}^2\in
D({\boldsymbol\Delta})\) and
\[
{\boldsymbol\Delta}\sfd_{\bar x}^2\leq
2N\,\tilde\tau_{K,N}\circ\sfd_{\bar x}\,\mm.
\]
\end{theorem}

On \({\sf RCD}(K,\infty)\) spaces, the \emph{dual heat flow}
\(\{\mathscr H_t\}_{t\geq 0}\) can be defined as follows: let \(\mu\geq 0\)
be a finite Borel measure on \(\X\) with \emph{finite second-moment},
meaning that \(\int\sfd^2_{\bar x}\,\d\mu<+\infty\) for some (thus any)
point \(\bar x\in\X\). Then for any \(t\geq 0\) one can define
\(\mathscr H_t\mu\) as the unique Borel measure on \(\X\) such that
\begin{equation}\label{eq:dual_h_t}
\int f\,\d\mathscr H_t\mu=\int{\sf h}_t f\,\d\mu,
\quad\text{ for every }f\in{\rm Lip}_b(\X)\cap W^{1,2}(\X).
\end{equation}
The above definition is well-posed, as the \(\sf RCD\) condition
ensures that \({\sf h}_t f\in{\rm Lip}_b(\X)\cap W^{1,2}(\X)\)
whenever \(f\in{\rm Lip}_b(\X)\cap W^{1,2}(\X)\). Indeed, this can be
proved by combining the weak maximum principle with the Bakry--\'{E}mery estimate \eqref{eq:BE_fcts} and the Sobolev-to-Lipschitz property.

It also holds that \(\mathscr H_t\mu\) has finite second-moment,
\(\mathscr H_t\mu\ll\mm\), and \(\mathscr H_t\mu(\X)=\mu(\X)\)
for every \(t>0\), thus in particular each measure \(\mathscr H_t\mu\)
is finite. Given any finite signed Borel measure \(\mu\) on \(\X\) such that
\(\mu^+\) and \(\mu^-\) have finite second-moment, we define
\[
{\sf h}_t^*\mu\coloneqq\frac{\d\mathscr H_t\mu^+}{\d\mm}-
\frac{\d\mathscr H_t\mu^-}{\d\mm}\in L^1(\mm),\quad\text{ for every }t>0.
\]
Observe that \(\{{\sf h}_t^*\mu\}_{t>0}\) is bounded in \(L^1(\mm)\),
more precisely \(\|{\sf h}_t^*\mu\|_{L^1(\mm)}\leq|\mu|(\X)\) for
all \(t>0\).
\begin{remark}[Semigroup property of \(\mathscr H_t\)]
\label{rmk:semigr_H_t}{\rm
Given that the heat flow satisfies the semigroup property
\({\sf h}_{t+s}={\sf h}_t\circ{\sf h}_s\) for every \(t,s\geq 0\),
we deduce that \(\mathscr H_{t+s}=\mathscr H_t\circ\mathscr H_s\) and
\[
{\sf h}_{t+s}^*\mu={\sf h}_t^*(\mathscr H_s\mu^+-\mathscr H_s\mu^-),
\quad\text{ for every }t,s>0.
\]
The proof of this fact is a direct consequence of the definitions
of \(\mathscr H_t\) and \({\sf h}_t^*\).
\fr}\end{remark}
\begin{lemma}\label{lem:h_star_meas}
Let \((\X,\sfd,\mm)\) be a \({\sf RCD}(K,\infty)\) space. Let \(\mu\)
be a finite signed Borel measure on \(\X\) with \(\mu^+\), \(\mu^-\)
having finite second-moment. Then the curve
\((0,+\infty)\ni t\mapsto{\sf h}_t^*\mu\in L^1(\mm)\) is strongly measurable.
\end{lemma}
\begin{proof}
As a preliminary step, we prove the claim under the additional assumption
that \(\mu\ll\mm\). Being \(L^1(\mm)\) separable, it suffices to check the
weak measurability of \(t\mapsto{\sf h}_t^*\mu\in L^1(\mm)\).
To this aim, fix any \(f\in L^\infty(\mm)\). Choose sequences
\((f_n)_n\subset{\rm Lip}_{bs}(\X)\) and \((\eta_k)_k\subset L^1(\mm)\cap L^2(\mm)\)
such that \(\sup_n\|f_n\|_{L^\infty(\mm)}<+\infty\),
\(|\eta_k|\leq\big|\frac{\d\mu}{\d\mm}\big|\), and \(f_n\to f\),
\(\eta_k\to\frac{\d\mu}{\d\mm}\) in the \(\mm\)-a.e.\ sense. By
applying the dominated convergence theorem, we thus obtain that
\[\begin{split}
\int f\,{\sf h}_t^*\mu\,\d\mm&=
\lim_{n\to\infty}\int f_n\,{\sf h}_t^*\mu\,\d\mm
=\lim_{n\to\infty}\bigg(\int f_n\,\d\mathscr H_t\mu^+ -
\int f_n\,\d\mathscr H_t\mu^-\bigg)=\lim_{n\to\infty}\int{\sf h}_t f_n\,\d\mu\\
&=\lim_{n\to\infty}\int\frac{\d\mu}{\d\mm}\,{\sf h}_t f_n\,\d\mm
=\lim_{n\to\infty}\lim_{k\to\infty}\int\eta_k\,{\sf h}_t f_n\,\d\mm,
\quad\text{ for every }t>0.
\end{split}\]
Given any \(n,k\in\N\), we know that the curve
\([0,+\infty)\ni t\mapsto{\sf h}_t f_n\in L^2(\mm)\) is continuous,
thus accordingly \([0,+\infty)\ni t\mapsto\int\eta_k\,{\sf h}_t f_n\,\d\mm\)
is continuous as well. By combining this fact with the above representation
of the function \((0,+\infty)\ni t\mapsto\int f\,{\sf h}_t^*\mu\,\d\mm\),
we conclude that it is Borel measurable. By the arbitrariness of
\(f\in L^\infty(\mm)\), we have shown the claim when \(\mu\ll\mm\).

Let us now drop the absolute continuity assumption on \(\mu\).
Fix any \(\varepsilon>0\). Then it holds that
\(\mathscr H_\varepsilon\mu^+-\mathscr H_\varepsilon\mu^-\ll\mm\),
thus the first part of the proof and \cref{rmk:semigr_H_t} ensure that
\[
(\varepsilon,+\infty)\ni t\mapsto{\sf h}_t^*\mu={\sf h}_{t-\varepsilon}^*
(\mathscr H_\varepsilon\mu^+ -\mathscr H_\varepsilon\mu^-)\in L^1(\mm),
\quad\text{ is strongly measurable.}
\]
Being \(\varepsilon>0\) arbitrarily chosen, the statement is achieved.
\end{proof}
In the sequel, we will need the following calculus rule, which we did
not find in the literature.
\begin{lemma}[Leibniz rule for Sobolev vector fields]\label{lem:mult_H12_C}
Let \((\X,\sfd,\mm)\) be an \({\sf RCD}(K,\infty)\) space. Fix any
\(v\in H^{1,2}_C(T\X)\cap L^\infty(T\X)\) and \(f\in{\rm Lip}_{bs}(\X)\).
Then \(f\cdot v\in H^{1,2}_C(T\X)\cap L^\infty(T\X)\) and
\begin{equation}\label{eq:Leibniz_Sob_vf}
\nabla(f\cdot v)=\nabla f\otimes v+f\,\nabla v.
\end{equation}
\end{lemma}
\begin{proof}
We know from \cite[Proposition 6.3.3]{GPbook} that
\(f\cdot v\in W^{1,2}_C(T\X)\cap L^\infty(T\X)\) and \eqref{eq:Leibniz_Sob_vf}
holds. Pick a sequence \((v_n)_n\subset{\rm TestVF}(\X)\) such that \(v_n\to v\)
strongly in \(W^{1,2}_C(T\X)\). Moreover, we can find
\((f_n)_n\subset{\rm Test}^\infty(\X)\) such that \(f_n\to f\)
and \(|D(f_n-f)|\to 0\) in the \(\mm\)-a.e.\ sense,
\(\sup_n\|f_n\|_{L^\infty(\mm)}\leq\|f\|_{L^\infty(\mm)}\),
and \(\sup_n{\rm Lip}(f_n)\leq{\rm Lip}(f)\); see the proof of
\cite[Proposition 6.1.8]{GPbook}. From the fact that
\({\rm Test}^\infty(\X)\) is an algebra
(cf.\ \cite[Theorem 6.1.11]{GPbook}), we infer
that \((f_n\cdot v_n)_n\subset{\rm TestVF}(\X)\). Since
\(|f_n\cdot v_n-f\cdot v|\leq\|f\|_{L^\infty(\mm)}|v_n-v|+|f_n-f||v|\) and
\[\begin{split}
&\big|\nabla(f_n\cdot v_n)-\nabla(f\cdot v)\big|\\
\overset{\eqref{eq:Leibniz_Sob_vf}}\leq\,&
|\nabla f_n||v_n-v|+\big|\nabla(f_n-f)\big||v|
+|f_n|\big|\nabla(v_n-v)\big|_{\sf HS}+|f_n-f||\nabla v|_{\sf HS}\\
\overset{\phantom{\eqref{eq:Leibniz_Sob_vf}}}\leq\,&
{\rm Lip}(f)|v_n-v|+\big|\nabla(f_n-f)\big||v|+\|f\|_{L^\infty(\mm)}
\big|\nabla(v_n-v)\big|_{\sf HS}+|f_n-f||\nabla v|_{\sf HS}
\end{split}\]
hold \(\mm\)-a.e., by using the dominated convergence theorem
we can conclude that \(f_n\cdot v_n\to f\cdot v\) strongly in
\(W^{1,2}_C(T\X)\) and thus \(f\cdot v\in H^{1,2}_C(T\X)\).
All in all, the statement is achieved.
\end{proof}
\subsection{Deformation Lemma in \texorpdfstring{\(\sf RCD\)}{RCD} spaces}\label{sec:Deformation2}
In this section we shall prove the Deformation Lemma in $\RCD$ spaces. We first need to introduce and study the $(p)$-divergence of vector fields, and the mollified heat flow for measures and vector fields.
\subsubsection{On the \texorpdfstring{\((p)\)}{p}-divergence}\label{pdivergence}
For technical reasons (cf.\ the proof of \cref{thm:key_estimate}),
we need to introduce, for any exponent \(p\in[1,\infty]\), a notion of
\emph{\((p)\)-divergence}
operator and investigate its basic properties.
We underline that, as we are going to see in the ensuing definition,
the exponent \(p\) is referring to the \(p\)-integrability of the divergence,
but the weak differential structure we are considering is always the one
associated with the \(2\)-Sobolev space \(W^{1,2}(\X)\).
\begin{definition}[\((p)\)-divergence]
Let \((\X,\sfd,\mm)\) be a metric measure space and \(p\in[1,\infty]\).
Then we define the space \(D({\rm div}_{(p)})\subset L^2(T\X)\) as the family
of all vector fields \(v\in L^2(T\X)\) for which there exists (a uniquely
determined) function \({\rm div}_{(p)}(v)\in L^p(\mm)\) such that
\[
\int\d f(v)\,\d\mm=-\int f\,{\rm div}_{(p)}(v)\,\d\mm,
\quad\text{ for every }f\in W^{1,2}(\X)\cap L^q(\mm),
\]
where \(q\in[1,\infty]\) is the conjugate exponent of \(p\),
namely \(\frac{1}{p}+\frac{1}{q}=1\).
\end{definition}

Observe that \(D({\rm div}_{(2)})=D({\rm div})\) and
\({\rm div}_{(2)}(v)={\rm div}(v)\) for every \(v\in D({\rm div}_{(2)})\).
\begin{lemma}\label{lem:indep_div_p}
Let \((\X,\sfd,\mm)\) be a metric measure space. Fix \(p,p'\in[1,\infty]\).
Suppose that a vector field \(v\in D({\rm div}_{(p)})\) satisfies
\({\rm div}_{(p)}(v)\in L^{p'}(\mm)\). Then \(v\in D({\rm div}_{(p')})\) and
\({\rm div}_{(p')}(v)={\rm div}_{(p)}(v)\).
\end{lemma}
\begin{proof}
Let \(q\) and \(q'\) be the conjugate exponents of \(p\) and \(p'\),
respectively. To prove the claim amounts to showing that
\(\int\d f(v)\,\d\mm=-\int f\,{\rm div}_{(p)}(v)\,\d\mm\) for every
\(f\in W^{1,2}(\X)\cap L^{q'}(\mm)\). Fix any such function \(f\).
Pick a point \(\bar x\in\X\) and a sequence \((\eta_n)_n\subset
{\rm Lip}_{bs}(\X)\) of \(1\)-Lipschitz functions \(\eta_n\colon\X\to[0,1]\)
such that \(\eta_n=1\) on \(B_n(\bar x)\) for every \(n\in\N\). The Leibniz
rule for the differential ensures that \(\eta_n f\in W^{1,2}(\X)\) and
\(\d(\eta_n f)=\eta_n\,\d f+f\,\d\eta_n\). In particular, we have that
\(f_n\coloneqq\big((\eta_n f)\wedge n\big)\vee(-n)\in W^{1,2}(\X)\cap L^{q'}(\mm)\)
and \(\d f_n=\nchi_{\{|\eta_n f|\leq n\}}(\eta_n\,\d f+f\,\d\eta_n)\).
By using the dominated convergence theorem, we deduce that
\(f_n\to f\) in \(W^{1,2}(\X)\) and in \(L^{q'}(\mm)\).
Given that each \(f_n\) is bounded and has bounded support, it holds that
\(f_n\in L^q(\mm)\) and so accordingly
\(\int\d f_n(v)\,\d\mm=-\int f_n\,{\rm div}_{(p)}(v)\,\d\mm\). By letting
\(n\to\infty\), we get \(\int\d f(v)\,\d\mm=-\int f\,{\rm div}_{(p)}(v)\,\d\mm\),
which yields the sought conclusion.
\end{proof}
\begin{remark}{\rm
It holds that
\begin{equation}\label{eq:int_div_zero}
\int{\rm div}_{(1)}(v)\,\d\mm=0,\quad
\text{ for every }v\in D({\rm div}_{(1)})\cap L^1(T\X).
\end{equation}
In order to prove it, fix any point \(\bar x\in\X\) and a sequence
\((\eta_n)_n\subset{\rm Lip}_{bs}(\X)\) of \(1\)-Lipschitz functions
\(\eta_n\colon\X\to[0,1]\) such that \(\eta_n=1\) on \(B_n(\bar x)\). By
dominated convergence theorem, it is easy to show that \(\eta_n\rightharpoonup 1\) weakly\(^*\)
in \(L^\infty(\mm)\) and
\(\d\eta_n(v)\rightharpoonup 0\) weakly in \(L^1(\mm)\).
Then
\[
\int{\rm div}_{(1)}(v)\,\d\mm=\lim_{n\to\infty}\int\eta_n\,{\rm div}_{(1)}(v)\,\d\mm
=-\lim_{n\to\infty}\int\d\eta_n(v)\,\d\mm=0
\]
is satisfied for every \(v\in D({\rm div}_{(1)})\cap L^1(T\X)\),
whence \eqref{eq:int_div_zero} follows.
\fr}\end{remark}
\begin{lemma}[Leibniz rule for \({\rm div}_{(p)}\)]\label{lem:Leibniz_div_p}
Let \((\X,\sfd,\mm)\) be a metric measure space. Let \(p\in[1,\infty]\),
\(v\in D({\rm div}_{(p)})\), and \(f\in{\rm Lip}_{bs}(\X)\) be given. Then it
holds that \(f\cdot v\in D({\rm div}_{(p)})\) and
\[
{\rm div}_{(p)}(f\cdot v)=\d f(v)+f\,{\rm div}_{(p)}(v).
\]
\end{lemma}
\begin{proof}
Fix any \(g\in W^{1,2}(\X)\cap L^q(\mm)\), where \(q\) stands
for the conjugate exponent of \(p\).
It holds that \(fg\in W^{1,2}(\X)\cap L^q(\mm)\) and
\(\d(fg)=g\,\d f+f\,\d g\). Therefore, we may compute
\[\begin{split}
\int\d g(f\cdot v)\,\d\mm&=\int f\,\d g(v)\,\d\mm
=\int\d(fg)(v)\,\d\mm-\int g\,\d f(v)\,\d\mm\\
&=-\int g\big(f\,{\rm div}_{(p)}(v)+\d f(v)\big)\,\d\mm,
\end{split}\]
which shows that \(f\cdot v\in D({\rm div}_{(p)})\) and
\({\rm div}_{(p)}(f\cdot v)=\d f(v)+f\,{\rm div}_{(p)}(v)\), as desired.
\end{proof}
\begin{lemma}\label{lem:div_h_Ht}
Let \((\X,\sfd,\mm)\) be an \({\sf RCD}(K,\infty)\) space.
Let \(f\in D({\boldsymbol\Delta})\cap{\rm S}^2(\X)\)
be such that \({\boldsymbol\Delta}f\) has bounded support.
Then it holds that \({\sf h}_{{\rm H},t}(\nabla f)\in D({\rm div}_{(1)})\)
for every \(t>0\) and
\begin{equation}\label{eq:div_h_Ht_claim}
{\rm div}_{(1)}({\sf h}_{{\rm H},t}(\nabla f))={\sf h}_t^*{\boldsymbol\Delta}f.
\end{equation}
\end{lemma}
\begin{proof}
Given any function \(g\in{\rm Lip}_{bs}(\X)\), it holds that
\[\begin{split}
\int\langle\nabla g,{\sf h}_{{\rm H},t}(\nabla f)\rangle\,\d\mm
&\overset{\eqref{eq:h_Ht_self-adj}}=
\int\langle{\sf h}_{{\rm H},t}(\nabla g),\nabla f\rangle\,\d\mm
\overset{\eqref{eq:consist_with_heat_flow}}=
\int\langle\nabla{\sf h}_t g,\nabla f\rangle\,\d\mm
=-\int{\sf h}_t g\,\d{\boldsymbol\Delta}f\\
&\overset{\phantom{\eqref{eq:h_Ht_self-adj}}}=
-\int g\,\d\mathscr H_t({\boldsymbol\Delta}f)^+
+\int g\,\d\mathscr H_t({\boldsymbol\Delta}f)^-=
-\int g\,{\sf h}_t^*{\boldsymbol\Delta}f\,\d\mm.
\end{split}\]
Since any \(g\in W^{1,2}(\X)\cap L^\infty(\mm)\) can be
approximated strongly in \(W^{1,2}(\X)\) by a sequence
\((g_n)_n\subset{\rm Lip}_{bs}(\X)\) such that \(g_n\) weakly\(^*\)
converges to \(g\) in \(L^\infty(\mm)\), we eventually conclude that
\[
\int\langle\nabla g,{\sf h}_{{\rm H},t}(\nabla f)\rangle\,\d\mm=
-\int g\,{\sf h}_t^*{\boldsymbol\Delta}f\,\d\mm,\quad
\text{ for every }g\in W^{1,2}(\X)\cap L^\infty(\mm),
\]
whence the statement follows.
\end{proof}
\subsubsection{Mollified heat flows}\label{Mollified}
Let us briefly recall the notion of \emph{mollified heat flow}
and its main properties, see e.g.\ \cite[Proposition 5.2.18]{GPbook}.
\medskip

Let \((\X,\sfd,\mm)\) be infinitesimally Hilbertian.
Let \(\varphi\in C^\infty_c(0,+\infty)\) and \(p\in[1,\infty]\) be given.
Then for any function \(f\in L^2(\mm)\cap L^p(\mm)\) we define the
\emph{\(\varphi\)-mollified heat flow} of \(f\) as
\[
{\sf h}_\varphi f\coloneqq\int_0^{+\infty}\varphi(t)\,{\sf h}_t f\,\d t
\in L^2(\mm)\cap L^p(\mm),
\]
where the integral is intended in the sense of Bochner.
Then it holds that \({\sf h}_\varphi f\in D(\Delta)\) and
\begin{equation}\label{eq:key_ineq_moll_hf}
\|\Delta{\sf h}_\varphi f\|_{L^p(\mm)}\leq C(\varphi)\,\|f\|_{L^p(\mm)},\quad
\text{ where we set }C(\varphi)\coloneqq\int_0^{+\infty}|\varphi'(t)|\,\d t.
\end{equation}
It is immediate to see that if \(\varphi\geq 0\) and
\(\int_0^{+\infty}\varphi(t)\,\d t=1\), then \({\sf h}_\varphi\)
fulfills the \emph{weak maximum principle}: if \(f\in L^2(\mm)\cap L^p(\mm)\)
satisfies \(f\leq C\) \(\mm\)-a.e.\ for some \(C\in\R\), then
\({\sf h}_\varphi f\leq C\) \(\mm\)-a.e..
\medskip

Taking inspiration from the above notion, we propose
the following two definitions, whose well-posedness will
be discussed in \cref{rmk:well-posed_moll}:
\begin{definition}
Let \((\X,\sfd,\mm)\) be an \({\sf RCD}(K,\infty)\) space. Let
\(\varphi\in C^\infty_c(0,+\infty)\) be given. Then for any finite
signed Borel measure \(\mu\) on \(\X\) with \(\mu^+\), \(\mu^-\) of
finite second-moment we define
\begin{equation}\label{eq:def_moll_h_star}
{\sf h}_\varphi^*\mu\coloneqq\int_0^{+\infty}\varphi(t)\,{\sf h}_t^*\mu\,\d t
\in L^1(\mm),
\end{equation}
where the integral is in the sense of Bochner.
\end{definition}
\begin{definition}
Let \((\X,\sfd,\mm)\) be an \({\sf RCD}(K,\infty)\) space. Let
\(\varphi\in C^\infty_c(0,+\infty)\) be given. Then for any
\(v\in L^2(T\X)\) we define
\begin{equation}\label{eq:def_moll_h_Ht}
{\sf h}_{{\rm H},\varphi}(v)\coloneqq\int_0^{+\infty}\varphi(t)\,
{\sf h}_{{\rm H},t}(v)\,\d t\in L^2(T\X),
\end{equation}
where the integral is in the sense of Bochner.
\end{definition}
\begin{remark}\label{rmk:well-posed_moll}{\rm
Let us comment on the well-posedness of \eqref{eq:def_moll_h_star}:
we know from \cref{lem:h_star_meas} that
\((0,+\infty)\ni t\mapsto{\sf h}_t^*\mu\in L^1(\mm)\) is strongly measurable.
Since \(\{{\sf h}_t^*\mu\}_{t>0}\subset L^1(\mm)\) is also bounded, we
deduce that \((0,+\infty)\ni t\mapsto\varphi(t)\,{\sf h}_t^*\mu\in L^1(\mm)\)
is Bochner integrable on \((\R,\mathcal L^1)\).

About \eqref{eq:def_moll_h_Ht}: we have that \([0,+\infty)\ni t\mapsto
{\sf h}_{{\rm H},t}(v)\in L^2(T\X)\) is continuous, thus accordingly
\([0,+\infty)\ni t\mapsto\varphi(t)\,{\sf h}_{{\rm H},t}(v)\in L^2(T\X)\)
is Bochner integrable on \((\R,\mathcal L^1)\).
\fr}\end{remark}
\begin{lemma}\label{lem:cont_at_0_moll_hf}
Let \((\X,\sfd,\mm)\) be an \({\sf RCD}(K,\infty)\) space. Let
\((\varphi_n)_n\subset C^\infty_c(0,+\infty)\) satisfy \(\varphi_n\geq 0\)
and \(\int_0^{+\infty}\varphi_n(t)\,\d t=1\) for every \(n\in\N\).
Suppose \(\mu_n\coloneqq\varphi_n\mathcal L^1\rightharpoonup\delta_0\)
with respect to the narrow topology. Then for any \(v\in L^2(T\X)\) it holds
that \({\sf h}_{{\rm H},\varphi_n}(v)\to v\) strongly in \(L^2(T\X)\)
as \(n\to\infty\).
\end{lemma}
\begin{proof}
Let \(v\in L^2(T\X)\) be fixed. We know that the curve
\([0,+\infty)\ni t\mapsto{\sf h}_{{\rm H},t}(v)\in L^2(T\X)\)
is continuous, so the function \(f(t)\coloneqq\big\|{\sf h}_{{\rm H},t}(v)
-v\big\|_{L^2(T\X)}\in\R\) belongs to \(C_b(0,+\infty)\). Then
\[\begin{split}
\big\|{\sf h}_{{\rm H},\varphi_n}(v)-v\big\|_{L^2(T\X)}
&=\bigg\|\int_0^{+\infty}\varphi_n(t)\,{\sf h}_{{\rm H},t}(v)\,\d t
-v\,\bigg\|_{L^2(T\X)}\\
&=\bigg\|\int_0^{+\infty}\varphi_n(t)\big({\sf h}_{{\rm H},t}(v)-v\big)
\,\d t\,\bigg\|_{L^2(T\X)}\\
&\leq\int_0^{+\infty}\varphi_n(t)f(t)\,\d t\to\int f\,\d\delta_0=f(0)=0,
\quad\text{ as }n\to\infty.
\end{split}\]
Therefore, the statement is achieved.
\end{proof}

Many of the results concerning \({\sf h}_t\), \({\sf h}_t^*\),
\({\sf h}_{{\rm H},t}\) have a natural counterpart for \({\sf h}_\varphi\),
\({\sf h}_\varphi^*\), \({\sf h}_{{\rm H},\varphi}\). We collect some of
them in the following result.
\begin{proposition}\label{prop:prop_moll_hf}
Let \((\X,\sfd,\mm)\) be an \({\sf RCD}(K,\infty)\) space. Fix any function
\(\varphi\in C^\infty_c(0,+\infty)\) such that \(\varphi\geq 0\) and
\(\int_0^{+\infty}\varphi(t)\,\d t=1\). Then the following properties
are satisfied:
\begin{itemize}
\item[\(\rm i)\)] \textsc{Bakry--\'{E}mery estimate.}
Given any \(v\in L^2(T\X)\cap L^\infty(T\X)\), it holds that
\begin{equation}\label{eq:BE_moll}
|{\sf h}_{{\rm H},\varphi}(v)|^2\leq e^{2\bar t\max\{0,-K\}}\,
{\sf h}_\varphi(|v|^2),\quad\text{ in the }\mm\text{-a.e.\ sense,}
\end{equation}
whenever \(\bar t>0\) is chosen so that \({\rm spt}(\varphi)\subset[0,\bar t\,]\).
\item[\(\rm ii)\)] Fix \(f\in D({\boldsymbol\Delta})\cap{\rm S}^2(\X)\)
with \({\boldsymbol\Delta}f\) of bounded support. Then
\({\sf h}_{{\rm H},\varphi}(\nabla f)\in D({\rm div}_{(1)})\) and
\begin{equation}\label{eq:div_h_H_moll}
{\rm div}_{(1)}({\sf h}_{{\rm H},\varphi}(\nabla f))
={\sf h}_\varphi^*{\boldsymbol\Delta}f.
\end{equation}
\item[\(\rm iii)\)] It holds that
\begin{equation}\label{eq:moll_hf_cont}
{\sf h}_{{\rm H},\varphi}\colon L^2(T\X)\to L^2(T\X),
\quad\text{ is linear and continuous.}
\end{equation}
\item[\(\rm iv)\)] Given any function \(f\in W^{1,2}(\X)\), it holds that
\begin{equation}\label{eq:consist_moll_hf}
{\sf h}_{{\rm H},\varphi}(\nabla f)=\nabla{\sf h}_\varphi f.
\end{equation}
\item[\(\rm v)\)] Let \(f\in{\rm S}^2(\X)\cap L^\infty(\mm)\)
be given. Then it holds that
\({\sf h}_{{\rm H},\varphi}(\nabla f)\in D({\rm div}_{(\infty)})\) and
\[
\big\|{\rm div}_{(\infty)}({\sf h}_{{\rm H},\varphi}(\nabla f))\big\|
_{L^\infty(\mm)}\leq C(\varphi)\,\|f\|_{L^\infty(\mm)}.
\]
\end{itemize}
\end{proposition}
\begin{proof}
\ \\
{\color{blue}i)} Let \(v\in L^2(T\X)\cap L^\infty(T\X)\) be given.
Then we may estimate
\[\begin{split}
|{\sf h}_{{\rm H},\varphi}(v)|^2&\overset{\phantom{\eqref{eq:BE_vf}}}=
\bigg|\int_0^{+\infty}\varphi(t)\,{\sf h}_{{\rm H},t}(v)\,\d t\bigg|^2\leq
\bigg(\int_0^{+\infty}\varphi(t)\,|{\sf h}_{{\rm H},t}(v)|\,\d t\bigg)^2\leq
\int_0^{+\infty}\varphi(t)\,|{\sf h}_{{\rm H},t}(v)|^2\,\d t\\
&\overset{\eqref{eq:BE_vf}}\leq
e^{2\bar t\max\{0,-K\}}\int_0^{+\infty}\varphi(t)\,{\sf h}_t(|v|^2)\,\d t
=e^{2\bar t\max\{0,-K\}}\,{\sf h}_\varphi(|v|^2),\quad\mm\text{-a.e.},
\end{split}\]
where in the second inequality we applied Chebyshev's inequality to \(t\mapsto|{\sf h}_{{\rm H},t}(v)|\)
with respect to the Borel probability measure \(\varphi\mathcal L^1\).
Therefore, the claimed inequality \eqref{eq:BE_moll} is proved.\\
{\color{blue}ii)} First, observe that
\({\rm div}_{(1)}\colon D({\rm div}_{(1)})\to L^1(\mm)\) is a closed operator,
namely if \((v_n)_n\subset D({\rm div}_{(1)})\) satisfies \(v_n\to v\) in
\(L^2(T\X)\) and \({\rm div}_{(1)}(v_n)\to H\) in \(L^1(\mm)\) for some
\(v\in L^2(T\X)\) and \(H\in L^1(\mm)\), then \(v\in D({\rm div}_{(1)})\) and
\({\rm div}_{(1)}(v)=H\). Now fix \(f\in D({\boldsymbol\Delta})\cap{\rm S}^2(\X)\)
with \({\boldsymbol\Delta}f\) of bounded support. \cref{lem:div_h_Ht}
says that \({\sf h}_{{\rm H},t}(\nabla f)\in D({\rm div}_{(1)})\) and
\({\rm div}_{(1)}({\sf h}_{{\rm H},t}(\nabla f))={\sf h}_t^*{\boldsymbol\Delta}f\)
hold for every \(t>0\). As observed in \cref{rmk:well-posed_moll},
the curves \(t\mapsto{\sf h}_{{\rm H},t}(\nabla f)\in L^2(T\X)\) and
\(t\mapsto{\sf h}_t^*{\boldsymbol\Delta}f\in L^1(\mm)\) are Bochner integrable
on \((\R,\varphi\mathcal L^1)\). Therefore, an application of Hille's
Theorem (see e.g.\ \cite[Theorem 1.3.15]{GPbook}) ensures that
\({\sf h}_{{\rm H},\varphi}(\nabla f)=\int_0^{+\infty}\varphi(t)\,
{\sf h}_{{\rm H},t}(\nabla f)\,\d t\in D({\rm div}_{(1)})\) and
\[\begin{split}
{\rm div}_{(1)}({\sf h}_{{\rm H},\varphi}(\nabla f))
&\overset{\phantom{\eqref{eq:dual_h_t}}}=
{\rm div}_{(1)}\bigg(\int_0^{+\infty}\varphi(t)\,{\sf h}_{{\rm H},t}(\nabla f)\,\d t\bigg)=
\int_0^{+\infty}\varphi(t)\,{\rm div}_{(1)}({\sf h}_{{\rm H},t}(\nabla f))\,\d t\\
&\overset{\eqref{eq:div_h_Ht_claim}}=
\int_0^{+\infty}\varphi(t)\,{\sf h}_t^*{\boldsymbol\Delta}f\,\d t
={\sf h}_\varphi^*{\boldsymbol\Delta}f.
\end{split}\]
This proves \eqref{eq:div_h_H_moll}, as desired.\\
{\color{blue}iii)} Linearity follows from the linearity of each
operator \({\sf h}_{{\rm H},t}\). To prove continuity, notice that
\[\begin{split}
\|{\sf h}_{{\rm H},\varphi}(v)\|_{L^2(T\X)}&=\bigg\|\int_0^{+\infty}
\varphi(t)\,{\sf h}_{{\rm H},t}(v)\,\d t\,\bigg\|_{L^2(T\X)}\leq
\int_0^{+\infty}\varphi(t)\,\|{\sf h}_{{\rm H},t}(v)\|_{L^2(T\X)}\,\d t\\
&\leq\int_0^{+\infty}\varphi(t)\,\|v\|_{L^2(T\X)}\,\d t=\|v\|_{L^2(T\X)},
\quad\text{ for every }v\in L^2(T\X).
\end{split}\]
{\color{blue}iv)} The argument is similar to the one in the proof of
item ii): the operator \(\nabla\) is closed, and the curves
\(t\mapsto{\sf h}_t f\in L^2(\mm)\) and
\(t\mapsto{\sf h}_{{\rm H},t}(\nabla f)\in L^2(T\X)\)
are Bochner integrable on \((\R,\varphi\mathcal L^1)\),
thus by applying Hille's Theorem we obtain that
\[
{\sf h}_{{\rm H},\varphi}(\nabla f)=\int_0^{+\infty}\varphi(t)\,
{\sf h}_{{\rm H},t}(\nabla f)\,\d t
\overset{\eqref{eq:consist_with_heat_flow}}=
\int_0^{+\infty}\varphi(t)\nabla{\sf h}_t f\,\d t
=\nabla\bigg(\int_0^{+\infty}\varphi(t)\,{\sf h}_t f\,\d t\bigg)
=\nabla{\sf h}_\varphi f.
\]
{\color{blue}v)} By a standard cut-off argument, we can find a sequence
\((f_n)_n\subset W^{1,2}(\X)\cap L^1(\mm)\) such that \(|f_n|\leq|f|\)
holds \(\mm\)-a.e.\ for every \(n\in\N\) and \(\nabla f_n\to\nabla f\)
in \(L^2(T\X)\). Thanks to \eqref{eq:key_ineq_moll_hf}, we have that
\(({\sf h}_\varphi f_n)_n\subset D(\Delta)\) and
\(\|\Delta{\sf h}_\varphi f_n\|_{L^\infty(\mm)}\leq C(\varphi)\,\|f\|_{L^\infty(\mm)}\)
for every \(n\in\N\). Hence, up to a not relabelled subsequence,
it holds \(\Delta{\sf h}_\varphi f_n\rightharpoonup H\) weakly\(^*\)
in \(L^\infty(\mm)\) for some \(H\in L^\infty(\mm)\). Consequently,
for any given function \(g\in W^{1,2}(\X)\cap L^1(\mm)\) we may compute
\[\begin{split}
\int g H\,\d\mm
&\overset{\phantom{\eqref{eq:consist_moll_hf}}}=
\lim_{n\to\infty}\int g\,\Delta{\sf h}_\varphi f_n\,\d\mm=-
\lim_{n\to\infty}\int\langle\nabla g,\nabla{\sf h}_\varphi f_n\rangle\,\d\mm\\
&\overset{\eqref{eq:consist_moll_hf}}=-\lim_{n\to\infty}
\int\langle\nabla g,{\sf h}_{{\rm H},\varphi}(\nabla f_n)\rangle\,\d\mm
\overset{\eqref{eq:moll_hf_cont}}=
-\int\langle\nabla g,{\sf h}_{{\rm H},\varphi}(\nabla f)\rangle\,\d\mm.
\end{split}\]
Then \({\sf h}_{{\rm H},\varphi}(\nabla f)\in D({\rm div}_{(\infty)})\)
and \(\|{\rm div}_{(\infty)}({\sf h}_{{\rm H},\varphi}(\nabla f))\|_
{L^\infty(\mm)}=\|H\|_{L^\infty(\mm)}\leq C(\varphi)\,\|f\|_{L^\infty(\mm)}\).
\end{proof}
\subsubsection{Gauss--Green formula}
Let \((\X,\sfd,\mm)\) be an \({\sf RCD}(K,N)\) space with \(N<\infty\)
and \(E\subset\X\) a set of finite perimeter. By the \emph{tangent module}
\(L^2_{\partial E}(T\X)\) \emph{over the boundary of \(E\)}
we mean the
\(L^2\big(P(E,\cdot)\big)\)-normed \(L^\infty\big(P(E,\cdot)\big)\)-module
whose existence was proved in \cite[Theorem 2.1]{BPS19}.
The \emph{trace} operator \({\rm tr}_{\partial E}\colon H^{1,2}_C(T\X)\cap
L^\infty(T\X)\to L^2_{\partial E}(T\X)\) is defined as in \cite[Section 2]{BPS19}.
\medskip

Given any \(v\in H^{1,2}_C(T\X)\cap L^\infty(T\X)\) and \(C\geq 0\),
it holds that
\begin{equation}\label{eq:max_princ_tr}
|v|\leq C,\text{ in the }\mm\text{-a.e.\ sense}
\quad\Longrightarrow\quad|{\rm tr}_{\partial E}(v)|\leq C,
\text{ in the }P(E,\cdot)\text{-a.e.\ sense.}
\end{equation}
It is also easy to prove that if \(v\in H^{1,2}_C(T\X)\cap L^\infty(T\X)\)
and \(f\in{\rm Lip}_c(\X)\), then
\begin{equation}\label{eq:tr_LIP_lin}
{\rm tr}_{\partial E}(f\cdot v)=f\cdot{\rm tr}_{\partial E}(v).
\end{equation}
Here, we are using the fact that \(f\cdot v\in H^{1,2}_C(T\X)
\cap L^\infty(T\X)\), as granted by \cref{lem:mult_H12_C}.
\begin{theorem}[Gauss--Green formula {\cite[Theorem 2.2]{BPS19}}]
\label{thm:Gauss-Green}
Let \((\X,\sfd,\mm)\) be an \({\sf RCD}(K,N)\) space with \(N<\infty\).
Let \(E\subset\X\) be a set of finite perimeter with \(\mm(E)<+\infty\).
Then there exists a unique vector field \(\nu_E\in L^2_{\partial E}(T\X)\)
such that \(|\nu_E|=1\) holds \(P(E,\cdot)\)-a.e.\ and
\begin{equation}\label{eq:Gauss-Green}
\int_E{\rm div}(v)\,\d\mm=\int\langle{\rm tr}_{\partial E}(v),\nu_E\rangle
\,\d P(E,\cdot),\quad\text{ for every }v\in H^{1,2}_C(T\X)\cap D({\rm div})
\cap L^\infty(T\X).
\end{equation}
The vector field \(\nu_E\) is said to be the \emph{outer unit normal}
to \(E\).
\end{theorem}
\subsubsection{Main estimate}

Let \((\X,\sfd,\mm)\) be an \({\sf RCD}(K,\infty)\) space.
Given any point \(\bar x\in\X\), we define the \(1\)-Lipschitz function
\(\sfd_{\bar x}\colon\X\to[0,+\infty)\) as
\[
\sfd_{\bar x}(x)\coloneqq\sfd(x,\bar x),\quad\text{ for every }x\in\X.
\]
Given that \((\X,\sfd)\) is a length space, it holds
that \({\rm lip}(\sfd_{\bar x})\equiv 1\). Hence, the results
in \cite{GigliHan14} (cf. also \cite{CheegerDiff}, \cite[Theorem 48]{ACDM15}) yield
\begin{equation}\label{eq:mwug_dist}
|D\sfd_{\bar x}|=1,\quad\text{ in the }\mm\text{-a.e.\ sense.}
\end{equation}
Moreover, the chain rule for minimal weak upper gradients gives
\(|D\sfd_{\bar x}^2|=2\sfd_{\bar x}|D\sfd_{\bar x}|\) and thus
\begin{equation}\label{eq:mwug_dist_2}
|D\sfd_{\bar x}^2|=2\sfd_{\bar x},\quad\text{ in the }\mm\text{-a.e.\ sense.}
\end{equation}
\begin{lemma}\label{lem:m_superp}
Let \((\X,\sfd,\mm)\) be an \({\sf RCD}(K,\infty)\) space.
Let \(\bar x\in\X\) be given. Then it holds that
\[
\mm=\int_0^{+\infty}P(B_r(\bar x),\cdot)\,\d r.
\]
\end{lemma}
\begin{proof}
Given any \(f\in{\rm BV}_{\rm loc}(\X)\), we know from Theorem
\ref{thm:coarea} that \(\R\ni t\mapsto P(\{f\geq t\},\Omega)\)
is Borel measurable for every open set \(\Omega\subset\X\).
A standard application of the Dynkin \(\pi-\lambda\) Theorem ensures that
\(\R\ni t\mapsto P(\{f\geq t\},B)\) is Borel measurable for every Borel
set \(B\subset\X\), thus in particular the set-function \(\mu_f\),
defined as \(\mu_f(B)\coloneqq\int_\R P(\{f\geq t\},B)\,\d t\) for
every \(B\subset\X\) Borel, defines a boundedly-finite Borel measure
on \(\X\). Hence, by applying the coarea formula we deduce that
\(|Df|(\Omega)=\mu_f(\Omega)\) for all open sets \(\Omega\subset\X\),
so that accordingly \(|Df|=\mu_f\) as measures.

Let us now consider the function \(f\coloneqq\sfd_{\bar x}\), which
is locally Lipschitz and so belongs to \({\rm BV}_{\rm loc}(\X)\).
Thanks to \eqref{eq:mwug_dist} and the results in \cite{GigliHan14},
we know that the total variation measure of \(\sfd_{\bar x}\) coincides
with \(\mm\). Moreover, we have that
\(P(B_r(\bar x),\cdot)=P(B_r(\bar x)^c,\cdot)=P(\{\sfd_{\bar x}\geq r\},\cdot)\)
holds for a.e.\ \(r>0\). Therefore, the identity \(|D\sfd_{\bar x}|=
\mu_{\sfd_{\bar x}}\) proved above yields the statement.
\end{proof}

Before passing to the proof of the main result we achieve in this section
(i.e., Theorem \ref{thm:MainEst}), we state and prove two preliminary results.
The first one says that on finite-dimensional \(\sf RCD\) spaces, for any
given point \(\bar x\in\X\) the outer unit normal \(\nu_{B_r(\bar x)}\)
to the ball \(B_r(\bar x)\) coincides with
\(\frac{1}{2r}\nabla\sfd_{\bar x}^2\) for a.e.\ radius \(r>0\),
in some suitable (weak) sense.
\begin{proposition}\label{prop:nu_sphere}
Let \((\X,\sfd,\mm)\) be an \({\sf RCD}(K,N)\) space with \(N<\infty\).
Let \(\bar x\in\X\) be given. Then for any Borel set \(E\subset\X\) and any
\(v\in H^{1,2}_C(T\X)\cap D({\rm div})\cap D({\rm div}_{(1)})\cap L^\infty(T\X)\)
it holds
\begin{equation}\label{eq:form_nabla_distb}
2r\int_E\langle{\rm tr}_{\partial B_r(\bar x)}(v),\nu_{B_r(\bar x)}\rangle
\,\d P(B_r(\bar x),\cdot)=\int_E\langle v,\nabla\sfd_{\bar x}^2\rangle
\,\d P(B_r(\bar x),\cdot),\quad\text{ for a.e.\ }r>0.
\end{equation}
\end{proposition}
\begin{proof}
Fix any
\(w\in H^{1,2}_C(T\X)\cap D({\rm div})\cap D({\rm div}_{(1)})\cap L^\infty(T\X)\)
such that \(|w|\) is concentrated on a compact set. Pick
\(\eta\colon\X\to[0,1]\) Lipschitz with compact support and \(\eta=1\)
on \({\rm spt}(|w|)\). Notice that one has \(\eta\,\sfd_{\bar x}^2
\in{\rm Lip}_c(\X)\). The Leibniz rule for the gradient gives
\(\nabla(\eta\,\sfd_{\bar x}^2)=\eta\nabla\sfd_{\bar x}^2
+\sfd_{\bar x}^2\nabla\eta\), so that accordingly \(\nabla(\eta\,
\sfd_{\bar x}^2)=\nabla\sfd_{\bar x}^2\) on \({\rm spt}(|w|)\).
Since \({\rm spt}({\rm div}(w))\subset{\rm spt}(|w|)\), we deduce that
\begin{equation}\label{eq:form_nabla_distb_aux1}
\eta\,\sfd_{\bar x}^2=\sfd_{\bar x}^2,\quad
\nabla(\eta\,\sfd_{\bar x}^2)=\nabla\sfd_{\bar x}^2,
\quad\text{ on }{\rm spt}({\rm div}(w)).
\end{equation}
The Gauss--Green formula \eqref{eq:Gauss-Green} and
\cref{lem:indep_div_p} ensure that for a.e.\ \(r>0\) it holds that
\[
-\int_{B_r(\bar x)^c}{\rm div}(w)\,\d\mm\overset{\eqref{eq:int_div_zero}}=
\int_{B_r(\bar x)}{\rm div}(w)\,\d\mm=
\int\langle{\rm tr}_{\partial B_r(\bar x)}(w),
\nu_{B_r(\bar x)}\rangle\,\d P(B_r(\bar x),\cdot).
\]
Multiplying the above identity by \(2r\) and then integrating it over
\(r\in(0,+\infty)\), we obtain
\[\begin{split}
2\int_0^{+\infty}r\int\langle{\rm tr}_{\partial B_r(\bar x)}(w),
\nu_{B_r(\bar x)}\rangle\,\d P(B_r(\bar x),\cdot)\,\d r
&=-2\int_0^{+\infty}r\int_{B_r(\bar x)^c}{\rm div}(w)\,\d\mm\,\d r\\
&=-2\int{\rm div}(w)(x)\int_0^{\sfd(x,\bar x)}r\,\d r\,\d\mm(x)\\
&=-\int{\rm div}(w)\,\sfd_{\bar x}^2\,\d\mm
\overset{\eqref{eq:form_nabla_distb_aux1}}=
\int\langle w,\nabla\sfd_{\bar x}^2\rangle\,\d\mm.
\end{split}\]
Now fix
\(v\in H^{1,2}_C(T\X)\cap D({\rm div})\cap D({\rm div}_{(1)})\cap L^\infty(T\X)\)
and \(\varphi\in C_c(\R)\). Pick any \((f_n)_n\subset{\rm Lip}_c(\X)\)
that converges to \(\varphi\circ\sfd_{\bar x}\,\nchi_E\) in \(L^2(\mm)\).
Note that \(f_n\cdot v\in H^{1,2}_C(T\X)\cap D({\rm div})\cap D({\rm div}_{(1)})
\cap L^\infty(T\X)\) by \cref{lem:mult_H12_C} and
\cref{lem:Leibniz_div_p}. By plugging \(w=f_n\cdot v\) into the previous
identities, we get
\[
2\int_0^{+\infty}r\int f_n\langle{\rm tr}_{\partial B_r(\bar x)}(v),
\nu_{B_r(\bar x)}\rangle\,\d P(B_r(\bar x),\cdot)\,\d r=
\int f_n\langle v,\nabla\sfd_{\bar x}^2\rangle\,\d\mm,
\quad\text{ for every }n\in\N.
\]
Given that \(\mm=\int_0^{+\infty}P(B_r(\bar x),\cdot)\,\d r\) by
\cref{lem:m_superp}, by letting \(n\to\infty\) in the above identity we
deduce that
\[
2\int_0^{+\infty}\varphi(r)\,r\int_E\langle{\rm tr}_{\partial B_r(\bar x)}(v),
\nu_{B_r(\bar x)}\rangle\,\d P(B_r(\bar x),\cdot)\,\d r=\int_0^{+\infty}
\varphi(r)\int_E\langle v,\nabla\sfd_{\bar x}^2\rangle\,
\d P(B_r(\bar x),\cdot)\,\d r.
\]
By exploiting the arbitrariness of \(\varphi\in C_c(\R)\), we eventually
conclude that \eqref{eq:form_nabla_distb} is verified.
\end{proof}
The next result provides a formula for the outer unit normal \(\nu_{E\cap F}\)
to the intersection between two sets of finite perimeter \(E\) and \(F\) in a
finite-dimensional \(\sf RCD\) space, under the additional assumption
that \(\partial^e E\) and \(\partial^e F\) are essentially disjoint
(which is sufficient for our purposes).
\begin{proposition}\label{prop:nu_inters}
Let \((\X,\sfd,\mm)\) be an \({\sf RCD}(K,N)\) space with \(N<\infty\). Let
\(E,F\subset\X\) be sets of finite perimeter satisfying \(\mm(E),\mm(F)<+\infty\)
and \(\mathcal H^{\rm cod\text{-}1}(\partial^e E\cap\partial^e F)=0\).
Then it holds
\begin{equation}\label{eq:nu_inters_claim}
\int\langle{\rm tr}_{\partial(E\cap F)}(v),\nu_{E\cap F}\rangle\,\d P(E\cap F,\cdot)=
\int_{F^{(1)}}\langle{\rm tr}_{\partial E}(v),\nu_E\rangle\,\d P(E,\cdot)+
\int_{E^{(1)}}\langle{\rm tr}_{\partial F}(v),\nu_F\rangle\,\d P(F,\cdot),
\end{equation}
for every \(v\in H^{1,2}_C(T\X)\cap D({\rm div})\cap L^\infty(T\X)\).
\end{proposition}
\begin{proof}
Fix any \(w\in H^{1,2}_C(T\X)\cap D({\rm div})\cap L^\infty(T\X)\).
\cref{thm:Gauss-Green},  \cref{lem:per_inters_general},
and \eqref{eq:density+bdry} yield
\[\begin{split}
\int_E{\rm div}(w)\,\d\mm&=
\int_{F^{(1)}}\langle{\rm tr}_{\partial E}(w),\nu_E\rangle\,\d P(E,\cdot)+
\int_{F^{(0)}}\langle{\rm tr}_{\partial E}(w),\nu_E\rangle\,\d P(E,\cdot),\\
\int_{E\setminus F}{\rm div}(w)\,\d\mm&=
\int_{F^{(0)}}\langle{\rm tr}_{\partial(E\setminus F)}(w),
\nu_{E\setminus F}\rangle\,\d P(E,\cdot)+
\int_{E^{(1)}}\langle{\rm tr}_{\partial(E\setminus F)}(w),
\nu_{E\setminus F}\rangle\,\d P(F,\cdot),\\
\int_{E\cap F}{\rm div}(w)\,\d\mm&=
\int_{F^{(1)}}\langle{\rm tr}_{\partial(E\cap F)}(w),\nu_{E\cap F}\rangle
\,\d P(E,\cdot)+\int_{E^{(1)}}\langle{\rm tr}_{\partial(E\cap F)}(w),
\nu_{E\cap F}\rangle\,\d P(F,\cdot).
\end{split}\]
By suitably adding and subtracting the above identities, we thus
obtain that
\[\begin{split}
0=&\int_{F^{(1)}}\big(\langle{\rm tr}_{\partial E}(w),\nu_E\rangle
-\langle{\rm tr}_{\partial(E\cap F)}(w),\nu_{E\cap F}\rangle\big)
\,\d P(E,\cdot)\\
&+\int_{F^{(0)}}\big(\langle{\rm tr}_{\partial E}(w),\nu_E\rangle
-\langle{\rm tr}_{\partial(E\setminus F)}(w),\nu_{E\setminus F}
\rangle\big)\,\d P(E,\cdot)\\
&-\int_{E^{(1)}}\big(\langle{\rm tr}_{\partial(E\setminus F)}(w),
\nu_{E\setminus F}\rangle+\langle{\rm tr}_{\partial(E\cap F)}(w),
\nu_{E\cap F}\rangle\big)\,\d P(F,\cdot).
\end{split}\]
Now pick a sequence \((f_n)_n\subset{\rm Lip}_c(\X)\)
such that \(f_n\to\nchi_{F^{(1)}\cap\partial^e E}\) in
\(L^1\big(P(E,\cdot)+P(F,\cdot)\big)\). Given any vector field
\(v\in H^{1,2}_C(T\X)\cap D({\rm div})\cap L^\infty(T\X)\),
we know from \cref{lem:Leibniz_div_p} and \cref{lem:mult_H12_C} that
\((f_n\cdot v)_n\subset H^{1,2}_C(T\X)\cap D({\rm div})\cap L^\infty(T\X)\),
thus in particular we can plug \(w=f_n\cdot v\) into the previous identity,
obtaining that
\[\begin{split}
0=&\int_{F^{(1)}}f_n\big(\langle{\rm tr}_{\partial E}(v),\nu_E\rangle
-\langle{\rm tr}_{\partial(E\cap F)}(v),\nu_{E\cap F}\rangle\big)
\,\d P(E,\cdot)\\
&+\int_{F^{(0)}}f_n\big(\langle{\rm tr}_{\partial E}(v),\nu_E\rangle
-\langle{\rm tr}_{\partial(E\setminus F)}(v),\nu_{E\setminus F}
\rangle\big)\,\d P(E,\cdot)\\
&-\int_{E^{(1)}}f_n\big(\langle{\rm tr}_{\partial(E\setminus F)}(v),
\nu_{E\setminus F}\rangle+\langle{\rm tr}_{\partial(E\cap F)}(v),
\nu_{E\cap F}\rangle\big)\,\d P(F,\cdot).
\end{split}\]
By letting \(n\to\infty\), we deduce that
\begin{equation}\label{eq:nu_inters_aux1}
\int_{F^{(1)}}\langle{\rm tr}_{\partial E}(v),\nu_E\rangle\,\d P(E,\cdot)
=\int_{F^{(1)}}\langle{\rm tr}_{\partial(E\cap F)}(v),\nu_{E\cap F}\rangle
\,\d P(E,\cdot).
\end{equation}
By means of a similar argument, one can also show that
\begin{equation}\label{eq:nu_inters_aux2}
\int_{E^{(1)}}\langle{\rm tr}_{\partial F}(v),\nu_F\rangle\,\d P(F,\cdot)
=\int_{E^{(1)}}\langle{\rm tr}_{\partial(E\cap F)}(v),\nu_{E\cap F}\rangle
\,\d P(F,\cdot).
\end{equation}
Therefore, by combining \eqref{eq:nu_inters_aux1} with
\eqref{eq:nu_inters_aux2}, we can finally conclude that
\eqref{eq:nu_inters_claim} is verified.
\end{proof}
\begin{theorem}\label{thm:key_estimate}
Let \((\X,\sfd,\mm)\) be an \({\sf RCD}(K,N)\) space with \(N<\infty\) and let $R>0$. Then there exists a constant
\(C_{K,N,R}>0\) such that the following holds. If $\bar x \in \X$ and \(E\subset\X\) is a set of locally finite perimeter, then
\begin{equation}\label{eq:key_estimate_claim}
r\,P\big(B_r(\bar x),E^{(1)}\big)\leq
C_{K,N,R}\,\mm\big(E\cap B_r(\bar x)\big)+r\,P\big(E,B_r(\bar x)\big),
\quad\text{ for every }r\in(0,R).
\end{equation}
\end{theorem}
\begin{proof}
First, notice that since the perimeter is local and $P(B_r(\bar x),\cdot)$ is concentrated on $\partial B_r(\bar x)$, it suffices to prove the theorem with $E\cap B_{R+1}(\bar x)$ instead of $E$. Hence we may suppose without loss of generality, from now on, that $E$ has finite perimeter and $\meas(E)<+\infty$.

First of all, we claim that for any \(\tilde r\in(0,R)\)
and \(\varepsilon\in(0,R-\tilde r)\) with $\varepsilon<\tilde r/2$ it holds that
\begin{equation}\label{eq:key_estimate_aux1}
r\,P\big(B_r(\bar x),E^{(1)}\big)\leq C_{K,N,R}\,\mm\big(E\cap B_r(\bar x)\big)
+(\tilde r+\varepsilon)P\big(E,B_r(\bar x)\big),
\quad\text{ for a.e.\ }r\in(0,\tilde r).
\end{equation}
In order to prove it, fix any \(C^2\)-function \(\psi\colon\R\to\R\)
satisfying the following properties:
\begin{itemize}
\item[a)] \(\psi(t)=t\), for every \(t\in[0,\tilde r^2]\),
\item[b)] \(0\leq\psi'(t)\leq(\tilde r+\varepsilon)/\sqrt t\), for
every \(t\in[\tilde r^2,\tilde r^2+\varepsilon]\),
\item[c)] \(\psi\) is constant on \([\tilde r^2+\varepsilon,+\infty)\),
\item[d)] \(\psi\) is concave on \([0,+\infty)\).
\end{itemize}
Now define the auxiliary function \(\eta\colon\X\to[0,+\infty)\)
as \(\eta\coloneqq\psi\circ\sfd_{\bar x}^2\). Thanks to the chain rule
for minimal weak upper gradients, one can easily see that
\(\eta\in{\rm S}^2(\X)\) and \(|D\eta|=|\psi'|\circ\sfd_{\bar x}^2\,
|D\sfd_{\bar x}^2|\). In particular, by exploiting a), b), c), and
\eqref{eq:mwug_dist_2} we obtain that
\begin{equation}\label{eq:key_estimate_aux2}
|D\eta|\leq 2(\tilde r+\varepsilon)
\nchi_{B(\bar x,\sqrt{\tilde r^2+\varepsilon})},
\quad\text{ in the }\mm\text{-a.e.\ sense.}
\end{equation}
Furthermore, by combining \cref{thm:Lapl_comp_dist} with \cref{prop:chain_rule_mv_Lapl} we deduce that
\(\eta\in D({\boldsymbol\Delta})\) and
\begin{equation}\label{eq:key_estimate_aux3}
{\boldsymbol\Delta}\eta=\psi'\circ\sfd_{\bar x}^2\,{\boldsymbol\Delta}
\sfd_{\bar x}^2+\psi''\circ\sfd_{\bar x}^2\,|D\sfd_{\bar x}^2|^2\,\mm
\overset{\rm d)}\leq\psi'\circ\sfd_{\bar x}^2\,{\boldsymbol\Delta}
\sfd_{\bar x}^2\overset{\rm c)}\leq \frac32 {\boldsymbol\Delta}\sfd_{\bar x}^2
|_{B(\bar x,\sqrt{\tilde r^2+\varepsilon})}\leq 2\,C_{K,N,R}\,\mm,
\end{equation}
where we used that $\psi'\circ \dist_{\bar x}^2 \le \tfrac32$ as $\varepsilon<\tilde r/2$ and the constant \(C_{K,N,R}>0\) is chosen so that
\(\tfrac32 N\,\tilde\tau_{K,N}(t)\leq C_{K,N,R}\) for every
\(t\in[0,\sqrt{\tilde r^2+\varepsilon}\,]\).
Since \(\psi'(t)=\psi''(t)=0\) for all \(t>\tilde r^2+\varepsilon\)
by c), we also have that \({\boldsymbol\Delta}\eta\)
is compactly-supported and in particular it is finite.
Choose any sequence \((\varphi_n)_n\subset C^\infty_c(0,+\infty)\) such that
\(\varphi_n\geq 0\), \(\int_0^{+\infty}\varphi_n(t)\,\d t=1\) for all \(n\in\N\)
and \(\varphi_n\mathcal L^1\rightharpoonup\delta_0\) with respect to the
narrow topology. Define
\[
v_n\coloneqq{\sf h}_{{\rm H},\varphi_n}(\nabla\eta)\in H^{1,2}_C(T\X),
\quad\text{ for every }n\in\N.
\]
The Bakry--\'{E}mery estimate \eqref{eq:BE_moll}, the inequality in
\eqref{eq:key_estimate_aux2}, and the weak maximum principle for
\({\sf h}_{\varphi_n}\) guarantee that
\(v_n\in L^\infty(T\X)\) for all \(n\in\N\). Moreover, since
\(\eta\in{\rm S}^2(\X)\cap L^\infty(\mm)\), we know from item v)
of \cref{prop:prop_moll_hf} that
\(v_n\in D({\rm div}_{(\infty)})\) for every \(n\in\N\). Finally,
an application of item ii) of \cref{prop:prop_moll_hf}
yields \(v_n\in D({\rm div}_{(1)})\) and
\({\rm div}_{(1)}(v_n)={\sf h}_{\varphi_n}^*{\boldsymbol\Delta}\eta\)
for every \(n\in\N\). In particular, thanks to \cref{lem:indep_div_p}
we know that \(v_n\in D({\rm div})\) and
\begin{equation}\label{eq:key_estimate_aux4}
{\rm div}(v_n)={\sf h}_{\varphi_n}^*{\boldsymbol\Delta}\eta,
\quad\text{ for every }n\in\N.
\end{equation}
All in all, we showed that \((v_n)_n\subset
H^{1,2}_C(T\X)\cap D({\rm div})\cap L^\infty(T\X)\),
thus \eqref{eq:nu_inters_claim} and \eqref{eq:Gauss-Green} give
\begin{equation}\label{eq:key_estimate_aux5}
\int_{E\cap B_r(\bar x)}{\rm div}(v_n)\,\d\mm=
\int_{B_r(\bar x)}\langle{\rm tr}_{\partial E}(v_n),\nu_E\rangle
\,\d P(E,\cdot)+\int_{E^{(1)}}\langle{\rm tr}_{\partial B_r(\bar x)}(v_n),
\nu_{B_r(\bar x)}\rangle\,\d P(B_r(\bar x),\cdot),
\end{equation}
for a.e.\ \(r\in(0,\tilde r)\). Here, we used the fact that
\(\mathcal H^{\rm cod\text{-}1}\big(\partial^e E\cap
\partial^e B_r(\bar x)\big)=0\) for a.e.\ \(r\in(0,\tilde r)\)
and that \(\Omega^{(1)}=\Omega\) whenever \(\Omega\subset\X\) is open.
Let us now fix such \(r\) and estimate the three terms appearing in
\eqref{eq:key_estimate_aux5}. Pick any \((f_i)_i\subset{\rm Lip}_c(\X)\)
satisfying \(0\leq f_i\leq 2\) for every \(i\in\N\), and
\(f_i\to\nchi_{E\cap B_r(\bar x)}\) both in \(L^1(\mm)\) and in \(L^2(\mm)\).
Hence, for any \(i,n\in\N\) it holds that
\[\begin{split}
\int f_i\,{\rm div}(v_n)\,\d\mm&\overset{\eqref{eq:key_estimate_aux4}}=
\int f_i\,{\sf h}_{\varphi_n}^*{\boldsymbol\Delta}\eta\,\d\mm
\overset{\eqref{eq:def_moll_h_star}}=
\int_0^{+\infty}\varphi_n(t)\int f_i\,{\sf h}_t^*{\boldsymbol\Delta}\eta
\,\d\mm\,\d t\\
&\overset{\phantom{\eqref{eq:key_estimate_aux4}}}=
\int_0^{+\infty}\varphi_n(t)\int{\sf h}_t f_i\,\d{\boldsymbol\Delta}\eta\,\d t
\overset{\eqref{eq:key_estimate_aux3}}\leq
2\,C_{K,N,R}\int_0^{+\infty}\varphi_n(t)\int{\sf h}_t f_i\,\d\mm\,\d t
\\
&\overset{\eqref{eq:h_t_mass-pres}}=2\,C_{K,N,R}\int_0^{+\infty}\varphi_n(t)
\int f_i\,\d\mm\,\d t=2\,C_{K,N,R}\int f_i\,\d\mm.
\end{split}\]
By letting \(i\to\infty\), we thus obtain that
\begin{equation}\label{eq:key_estimate_aux6}
\int_{E\cap B_r(\bar x)}{\rm div}(v_n)\,\d\mm\leq
2\,C_{K,N,R}\,\mm\big(E\cap B_r(\bar x)\big),\quad\text{ for every }n\in\N.
\end{equation}
Moreover, it follows from \eqref{eq:key_estimate_aux2},
\eqref{eq:BE_moll}, \eqref{eq:max_princ_tr}, and the weak maximum
principle for \({\sf h}_{\varphi_n}\) that the \(P(E,\cdot)\)-a.e.\ inequality
\(|{\rm tr}_{\partial E}(v_n)|\leq 2\,e^{\bar t_n\max\{0,-K\}}
(\tilde r+\varepsilon)\) is satisfied, where \(\bar t_n>0\)
is chosen so that \({\rm spt}(\varphi_n)\subset[0,\bar t_n]\) for
every \(n\in\N\). Observe that we can require that \(\lim_n\bar t_n=0\).
We may estimate
\begin{equation}\label{eq:key_estimate_aux7}
\bigg|\int_{B_r(\bar x)}\langle{\rm tr}_{\partial E}(v_n),\nu_E\rangle\,\d
P(E,\cdot)\bigg|\leq 2\,e^{\bar t_n\max\{0,-K\}}(\tilde r+\varepsilon)
P\big(E,B_r(\bar x)\big),\quad\text{ for every }n\in\N.
\end{equation}
Finally, by using \eqref{eq:mwug_dist_2}, \cref{prop:nu_sphere},
the fact that \(v_n\to\nabla\eta\) strongly in \(L^2(T\X)\) as
\(n\to\infty\) (by \cref{lem:cont_at_0_moll_hf}) and
that \(\mm=\int_0^{+\infty}P(B_r(\bar x),\cdot)\,\d r\),
and Fubini's theorem, we deduce that for a.e.\ \(r\in(0,\tilde r)\)
it holds that \(|D\sfd_{\bar x}^2|^2=4r^2\) in the
\(P(B_r(\bar x),\cdot)\)-a.e.\ sense and
\[\begin{split}
\int_{E^{(1)}}\langle{\rm tr}_{\partial B_r(\bar x)}(v_n),
\nu_{B_r(\bar x)}\rangle\,\d P(B_r(\bar x),\cdot)&=
\frac{1}{2r}\int_{E^{(1)}}\langle v_n,\nabla\sfd_{\bar x}^2\rangle
\,\d P(B_r(\bar x),\cdot)\\
&\to\frac{1}{2r}\int_{E^{(1)}}\langle\nabla\eta,\nabla\sfd_{\bar x}^2\rangle
\,\d P(B_r(\bar x),\cdot).
\end{split}\]
Given that \(\eta=\sfd_{\bar x}^2\) on a neighbourhood of \(B_r(\bar x)\),
we have that the identity \(\langle\nabla\eta,\nabla\sfd_{\bar x}^2\rangle=4r^2\)
holds \(P(B_r(\bar x),\cdot)\)-a.e.\ and thus accordingly
\begin{equation}\label{eq:key_estimate_aux8}
2r P\big(B_r(\bar x),E^{(1)}\big)=\lim_{n\to\infty}
\int_{E^{(1)}}\langle{\rm tr}_{\partial B_r(\bar x)}(v_n),
\nu_{B_r(\bar x)}\rangle\,\d P(B_r(\bar x),\cdot).
\end{equation}
By plugging \eqref{eq:key_estimate_aux6}, \eqref{eq:key_estimate_aux7},
and \eqref{eq:key_estimate_aux8} into \eqref{eq:key_estimate_aux5},
and letting \(n\to\infty\), we can finally conclude that the claimed
property \eqref{eq:key_estimate_aux1} is satisfied.

Let us pass to the verification of the main statement. Let \(r\in(0,R)\)
be fixed. Thanks to the first part of the proof and
\cref{cor:per_inters_ball}, we can find a sequence \((r_n)_n\subset(0,r)\)
with \(r_n\nearrow r\) such that \(P\big(E\cap B_{r_n}(\bar x)\big)
=P\big(E,B_{r_n}(\bar x)\big)+P\big(B_{r_n}(\bar x),E^{(1)}\big)\)
for every \(n\in\N\) and
\[
r_n\,P\big(B_{r_n}(\bar x),E^{(1)}\big)\leq
C_{K,N,R}\,\mm\big(E\cap B_{r_n}(\bar x)\big)+
(r+1/k)\,P\big(E,B_{r_n}(\bar x)\big),\quad\text{ for all }n,k\in\N.
\]
By letting \(k\to\infty\) in the above estimate, we get that
\begin{equation}\label{eq:key_estimate_aux9}
r_n\,P\big(B_{r_n}(\bar x),E^{(1)}\big)\leq
C_{K,N,R}\,\mm\big(E\cap B_{r_n}(\bar x)\big)+
r\,P\big(E,B_{r_n}(\bar x)\big),\quad\text{ for all }n\in\N.
\end{equation}
Since \((r_n)_n\) is increasing and \(\bigcup_n B_{r_n}(\bar x)=B_r(\bar x)\),
one has \(\mm\big(E\cap B_r(\bar x)\big)=\lim_n\mm\big(E\cap B_{r_n}(\bar x)\big)\)
and \(P\big(E,B_r(\bar x)\big)=\lim_n P(E,B_{r_n}(\bar x)\big)\).
Also, given that \(\nchi_{E\cap B_{r_n}(\bar x)}\to\nchi_{E\cap B_r(\bar x)}\)
in \(L^1(\mm)\) and the perimeter is lower semicontinuous with respect to
the \(L^1\)-convergence of sets, we have that \(P\big(E\cap B_r(\bar x)\big)
\leq\varliminf_n P\big(E\cap B_{r_n}(\bar x)\big)\).
Therefore, we can eventually conclude that
\[\begin{split}
&r\,P\big(B_r(\bar x),E^{(1)}\big)+r\,P\big(E,B_r(\bar x)\big)\\
\overset{\eqref{eq:ineq_per_inters}}\leq&r\,P\big(E\cap B_r(\bar x)\big)
\leq\varliminf_{n\to\infty}r_n\,P\big(E\cap B_{r_n}(\bar x)\big)
=\varliminf_{n\to\infty}r_n\Big(P\big(E,B_{r_n}(\bar x)\big)
+P\big(B_{r_n}(\bar x),E^{(1)}\big)\Big)\\
\overset{\eqref{eq:key_estimate_aux9}}\leq&
C_{K,N,R}\lim_{n\to\infty}\mm\big(E\cap B_{r_n}(\bar x)\big)
+\lim_{n\to\infty}(r_n+r)P\big(E,B_{r_n}(\bar x)\big)\\
\overset{\phantom{\eqref{eq:key_estimate_aux9}}}=&
C_{K,N,R}\,\mm\big(E\cap B_r(\bar x)\big)+2r\,P\big(E,B_r(\bar x)\big).
\end{split}\]
By subtracting \(r\,P\big(E,B_r(\bar x)\big)\), we obtain
\eqref{eq:key_estimate_claim}, thus yielding the sought conclusion.
\end{proof}

\begin{remark}[Comparison with the Gauss--Green formulae in \cite{BCM20}]\label{rem:BCM}
We stress that recently in \cite{BCM20} the authors proved Gauss--Green formulae for low regularity vector fields in locally compact $\RCD(K,\infty)$ spaces, see \cite[Section 6]{BCM20} and in particular \cite[Theorem 6.9]{BCM20}. Such a level of generality has a drawback when compared with the Gauss--Green formula in \cref{thm:Gauss-Green}, which is precisely due to the fact that the boundary term is a bit more difficult to handle. 

We notice anyway that even if \cite[Theorem 6.9]{BCM20} is proved for vector fields whose divergence is a signed Radon measure, we cannot directly apply it to the vector field $\nabla\dist_{\bar x}^2$, because it is not in $L^\infty(T\X)$ as required in the hypotheses of \cite[Theorem 6.9]{BCM20}.

All in all, if one wishes to apply \cite[Theorem 6.9]{BCM20} to the vector field $\nabla\dist_{\bar x}^2$ a truncation argument as the one we performed in the proof of \cref{thm:key_estimate} seems to be necessary. 

A way to overcome this could be to perform a different proof and apply \cite[Theorem 6.9]{BCM20} to the vector field $\nabla\dist_{\bar x}$, which is in $L^\infty(T\X)$ and has as divergence a measure on $\X\setminus \{\bar x\}$; and to the set of finite perimeter $E\cup B_r(\bar x)$. To conclude by using such an approach we should get a sharp estimate on the boundary term in the right-hand side of \cite[Equation (6.29)]{BCM20}, by using proper analogues of \cref{prop:nu_sphere}, and \cref{prop:nu_inters}. 

Since the latter described approach does not seem to produce a considerably shorter proof of our result in \cref{thm:key_estimate}, and since the study of the $(p)$-divergence and the mollified heat flow for measures and vector fields might have its own interest, we decided to reduce ourselves to use the Gauss--Green formula in \cref{thm:Gauss-Green}.
\fr\end{remark}

We can now prove the first of our main results anticipated in the Introduction.

\begin{proof}[Proof of \cref{thm:MainEst}]
Let us first prove that 
\begin{equation}\label{eqn:ToObtain1}
rP(E\setminus B_r(\bar x))\leq C_{K,N,R}\mm(E\cap B_r(\bar x))+rP(E), \quad\text{ for a.e.\,}r\in(0,R).
\end{equation}
First, we know that for a.e.\! $r\in (0,R)$ we have that 
\begin{equation}\label{eqn:AUX1}
P(E,\partial B_r(\bar x))=0,
\end{equation}
and 
\begin{equation}\label{eqn:AUX2}
P(E,\X\setminus B_r(\bar x))+P(B_r(\bar x),E^{(1)})=P(E\setminus B_r(\bar x)),
\end{equation}
where the last equality follows from the analogous version of \cref{cor:per_inters_ball} for the complement of balls, cf. \cref{lem:per_inters_general}.
Hence, by using the same notation as in \cref{thm:key_estimate}, we have that \eqref{eq:key_estimate_claim} holds. Adding $rP(E,\X\setminus B_r(\bar x))$ to both sides of \eqref{eq:key_estimate_claim}, and by using \eqref{eqn:AUX1}, and \eqref{eqn:AUX2}, we get that \eqref{eqn:ToObtain1} holds.

Now we need to prove that the conclusion in \eqref{eqn:ToObtain1} above holds for \textit{every} $r\in (0,R)$. Let us fix $r\in (0,R)$ and take $r_j\to r$, where $r_j$ is a radius for which the inequality \eqref{eqn:ToObtain1} holds. Since $\nchi_{E\cap (\X\setminus B_{r_j}(\bar x))}\to\nchi_{E\cap (\X\setminus B_{r}(\bar x))}$ in $L^1(\X,\meas)$, we can use the lower semicontinuity of the perimeter, cf. \cref{rem:SemicontPerimeter}, to deduce that 
$$
rP(E\setminus B_r(\bar x))\leq \liminf_{j\to +\infty}r_jP(E\setminus B_{r_j}(\bar x)) \leq C_{K,N,R}\mm(E\cap B_r(\bar x))+rP(E),
$$
which finally gives \eqref{eq:main_claim2}.

In order to prove \eqref{eq:main_claim2NEW} it suffices to apply \eqref{eq:main_claim2} to $\X\setminus E$ and to use that $P(F)=P(\X\setminus F)$ for every finite perimeter set $F\subset \X$.
\end{proof}

In case a set of finite perimeter $E$ has interior or exterior points, it is possible to apply \cref{thm:MainEst} in order to get useful localized deformations of $E$ prescribing the variation of the measure of the set and such that the variation of the perimeter is linear in the variation of the measure. Before stating such result, let us recall a basic fact at the level of geodesic metric spaces.

\begin{remark}\label{rem:LispchitzPath}
If $(\X,\dist)$ is a proper geodesic metric space and $U\subset \X$ is open, bounded, and connected, then it can be proved that for any $x,y \in U$ there exists a Lipschitz curve $\gamma:[0,L]\to U$ with finite length $L$ such that $\gamma(0)=x$ and $\gamma(1)=y$.
\fr\end{remark}

\begin{theorem}[Measure prescribing localized deformations]\label{thm:VariazioniMaggi}
    Let \((\X,\sfd,\mm)\) be an \({\sf RCD}(K,N)\) space with \(N<\infty\).
	Let \(E\subset\X\) be a set of locally finite perimeter and let $A\subset \X$ be a connected open set. Assume that $P(E,A)>0$.
	\begin{itemize}
	    \item[i)] If $E\cap A$ has interior points, then there exist a ball $B\Subset A$, $\eta_1=\eta_1(E,A)>0$ and $C_1(E,A)>0$ such that for every $\eta \in [0,\eta_1)$ there is a set $F\supset E$ such that
	    \[
	    E\Delta F \subset B,
	    \qquad
	    \meas(F\cap B)=\meas(E\cap B)+\eta,
	    \qquad
	    P(F,A)\le C_1(E,A)\eta+ P(E,A). 
	    \]
	    
	    \item[ii)] If $E\cap A$ has exterior points, then there exist a ball $B\Subset A$, $\eta_2=\eta_2(E,A)>0$, and $C_2(E,A)>0$ such that for every $\eta \in [0,\eta_2)$ there is a set $F\subset E$ such that
	    \[
	    E\Delta F \subset B,
	    \qquad
	    \meas(F\cap B)=\meas(E\cap B)-\eta,
	    \qquad
	    P(F,A)\le C_2(E,A)\eta+ P(E,A).
	    \]
	\end{itemize}
\end{theorem}

\begin{proof}
It is readily seen that i) implies ii) passing to the complement, hence it suffices to prove i). As balls are path-connected, $A$ is path-connected as well. Since $P(E,A)>0$ and $E\cap A$ has interior points, there are $\rho>0$ and $x,y \in A$ such that $B_\rho(x)\Subset A$, $\meas(B_\rho(x) \setminus E)=0$, $y \in E^{(0)}\cap A$, and $\dist(x,y)>\rho$. By \cref{rem:LispchitzPath} there is a Lipschitz curve $\gamma:[0,L]\to A$ with finite length $L$ such that $\gamma(0)=x$ and $\gamma(L)=y$, and we parametrize $\gamma$ so that its metric derivative $|\gamma'|=1$ almost everywhere. Let
\[
r\eqdef\frac12 \min\left\{\rho,\inf\{ \dist(\gamma(t),\partial A)\st t \in [0,L]\}\right\}>0.
\]
Consider the family of balls
\[
\left\{B_r(\gamma(t_k)) \st
t_k = \min\left\{ k\frac{r}{2}, L\right\}, \quad 
k\in \N
\right\}.
\]
By construction $\dist(\gamma(t_{k+1}),\gamma(t_k))\le \int_{t_k}^{t_{k+1}} |\gamma'| \le  r/2$, that implies $\gamma(t_{k+1}) \in B_r(\gamma(t_k))$. Hence, if $k$ satisfies $\meas (B_r(\gamma(t_k))\setminus E) =0$, then $\gamma(t_{k+1})$ is an interior point. Since $y \in E^{(0)}$ there exists a first index $k_0$ such that
\[
\meas(B_r(\gamma(t_{k_0}))\setminus E) >0.
\]
The point $\gamma(t_{k_0})$ is an interior point, indeed, since $k_0$ is the first index satisfying $\meas(B_r(\gamma(t_{k_0}))\setminus E) >0$, then $\meas(B_r(\gamma(t_{k_0-1}))\setminus E) =0$, and thus $\gamma(t_{k_0})$ is an interior point. Hence $0<t_{k_0}<L$. Relabelling $z\eqdef \gamma(t_{k_0})$, we found a point $z \in A$ and a radius $r>0$ such that
\begin{equation*}
    \meas(B_{\frac{r}{2}}(z)\setminus E)=0,
    \qquad
    \meas(B_r(z)\setminus E)>0,
    \qquad
    B_r(z)\Subset A.
\end{equation*}
We thus define $B:=B_r(z)$, $\eta_1\eqdef \meas(B_r(z)\setminus E)$ and $C_1(E,A)\eqdef 2 C_{K,N,r} / r$, where $C_{N,K,r}$ is given by \cref{thm:MainEst}. By continuity, for any $\eta \in [0,\eta_1)$ there is $t\in[r/2,r)$ such that $\meas(B_t(z) \setminus E) =\eta$. For such $\eta,t$ we take $F= E\cup B_t(z)$. Hence $E\Delta F \subset B_t(z)\Subset A$, and by \cref{thm:MainEst} applied locally in $A$ we estimate
\[
P(F,A) \le \frac{C_{K,N,r}}{t}\eta + P(E,A) \le C_1(E,A)\eta + P(E,A).
\]
\end{proof}

\section{Volume constrained minimizers of quasi-perimeters}\label{sec:VolumeConstrained}

This section is devoted to the proof of the results stated in \cref{thm:MainIntro2}, \cref{cor:FINALEIntro}, and \cref{cor:Manifold}. The proof is quite long and involved, and the main effort is needed in proving that a volume constrained minimizer $E$ as in \cref{thm:MainIntro2} or \cref{cor:FINALEIntro} has interior and exterior points, i.e., there exist two nonempty balls $B_1$, $B_2$ such that $B_1\setminus E$ is negligible and $B_2\setminus E$ has full measure. This is achieved in \cref{thm:InteriorAndExteriorPoints} and its proof follows the ideas in \cite{Xia05}, which is, in turn, based on the strategy of \cite{GonzalezMassariTamanini}. Once \cref{thm:InteriorAndExteriorPoints} is proved, the rest of the claims in \cref{thm:MainIntro2} and \cref{cor:FINALEIntro} follows by adapting classical ideas from the regularity theory of isoperimetric sets. In particular, it is proved that a volume constrained minimizer, having interior and exterior points, enjoys minimality properties \emph{without a volume constraint}, namely the properties of being $(\Lambda,r_0,\sigma)$-perimeter minimizer or $(K,r_0)$-quasi minimal (cf. \cref{def:QuasiMinimi} and \cref{thm:FromMinToQuasiMin}). Then density estimates at points in the topological boundary of a volume constrained minimizer $E$ hold, see \cref{prop:FromQuasiMinToOpen}, implying that $E^{(1)}$ is open and that topological boundary and essential boundary coincide. Finally, since the presence of interior and exterior points allows the use of \cref{thm:VariazioniMaggi}, we deduce the boundedness of volume constrained minimizers (cf. \cref{thm:Boundedness}). Finally \cref{cor:Manifold} is proved.

\subsection{Isoperimetric and comparison inequalities}\label{sec:IsoperimetricAndComparison}

In this section we prove some inequalities playing a crucial role in the proof of \cref{thm:MainIntro2}. We start by proving an isoperimetric inequality having a constant which differs from the Euclidean one by a small error. Such an inequality, which is a consequence of the results in \cite{CavallettiMondinoAlmostEuclidean}, holds at sufficiently small scales in a neighborhood of points of density $1$ in $\RCD(K,N)$ spaces $(\X,\dist,\haus^N)$. We recall that the density function $\vartheta[X,\dist,\haus^N]$ is defined in \cref{rem:PerimeterMMS2}. For a more general version of the local isoperimetric inequality in the setting of $\CD$ spaces, we refer the reader to \cite[Theorem 3.9]{NobiliViolo21}, where it is exploited to show Sobolev inequalities with optimal constants.

\begin{proposition}[Almost Euclidean isoperimetric inequality]\label{prop:AlmostEuclIsop}
For any $N\in \N$ with $N\ge2$ there exists $\bar\eps\in(0,1)$ such that the following holds. If $(\X,\dist,\haus^N)$ is an $\RCD(K,N)$ space for $K\in \R$, then for every $o\in \X$ such that $\vartheta[X,\dist,\haus^N](o)=1$ and for any $\eps\in(0,\bar\eps)$ there is $\rho(o,\eps)>0$, $R=R(o,\eps)>0$, and $\bar{C}:=\bar C(o,\eps)>0$, such that
\[
P(E) \ge C_N\left(1-\eps - \bar{C}\rho \right)(\haus^N(E))^{\frac{N-1}{N}},
\]
for any $E$ contained in a ball $B_\rho(x)$ with $ x \in B_R(o)$ and $\rho<\rho(o,\eps)$, where $C_N=N\omega_N^{1/N}$ is the $N$-dimensional Euclidean isoperimetric constant.
\end{proposition}

\begin{proof}
For $k\le0$ and $N\in \N$ with $N\ge2$, by \cite[Corollary 1.6]{CavallettiMondinoAlmostEuclidean} we know that there exist $\bar{C}_{k,N},$ $ \bar{\eta}_{k,N},$ $ \bar{\delta}_{k,N},$ $ \bar{r}_{k,N}>0$ such that if $(Y,\dist_Y, \haus_Y^N)$ is an $\RCD(k,N)$ space, if $\haus^N_Y(B_{\bar{r}_{k,N}}(x))\ge (1-\eta) v(N,k/(N-1),\bar{r}_{k,N})$ for some $\eta \in (0,\bar{\eta}_{k,N})$ then
\[
P_Y(E) \ge C_N(1-\bar{C}_{k,N}(\delta+\eta))[\haus_Y^N(E)]^{1-1/N},
\]
for any $E\subset B_\delta(x)$ with $\delta\in(0,\bar{\delta}_{k,N})$. Indeed the assumption on the density at $x$ contained in \cite[Corollary 1.6]{CavallettiMondinoAlmostEuclidean} is automatically satisfied since the reference measure is $\haus^N_Y$ and the density is lower semicontinuous as a function on $Y$. Moreover, in the above notation, $\bar{r}_{k,N}$  is defined by the identity $v(N,k/(N-1),\bar{r}_{k,N}) = 1$ and then $\lim_{k\to0}\bar{r}_{k,N} = \bar{r}_{0,N}>0$. Also $\bar{C}_{k,N}\ge1$ without loss of generality. By inspection of the proof of the main result \cite[Theorem 1.4]{CavallettiMondinoAlmostEuclidean}, assuming the reference measure is $\haus^N_Y$, it can be easily checked that the parameter $\eta$ in the volume assumption
\[
\haus^N_Y(B_{\bar{r}_{k,N}}(x))\ge (1-\eta) v(N,k/(N-1),\bar{r}_{k,N}) = 1-\eta
\]
can be chosen independently of $k,N$, i.e., one can take $\bar{\eta}_{k,N}=\bar\eta>0$. 
Moreover $\bar{C}_{k,N}$ and $\bar{\delta}_{k,N}$, which are defined in the chain of inequalities \cite[(4.6)-(4.12)]{CavallettiMondinoAlmostEuclidean}, only depend on the $C^1$ norm of the function $r\mapsto v(N,k/(N-1),r)$ in a neighborhood of $\bar{r}_{k,N}$ and, in the notation of \cite{CavallettiMondinoAlmostEuclidean}, on the $C^1$ norm of $(k,D)\mapsto \mathcal{I}_{k,N,D}(v)$ for small volumes $v\le v_{k,N}$ locally independent of $k$. The latter function depends smoothly on its variables by its very definition, see \cite[Section 2.2]{CavallettiMondinoAlmostEuclidean} and the discussion therein. Observing that the functions $r\mapsto v(N,k/(N-1),r)$ also depend smoothly on $k$, we can state the following consequence.

Letting $\bar{r}_{k,N}$ such that $v(N,k/(N-1),\bar{r}_{k,N}) = 1$, for $N\in \N$ with $N\ge2$, there is $k_0<0$, $\bar{C}'\ge1,$ $ \bar{\delta}>0$ such that if $(Y,\dist_Y, \haus_Y^N)$ is an $\RCD(k,N)$ space with $k\in[k_0,0]$, if $\haus^N_Y(B_{\bar{r}_{k,N}}(x))\ge 1-\eta$ for some $\eta \in (0,1/3)$ then
\[
P_Y(E) \ge C_N(1-\bar{C}'(\delta+\eta))[\haus_Y^N(E)]^{1-1/N},
\]
for any $E\subset B_\delta(x)$ with $\delta\in(0,\bar{\delta})$.

We can now prove that the statement holds for any $o\in \X$ such that $\lim_{r\to0} \haus^N(B_r(o))/\omega_Nr^N = 1$. 
Without loss of generality we can assume that $K\le0$. Fix such $o$ and let $\eps\in(0,\bar\eta)$. There exists $R(\eps,\bar{C}')>0 $ such that
\begin{equation}\label{eq:StimaPerCM}
    \haus^N(B_r(o)) \ge \left(1-\frac{\eps}{4\bar{C}'}\right) v(N,k/(N-1),r),
\end{equation}
for any $r\in(0,R(\eps,\bar{C}')]$ and $k\in[k_0,0]$.

Since $K\le0$, the space $(\X,\dist'\eqdef \bar{R}^{-1}\dist, \haus_{\dist'}^N)$ is $\RCD(\bar{R}^2 K,N)$ for any $\bar{R}>0$. Using \eqref{eq:StimaPerCM} we deduce that if $\bar{R}$ is small enough, so that $k\eqdef\bar{R}^2K\in(k_0,0]$ and $\bar{R}\bar{r}_{k,N}< R(\eps,\bar{C}')$, we have
\begin{equation*}
\begin{split}
\haus_{\dist'}^N(B^{\dist'}_{\bar{r}_{k,N}}(o))  
&= \frac{1}{\bar{R}^N} \haus^N \left( B_{\bar{R}\bar{r}_{k,N}}(o) \right) 
\ge  \frac{1}{\bar{R}^N}\left(1-\frac{\eps}{4\bar{C}'}\right) v(N,k/(N-1),\bar{R}\bar{r}_{k,N}) \\
&\ge \left(1-\frac{\eps}{2\bar{C}'}\right) v(N,k/(N-1),\bar{r}_{k,N})=1-\frac{\eps}{2\bar{C}'},
\end{split}    
\end{equation*}
where the last inequality follows by taking $\bar{R}$ small enough, since
\[
\lim_{\bar{R}\to0^+} \frac{ v(N,k/(N-1),\bar{R}\bar{r}_{k,N})}{\bar{R}^N  v(N,k/(N-1),\bar{r}_{k,N})} =1.
\]
By continuity of the map $x\mapsto \haus_{\dist'}^N(B^{\dist'}_{\bar{r}_{k,N}}(x))$, we conclude that there exists $R'>0$ such that
\begin{equation}\label{eq:Stima1}
    \haus_{\dist'}^N(B^{\dist'}_{\bar{r}_{k,N}}(x)) \ge 1-\frac{\eps}{\bar{C}'},
\end{equation}
for any $x$ with $\dist'(x,o)<R'$. It follows that there exist $\bar{C}'\ge1,$ $ \bar{\delta}>0$ such that
\[
P_{\dist'}(E) \ge C_N\left(1-\bar{C}'\left(\delta+\frac{\eps}{\bar{C}'}\right)\right)[\haus_{\dist'}^N(E)]^{1-1/N}
=C_N\left(1-\eps-\bar{C}'\delta\right)[\haus_{\dist'}^N(E)]^{1-1/N}
,
\]
for any $E\subset B^{\dist'}_\delta(x)$ with $\delta\in(0,\bar{\delta})$, for any $x$ with $\dist'(x,o)<R'$.

Since $\dist'=\bar{R}^{-1}\dist$, by scaling back, this means that
\[
P(E) \ge C_N(1-\eps-\bar{C}'\bar R^{-1}\delta)[\haus^N(E)]^{1-1/N},
\]
for any $E\subset B_\delta(x)$ with $\delta\in(0,\bar{R}\bar{\delta})$, for any $x$ with $\dist(x,o)<\bar{R}R'$. Defining $\bar C\eqdef \bar{C}'\bar R^{-1}$ the statement follows.
\end{proof}

The moral behind the previous Proposition is that when the space is locally sufficiently close to the Euclidean model, then an isoperimetric inequality with a constant close to the Euclidean one holds at small scales. The latter is precisely the content of the result in \cite[Theorem 1.3]{CavallettiMondinoAlmostEuclidean}, which is only stated for smooth Riemannian manifolds. With the help of the study later developed in \cite{DePhilippisGigli18} we are able to give an analogue of \cite[Theorem 1.3]{CavallettiMondinoAlmostEuclidean} in the setting of $\RCD(K,N)$ spaces equipped with the Hausdorff measure $\haus^N$. 
We will not use the forthcoming result in the paper, but we register here for future references. We do not claim any originality in the proof of the forthcoming result, since it was essentially already contained in \cite[Theorem 1.3]{CavallettiMondinoAlmostEuclidean}. We denote by $\dist_{\rm GH}$ the Gromov--Hausdorff distance between compact metric spaces.

\begin{corollary}\label{cor:GHcloseImpliesEuclideanIsop}
Let us fix $N\in \N$ with $N\ge2$. Then there exist $\varepsilon_N, \alpha_N >0$ such that the following holds. For every $\varepsilon<\varepsilon_N$ there exists $\delta:=\delta(\varepsilon,N)>0$ such that if $(\X,\dist,\haus^N)$ is an $\RCD(-\delta,N)$ space and 
\[
\dist_{\rm GH}\left( \overline{B}^{\X}_{1}(o) ,  \overline{B}^{\R^N}_{1}(0)  \right) \le \delta,
\]
for some $o \in \X$, then
\[
P(E) \ge C_N\left(1 -  \eps \right)(\haus^N(E))^{\frac{N-1}{N}},
\]
for any $E\subset B_{\varepsilon\alpha_N}(o)$, where $C_N=N\omega_N^{1/N}$ is the $N$-dimensional Euclidean isoperimetric constant.
\end{corollary}

\begin{proof}
Let $\bar{r}_{k,N}$ be the unique radius satisfying the identity $v(N,k/(N-1),\bar{r}_{k,N}) = 1$, for any $k \in \R$ and $N\ge 2$. In the proof of \cref{prop:AlmostEuclIsop} we obtained the following. For $N\in \N$ with $N\ge2$, there is $k_0<0$, $\bar{C}'\ge1,$ $ \bar{\delta}>0$ such that if $(Y,\dist_Y, \haus_Y^N)$ is an $\RCD(k,N)$ space with $k\in[k_0,0]$, if $\haus^N_Y(B_{\bar{r}_{k,N}}(x))\ge 1-\eta$ for some $\eta \in (0,1/3)$ then
 \begin{equation}\label{eq:AlmostEuclInProof}
 P_Y(E) \ge C_N(1-\bar{C}'(\delta'+\eta))[\haus_Y^N(E)]^{1-1/N},
 \end{equation}
for any $E\subset B_{\delta'}(x)$ with $\delta'\in(0,\bar{\delta})$.

Let us fix $(\X,\dist,\mathcal{H}^N)$ as in the hypotheses, where $\delta$ has to be chosen. For any $k\in [k_0,0]$ we have
\[
\dist_{\rm GH} \left( \overline{B}^{\X'}_{\bar{r}_{k,N}}(o) ,  \overline{B}^{\R^N}_{\bar{r}_{k,N}}(0) \right) \le \delta \bar{r}_{k,N},
\]
where $\X'$ denotes the $\RCD(-\delta/\bar{r}_{k,N}^2,N)$ space $(\X,\dist',\haus^N_{\dist'})$ and $\dist'\eqdef \bar{r}_{k,N}\dist$. Up to taking $k$ sufficiently close to $0$ from the left, since $\bar{r}_{k,N} \to \bar {r}_{0,N}$ as $k\to 0$, there exists $\delta$ small enough such that $\haus^N_{\dist'}(\overline{B}^{\X'}_{\bar{r}_{k,N}}(o)) \ge 1-\tfrac{\eps}{2\bar{C}'}$ by \cite[Theorem 1.3]{DePhilippisGigli18} and $-\delta/\bar{r}_{k,N}^2\ge k_0$. For $\eps_N$ small, i.e., $\eps_N/(2\bar{C}') <\bar\delta$, we can apply \eqref{eq:AlmostEuclInProof} on any set $E \subset B^{\X'}_{\eps/(2\bar{C}')} (o) = B^{\X}_{\eps/(2\bar{C}'\bar{r}_{k,N})} (o)$, and the claim follows.
\end{proof}

The next proposition states a comparison between the perimeter of balls of equal volume in an $\RCD$ space and in the corresponding reference model. The result was firstly proved in \cite{MorganJohnson00} in the smooth setting, but the proof can be easily generalized to the nonsmooth realm of $\RCD$ spaces.

\begin{proposition}\label{prop:MorganJohnson}
Let $(\X,\dist,\haus^N)$ be a $\RCD(K,N)$ space for $K\in\R$ and $N\in \N$ with $N\ge2$. Then for any ball $B\subset \X$ it holds that
\[
P(B)\le P_K(\mathbb B^K(\haus^N(B)) ),
\]
where $P_K$ is the perimeter functional on the model $\mathbb M^N_{K/(N-1)}$ and $\mathbb B^K(v)$ is a ball in $\mathbb M^N_{K/(N-1)}$ of volume $v$ for any $v\ge0$.
\end{proposition}

\begin{proof}
The proof follows by the very same arguments of \cite[Theorem 3.5]{MorganJohnson00}. Indeed, the proof therein only uses Bishop--Gromov monotonicity of the ratios
\[
\frac{\haus^N(B_r(x))}{v(N,K/(N-1),r)},
\qquad
\frac{P(B_r(x))}{s(N,K/(N-1),r)},
\]
the inequality
\[
\frac{P(B_r(x))}{s(N,K/(N-1),r)} \le \frac{\haus^N(B_r(x))}{v(N,K/(N-1),r)},
\]
which holds by Bishop--Gromov monotonicity and the coarea formula, and finally the fact that
\[
\frac{\haus^N(B_r(x))}{v(N,K/(N-1),r)} \le 1,
\]
which holds since
\[
\frac{\haus^N(B_r(x))}{v(N,K/(N-1),r)} \le \lim_{\rho\to0^+} \frac{\haus^N(B_\rho(x))}{v(N,K/(N-1),\rho)} = \vartheta[\X,\dist,\haus^N](x) \le1 ,
\]
where the last inequality follows from the fact that the reference measure is $\haus^N$ and the density function $x \mapsto \vartheta[\X,\dist,\haus^N](x)$ is lower semicontinuous.
\end{proof}

\begin{corollary}\label{cor:MJ-PalleVolumiPiccoli}
Let $(\X,\dist,\haus^N)$ be an $\RCD(K,N)$ space for $K\in\R$ and $N\in \N$ with $N\ge2$. Then for every $\eps_1>0$ there is $V_{\eps_1}>0$ such that if $B\subset \X$ is a ball with $\haus^N(B)\le V_{\eps_1}$, then
\[
P(B) \le  (C_N+\eps_1) (\haus^N(B))^{\frac{N-1}{N}},
\]
where $C_N=N\omega_N^{1/N}$.
\end{corollary}

\begin{proof}
Given $\eps_1>0$ there is $V_{\eps_1}>0$ such that if $\mathbb B^K \subset \mathbb M^N_{K/(N-1)}$ is a ball in the model $\mathbb M^N_{K/(N-1)}$ with volume $v\le V_{\eps_1}$, then
\begin{equation}\label{eq:Stima2}
P_K (\mathbb B^K) \le (1+ \eps_1/C_N) P_0(\mathbb B^0 ( v ) ) = (1+ \eps_1/C_N) C_N v^{1-1/N},
\end{equation}
where $P_K, P_0$ are the perimeter functionals on $\mathbb M^N_{K/(N-1)} , \R^n$ respectively, and $\mathbb B^0(v)$ is a ball in $\R^n$ having volume $v$.

Given $B\subset \X$ with volume $v=\haus^N(B)\le V_{\eps_1}$, the claim immediately follows from \cref{prop:MorganJohnson} and \eqref{eq:Stima2}.
\end{proof}

We recall here the statement of the relative isoperimetric inequality in the context of $\RCD$ spaces. The next result is a well-known consequence of \cite[Theorem 9.7]{HajlaszKoskela}, taking into account that a $\RCD(K,N)$ space is a geodesic space in which a weak local $(1,1)$-Poincar\'{e} inequality holds, see \cite{Rajala12}, and the uniform locally doubling property holds, cf. \cref{rem:PerimeterMMS2}. See also \cite[Theorem 4.5]{Miranda03} and \cite[Remark 3.4]{AmbrosioAhlfors}. We omit the proof of the next result since it is rather classical.

\begin{proposition}[Relative isoperimetric inequality]\label{prop:RelativeIsoperimetricInequality}
Let $(\X,\dist,\meas)$ be an $\RCD(K,N)$ space. Let $\Omega\subset \X$ be a bounded set and let $R>0$. Then there is a constant $C_{\rm RI}= C_{\rm RI}(N,K,\Omega,R)$ such that
\begin{equation*}
    \min\left\{ \meas(B_r(x)\cap E)^{1-\frac1N} , \meas(B_r(x)\setminus E)^{1-\frac1N}  \right\} \le C_{\rm RI} P(E, B_r(x)),
\end{equation*}
for any $r\le R$ and $x \in \Omega$, for any set of locally finite perimeter $E$ in $\X$. 

Moreover, if there exists $v_0$ such that $\meas(B_1(x))\geq v_0$ for every $x\in \X$, then the constant $C_{\mathrm{RI}}$ only depends on $N,K,R,v_0$.
\end{proposition}

\subsection{Preparatory estimates}\label{sec:Preparatory}

We start by proving some technical results holding for arbitrary sets of finite perimeter in $\RCD(K,N)$ spaces $(\X,\dist,\meas)$. We begin with a well-known classical covering lemma, including a short proof for the convenience of the reader. We stress that the hypotheses of \cref{lem:CoveringLemma} are satisfied, e.g., whenever $(\X,\dist,\meas)$ is a $\CD(K,N)$ space with $N<+\infty$ by virtue of \cref{rem:PerimeterMMS2}.

\begin{remark}\label{rem:BGBella}
If $(\X,\dist,\meas)$ is a metric measure space where $\meas$ is uniformly locally doubling, then for every $R>0$ there exist constants $C_1,C_2$ such that 
$$
\frac{\meas(B_{r_2}(x))}{\meas(B_{r_1}(x))}\leq C_1\left(\frac{r_2}{r_1}\right)^{C_2},
$$
for every $x\in \X$ and every $r_1\leq r_2\leq R$. We can take $C_2:=\log_2 C$, where $C$ is the doubling constant associated to $R$. This is a pretty standard observation coming from the iteration of the uniformly local doubling property, cf. \cite[Lemma 14.6]{HajlaszKoskela}.
\fr\end{remark}

\begin{lemma}[Covering Lemma]\label{lem:CoveringLemma}
Let $(\X,\dist,\meas)$ be a metric measure space where $\meas$ is uniformly locally doubling, let $\overline R>0$, and let $\rho\leq \overline R$.
\begin{itemize}
    \item[(i)] Let $C_1,C_2$ the constants associated to the radius $\overline R$ as in \cref{rem:BGBella}. Hence, for any $\alpha>0$, $z \in \X$, $\alpha \rho \le R \le \overline R$, it holds
    \begin{equation*}
        \sharp \mathscr{F}  \le C_1\left(\frac{2R}{\alpha\rho}\right)^{C_2},
    \end{equation*}
    for any family $\mathscr{F}$ of disjoint balls of radius $\alpha\rho$ contained in $B_R(z)$.
    
    \item[(ii)] If $\Omega \subset \X$ is open and $D\subset \Omega$ is dense in $\Omega$, there exist countably many points $\{x_i\}_{i\in\N} \subset D$ such that
    \begin{equation}\label{eq:Covering}
    \begin{split}
        B_{\frac\rho2}(x_i) \cap B_{\frac\rho2}(x_j) = \emptyset& \qquad \forall\,i\neq j, \\
        \bigcup_i B_\rho(x_i) \supset D& , \\
        \bigcup_i B_{\lambda\rho}(x_i) \supset \Omega&
        \qquad \forall\,\lambda>1.
    \end{split}
    \end{equation}
    Moreover, for any $\beta\leq \overline R$ and $z \in \bigcup_i B_{\beta\rho}(x_i)$, it holds
    \begin{equation}\label{eq:StimaCoveringOverlap}
        \sharp \left\{ \text{balls $B_{\beta\rho}(x_i) \st z \in B_{\beta\rho}(x_i)$} \right\} \le \max\left\{ 1,C_1(8\beta)^{C_2} \right\},
    \end{equation}
    where $C_1,C_2$ are the constants as in \cref{rem:BGBella} associated to $2\overline R^2$. 
\end{itemize}
\end{lemma}
\begin{proof}
If $B_{\alpha\rho}(y_1),\ldots,B_{\alpha\rho}(y_\ell)$ are disjoint balls in $B_R(z)$ as in item (i) and, say, $\meas(B_{\alpha\rho}(y_1)) = \min_{i=1,\ldots,\ell} \meas(B_{\alpha\rho}(y_i))$, then
\[
\ell \le \frac{\meas(B_R(z))}{\meas(B_{\alpha\rho}(y_1))} \le \frac{\meas(B_{2R}(y_1))}{\meas(B_{\alpha\rho}(y_1))} \le C_1\left(\frac{2R}{\alpha\rho}\right)^{C_2},
\]
and item (i) follows.

Now let $D$ and $\Omega$ be as in item (ii). The existence of the family $\{x_i\}_{i\in \N}$ satisfying \eqref{eq:Covering} easily follows by applying Zorn Lemma on $D$ as in \cite[Lemma 4.4]{AFP21}. Observe that the third line in \eqref{eq:Covering} follows from the second one and the density of $D$ in $\Omega$.

In order to prove \eqref{eq:StimaCoveringOverlap} for given $\beta$ and $z$, observe that if $\beta\le1/2$ then the left-hand side in \eqref{eq:StimaCoveringOverlap} is at most $1$, so we can assume $\beta>1/2$. So, if $z \in B_{\beta\rho}(x_i)$ then $B_{\rho/2}(x_i)\subset B_{\beta\rho}(x_i) \subset B_{2\beta\rho}(z)$. Hence applying item (i) on the family $\mathscr{F}=\{ \text{balls $B_{\rho/2}(x_i) \st z \in B_{\beta\rho}(x_i)$} \}$ with respect to $R=2\beta\rho$ and $\alpha=1/2$, \eqref{eq:StimaCoveringOverlap} follows.
\end{proof}

The next lemma states that in a neighborhood of almost every density $0$ point of a set of finite perimeter $E$, the measure of $E$ inside small balls decay as a power strictly greater than $N$ of the radius of such balls.

\begin{lemma}[Volume decay estimate]\label{lem:VolumeDecayEstimate}
Let $(\X,\dist,\haus^N)$ be an $\RCD(K,N)$ space for $K\in \R$ and $N\in \N$ with $N\ge2$. Let $E$ be a set of locally finite perimeter.
Then for any  $o \in E^{(0)} \cap \{\vartheta[X,\dist,\haus^N]=1 \}$ there is $R_o\in(0,1)$, $C_{E}>0$ such that
\[
f(r)\eqdef \inf\left\{\haus^N(B_r(x)\cap E) \st x \in B_{R_o}(o) \right\} \le C_{E} r^{\frac{N^2}{N-1}}
\qquad
\forall\,r\in[0,R_o).
\]
\end{lemma}

\begin{proof}
Since $o \in E^{(0)} \cap \{\vartheta[X,\dist,\haus^N]=1 \}$, there exists $R_o\in(0,1)$ such that
\begin{equation}\label{eq:LambdaDensitaZero}
\haus^N(B_R(o) \cap E) \le \Lambda R^N 
\qquad
\forall R \in (0,2R_o],
\end{equation}
where $\Lambda\eqdef C_o C_{o,3}/(2C)$ is a constant depending on $N,K,\haus^N(B_1(o))$ defined in terms of constants appearing below.

For any $r<R_o$ we can apply item (ii) of \cref{lem:CoveringLemma} with $\rho=r$ and $\Omega=D=B_{R_o}(o)$, thus getting points $\{x_i\}\subset B_{R_o}(o)$. By Bishop--Gromov monotonicity, as $\haus^N$ is the reference measure, we have that
\begin{equation*}
    C_{o,1}(R_o)^N \le \haus^N(B_{R_o}(o)) \le \sum_i \haus^N(B_r(x_i)) \le \sharp\{x_i\} v(N,K/(N-1),r) \le \sharp\{x_i\} C_{o,2}r^N,
\end{equation*}
so that
\begin{equation}\label{eq:CardinalitaPunti}
    \sharp\{x_i\} \ge C_o \left(\frac{R_o}{r} \right)^N,
\end{equation}
and $C_o>0$ only depends on $N,K,\haus^N(B_1(o))$. Let us define
\[
\begin{split}
    \mathscr{A}\eqdef& \left\{ x_i \st \haus^N(E\cap B_r(x_i))> \frac12 \haus^N(B_r(x_i)) \right\},\\
\mathscr{B}\eqdef& \left\{ x_i \st \haus^N(E\cap B_r(x_i))\le \frac12 \haus^N(B_r(x_i)) \right\}.
\end{split}
\]
Hence
\[
\begin{split}
    C_{o,3}r^N \sharp\mathscr{A} &< \sum_{x_i \in \mathscr{A}} \haus^N(E\cap B_r(x_i))
\le \max\left\{1, C_1 8^{C_2} \right\} \haus^N(E \cap B_{R_o+r}(o)) \\
&\le C \Lambda (R_o)^N,
\end{split}
\]
where we used \eqref{eq:StimaCoveringOverlap} in the second inequality, \eqref{eq:LambdaDensitaZero} in the third inequality, $C_{o,3}$ depends on $N,K,\haus^N(B_1(o))$, and $C$ only depends on $N,K$. Therefore
\begin{equation}\label{eq:CardinalitaA}
    \sharp\mathscr{A} < \frac{C\Lambda}{C_{o,3}} \left(\frac{R_o}{r} \right)^N.
\end{equation}
Using \eqref{eq:CardinalitaPunti} and \eqref{eq:CardinalitaA}, by the choice of $\Lambda$ we obtain
\begin{equation}\label{eq:CardinalitaB}
    \sharp \mathscr{B} = \sharp\{x_i\} - \sharp\mathscr{A} > \frac{C_o}{2} \left(\frac{R_o}{r} \right)^N.
\end{equation}
If $f$ is as in the statement, using the relative isoperimetric inequality of \cref{prop:RelativeIsoperimetricInequality} we estimate
\[
\begin{split}
\sharp\mathscr{B} f(r)^{1-1/N} \le
& \sum_{x_i \in \mathscr{B}} \haus^N(E \cap B_r(x_i))^{1-1/N}
\le C_{\rm RI} \sum_{x_i \in \mathscr{B}} P(E, B_{r}(x_i)) \\
& \le C_{\rm RI}  \max\left\{ 1,C_1 8^{C_2} \right\} P(E, B_{R_o+r}(o)) \\
&\le C_{o,4} P(E, B_2(o)),
\end{split}
\]
where we used \eqref{eq:StimaCoveringOverlap} with $\beta=1$ in the third inequality and $C_{o,4}$ only depends on $N,K,o$. Rewriting the last estimate taking into account \eqref{eq:CardinalitaB} yields
\[
f(r) \le \left(\frac{C_{o,4} P(E, B_{2}(o))}{\sharp\mathscr{B}} \right)^{\frac{N}{N-1}} 
\le \left(\frac{2 C_{o,4} P(E, B_{2}(o))}{C_o(R_o)^N} \right)^{\frac{N}{N-1}} r^{\frac{N^2}{N-1}},
\]
which defines the desired constant $C_{E}$ and completes the proof.
\end{proof}

We shall need the next observation on sets of finite perimeter.

\begin{remark}\label{rem:QuasiOgniPuntoRaggio}
Let $(\X,\dist,\meas)$ be a metric measure space and let $E\subset \X$ be a set of locally finite perimeter. Then
\begin{itemize}
    \item For any $x \in \X$ there exist at most countably many radii $r>0$ such that $P(E,\partial B_r(x)) >0$.
    
    \item For almost every $r>0$ and $\meas$-a.e. $x \in \X$ it holds that $P(E,\partial B_r(x)) =0$.
\end{itemize}
The second item follows from the first one and the following observation. The map $(x,r)\mapsto P(E,\partial B_r(x)) = P(E,\overline{B}_r(x)) - P(E,B_r(x))$ is $(\meas\times \mathcal{L}^1)$-measurable, indeed $(x,r) \mapsto P(E,B_r(x))$ is lower semicontinuous and $(x,r) \mapsto P(E,\overline{B}_r(x))$ is upper semicontinuous. Let $o \in \X$ and $R,\rho>0$ be fixed. By the first item and Fubini's Theorem we deduce
\[
0 = \int_{B_R(o)} \int_0^\rho  P(E,\partial B_r(x)) \de r \de \meas(x) =
 \int_0^\rho  \int_{B_R(o)} P(E,\partial B_r(x))  \de \meas(x) \de r,
\]
and then for almost every $r \in (0,\rho)$ and $\meas$-a.e.\ $x \in B_R(o)$ it holds that $P(E,\partial B_r(x)) =0$. Sending $R,\rho\to+\infty$ the second item follows.
\fr\end{remark}

In the next lemma, we prove that whenever the measure of a set of finite perimeter $E$ inside a ball satisfies the decay in \cref{lem:VolumeDecayEstimate}, then we can decrease the radius of such ball maintaining the original decay and also additionally controlling the perimeter of the new ball inside $E^{(1)}$.

\begin{lemma}\label{lem:CoareaDiminuisceRaggio}
Let $(\X,\dist,\meas)$ be an $\RCD(K,N)$ space and let $E\subset \X$ be a set of locally finite perimeter. Let $\tau\in(0,1)$ and suppose that there exists a ball $B_r(x)$ such that
\begin{equation*}
    r< \min\left\{\frac{1}{(2C_{E})^{1-\frac1N}} \left( \frac{\tau}{2^{N+2}} \right)^{N-1} , 
    \frac{1}{(2C_{E})^{1-\frac1N}} 
    \left( \frac{N-1}{2N^2}  \tau \right)^{N-1}
    \right\},
\end{equation*}
and
\begin{equation*}
    \meas(B_r(x) \cap E) \le 2 C_{E} r^{\frac{N^2}{N-1}},
    \qquad
    P(E,\partial B_r(x)) =0,
\end{equation*}
for some constant $C_{E}>0$.

Then there exists $\rho \in [r/2,r]$ such that
\[
\begin{split}
\meas(B_\rho (x) \cap E) \le 2 C_{E} \rho^{\frac{N^2}{N-1}}, \\
P(B_\rho(x),E^{(1)}) \le \tau\, \meas(B_\rho(x) \cap E)^{1-\frac1N}, \\
P(E,\partial B_\rho(x)) = 0.
\end{split}
\]
\end{lemma}

\begin{proof}
Let $g(s)\eqdef \meas(E \cap B_s(x))$ for $s \in [0,r]$. Then, by the coarea formula, $g'(s) = P(B_s(x), E^{(1)})$ for almost every $s$. We estimate
\[
\int_0^r (g^{1/N})' = g^{1/N}(r) \le (2C_{E})^{1/N} r^{\frac{N}{N-1}} < \frac{\tau}{N} r.
\]
Hence there is a set of positive measure in $(0,r)$ such that $(g(s)^{1/N})'\le \tau/N$ at such $s$'s. Observe that for almost every such $s$ this is equivalent to
\[
P(B_s(x),E^{(1)}) = g'(s) \le \tau g(s)^{1-\frac1N} =\tau\, \meas(B_s(x) \cap E)^{1-\frac1N}.
\]
It is well defined the supremum
\[
s_0\eqdef \sup \left\{ s \in [0,r] \st 
 P(B_s(x),E^{(1)})\le \tau g(s)^{1-\frac1N} , \quad
 P(E,\partial B_s(x)) = 0
\right\}.
\]
If $s_0=r$ then the supremum is a maximum. Indeed, since $\liminf_{s\to r^-} P(B_s(x),E^{(1)}) \ge P(B_r(x),E^{(1)})$ by \cref{prop:Lahti}, if there is a sequence of competitors $s_i\nearrow s_0=r$ then $P(B_r(x),E^{(1)}) \le \tau g(r)^{1-\frac1N}$. Hence, as $ P(E,\partial B_r(x)) = 0$ by assumption, the statement follows by taking $\rho=r$.

So we can assume that $s_0\in (0,r)$.
For any $s \in (s_0,r]$ we have that either $P(B_s(x),E^{(1)}) > \tau g(s)^{1-\frac1N}$, or $P(E,\partial B_s(x)) > 0$. Hence for almost every $s \in (s_0,r]$ we have that $g'(s) > \tau g(s)^{1-\frac1N}$, that is $(g(s)^{1/N})'>\tau/N$ for almost every $s \in (s_0,r]$. Integrating from $s_0$ to $r$ gives
\begin{equation}\label{eq:StimaS0r}
    g(r)^{1/N} - g(s_0)^{1/N} \ge \frac{\tau}{N}(r-s_0),
\end{equation}
that implies
\begin{equation*}
    r-s_0 \le \frac{N}{\tau}(2C_{E})^{1/N} r^{\frac{N}{N-1}} \le \frac{N}{2^{N+2}} r \le \frac14 r,
\end{equation*}
so that $s_0 \ge \tfrac34 r$. If $s_0$ is a maximum, then the statement follows by taking $\rho=s_0$, indeed
\[
g(s_0)^{1/N} \le g(r)^{1/N} -  \frac{\tau}{N}(r-s_0) \le (2C_{E})^{1/N}r^{\frac{N}{N-1}} -  \frac{\tau}{N}(r-s_0) 
\le (2C_{E})^{1/N}s_0^{\frac{N}{N-1}},
\]
where the last inequality follows from the smallness assumption on $r$ as in \cite[Proposition 3.2]{Xia05}.

So we can assume that the supremum $s_0$ is not attained. Then there exists $\alpha\in\left(\tfrac23, 1 \right)$ such that
\begin{equation}\label{eq:defAlphaproof}
    \begin{split}
    & s_0(1-\alpha) \le r-s_0, \\
    & \bar{s}\eqdef \alpha s_0 \quad\text{ satisfies }\quad   P(B_{\bar s}(x),E^{(1)}) \le \tau g(\bar s)^{1-\frac1N} , \quad
 P(E,\partial B_{\bar s}(x)) = 0,
 \quad \bar s \ge \frac{r}{2}.
    \end{split}
\end{equation}
We claim that $\rho=\bar s$ satisfies the statement. We need to prove that $g(\bar s) \le 2C_{E} \bar{s}^{\frac{N^2}{N-1}}$. By \eqref{eq:StimaS0r} we have $g(r)^{1/N} - g(\bar s)^{1/N} \ge \frac{\tau}{N}(r-s_0)$ and then
\[
g(\bar s)^{1/N} \le (2C_{E})^{1/N} r^{\frac{N}{N-1}} - \frac{\tau}{N}(r-s_0).
\]
Hence it is enough to check that
\begin{equation}\label{eq:StimaSbarraFinale}
    (2C_{E})^{1/N} r^{\frac{N}{N-1}} - \frac{\tau}{N}(r-s_0) \le (2C_{E})^{1/N} \bar{s}^{\frac{N}{N-1}}.
\end{equation}
Indeed
\[
\begin{split}
 r^{\frac{N}{N-1}}  - \bar{s}^{\frac{N}{N-1}}
    &= \frac{N}{N-1}\left(\int_{s_0}^r t^{\frac{1}{N-1}} \de t + \int_{\bar s}^{s_0}  t^{\frac{1}{N-1}} \de t  \right)\\
    &\le \frac{N}{N-1} \left(r^{\frac{1}{N-1}}(r-s_0) + s_0^{\frac{1}{N-1}} s_0(1-\alpha) \right) \\
    &\le \frac{Nr^{\frac{1}{N-1}}}{N-1} \left((r-s_0) +  s_0(1-\alpha) \right) \\
    &\le \frac{2Nr^{\frac{1}{N-1}}}{N-1} (r-s_0) \\
    &\le \frac{1}{(2C_{E})^{1/N}}\frac{\tau}{N}(r-s_0),
\end{split}
\]
that is equivalent to \eqref{eq:StimaSbarraFinale}. In the above chain, we used $s_0\le r$ in the second inequality, \eqref{eq:defAlphaproof} in the third inequality, and the smallness assumption on $r$ in the last inequality.
\end{proof}

We are now ready for the definitions playing a key role towards the proof of \cref{thm:InteriorAndExteriorPoints}.
From now on and for the rest of the section, we fix
\begin{equation}\label{eq:DefTau}
\begin{split}
    0<\tau = \left(2-2^{1-\frac1N}\right)\gamma_N \frac{C_N}{8},\\
    \gamma_N \le 8^{-N},
\end{split}
\end{equation}
where we recall that $C_N\eqdef N(\omega_N)^{1/N}$, and
\begin{equation}\label{eq:DefEpsilon}
    \epsilon= \min \left\{ \frac{\tau}{4C_N} , \frac{\bar\epsilon}{2} \right\},
\end{equation}
for $\bar\eps$ as given by \cref{prop:AlmostEuclIsop}.

For a given $\RCD(K,N)$ space $(\X,\dist,\haus^N)$, with $N\ge2$ natural number, and a given set of locally finite perimeter $E$ with $\haus^N(E^{(0)})>0$, let $o \in E^{(0)} \cap \{\vartheta[\X,\dist,\haus^N]=1 \}$ be fixed. Then let $R_o$, $C_{E}$ be given by \cref{lem:VolumeDecayEstimate}.

By Bishop--Gromov monotonicity, and the fact that the density at any point is $\leq 1$, there exist $b>a>0$ depending on $N,K,o$ such that
\begin{equation}\label{eq:Defab}
    a r^N \le \haus^N(B_r(x)) \le b r^N,
\end{equation}
for any $x \in B_{R_o}(o)$ and any $r\le 2$.

We define $A_\tau(E)$ as the family of balls $B_r(x)\subset \X$ with $x \in B_{R_o}(o)$ such that
\begin{equation}\label{eq:DefAtau}
\begin{split}
     r<   \left( \frac{a}{2b} \right)^{N-1} \min\left\{\frac{1}{(2C_{E})^{1-\frac1N}} \left( \frac{\tau}{2^{N+2}} \right)^{N-1} , 
    \frac{1}{(2C_{E})^{1-\frac1N}} 
    \left( \frac{N-1}{2N^2}  \tau \right)^{N-1}
    \right\} ,\\
     \haus^N(B_r(x) \cap E) \le 2 C_{E} r^{\frac{N^2}{N-1}},\\
     P(B_r(x),E^{(1)}) \le \tau\, \haus^N(B_r(x) \cap E)^{1-\frac1N},\\
     P(E, \partial B_r(x)) = 0.
\end{split}
\end{equation}

\begin{remark}\label{rem:AtauNonempty}
The family $A_\tau(E)$ is nonempty and contains balls of arbitrarily small radii.
Indeed for almost every $r\in(0,R_o)$ with $r$ bounded above as in \eqref{eq:DefAtau} we know that $P(E,\partial B_r(x))=0$ for $\haus^N$-a.e. $x$ by \cref{rem:QuasiOgniPuntoRaggio}. For any such $r$, there is $y\in B_{R_o}(o)$ such that $\haus^N(B_r(y) \cap E)\le \tfrac32 C_{E} r^{\frac{N^2}{N-1}}$ by \cref{lem:VolumeDecayEstimate}. As $x \mapsto \haus^N(B_r(x) \cap E) $ is continuous, we find $x\in B_{R_o}(o)$ such that $\haus^N(B_r(x) \cap E)\le 2C_{E}r^{\frac{N^2}{N-1}}$ and $P(E,\partial B_r(x))=0$. Hence applying \cref{lem:CoareaDiminuisceRaggio}, we get that $B_\rho(x) \in A_\tau(E)$ for some $\rho \in [r/2,r]$.
\fr\end{remark}

Without loss of generality we can assume that $R_o<R(\eps,o)$, where $R(\eps,o)$ is given by \cref{prop:AlmostEuclIsop}. Hence we also have that
\begin{equation}\label{eq:QuasiIsopApplicata}
    P(F) \ge C_N\left(1-\eps - \bar{C}\rho \right)(\haus^N(F))^{\frac{N-1}{N}},
\end{equation}
for any $F\subset B_r(x)$ with $ x \in B_{R_o}(o)$ and $r<\rho(o,\eps)$, where $\bar C$ and $\rho(o,\eps)$ were given by \cref{prop:AlmostEuclIsop}.

\begin{proposition}\label{prop:KeyAlternative}In the above setting, assume that $B_r(x) \in A_\tau(E)$ with
\[
0<r\le \min\left\{\rho(o,\eps), \frac{1}{\bar{C}}\left(\frac34 - \eps \right) \right\}.
\]
Then either
\[
P(E,B_r(x)) \ge  \left(C_N\left(1-\eps - \bar{C}r \right) + 4\tau \right)\haus^N(B_r(x) \cap E)^{1-\frac1N},
\]
or there exists $B_{r'}(x')\in A_\tau(E)$ such that
\[
B_{r'}(x')\subset B_r(x),
\qquad
r' \in \left(\frac{r}{16},\frac38 r \right).
\]
In particular, one of the following alternatives occurs:
\begin{itemize}
    \item Either for any $0<\lambda\le \min\left\{\rho(o,\eps), \tfrac{1}{\bar{C}}\left(\tfrac34 - \eps \right) \right\} $ there is $B_t(z) \in A_\tau(E)$ with $t\in(0,\lambda)$ and such that
    \[
     P(E,B_t(z)) \ge  \left(C_N(1-\eps - \bar{C}t) + 4\tau \right)\haus^N(B_t(z) \cap E)^{1-\frac1N};
    \]
    \item Or there is a sequence of balls $\{B_{r_i}(x_i)\}_{i \in \N} \subset A_\tau(E)$ such that $B_{r_{i+1}}(x_{i+1}) \subset B_{r_i}(x_i)$ and $r_{i+1}\in \left(\tfrac{r_i}{16},\tfrac38 r_i \right)$.
\end{itemize}
\end{proposition}

\begin{proof}
Let $\left\{ B_{r/2}(y_i) \st i=1,\ldots,M\right\}$ be a maximal family of disjoint balls of radius $r/2$ contained in $B_r(x)$. By \eqref{eq:Defab} we have
\[
a \frac{r^N}{2^N} M \le \sum_{i=1}^M \haus^N(B_{r/2}(y_i)) \le \haus^N(B_r(x)) \le b r^N,
\]
so that $M \le 2^N b/a$. Using the coarea formula we estimate
\[
\begin{split}
    \frac{2^Nb}{a} \haus^N(B_r(x) \cap E) &\ge M \haus^N(B_r(x) \cap E) = \sum_{i=1}^M \int_0^{2r} P(B_t(y_i), E^{(1)} \cap B_r(x)) \de t \\
    &\ge k \mathcal{L}^1 \left( \left\{t \in (0,2r) \st \sum_{i=1}^M  P(B_t(y_i), E^{(1)} \cap B_r(x)) \ge k\right\} \right),
\end{split}
\]
for any $k>0$. Setting $k = \tfrac\tau2 \haus^N(B_r(x) \cap E)^{1-\frac1N}$ and using \eqref{eq:DefAtau} we get that
\[
\begin{split}
    \mathcal{L}^1 &\left( \left\{t \in (0,2r) \st \sum_{i=1}^M  P(B_t(y_i), E^{(1)} \cap B_r(x)) \ge k\right\} \right) 
    \le \frac{2^{N+1}b}{a\tau} \haus^N(B_r(x) \cap E)^{\frac1N} \\
    &\le \frac{2^{N+1}b}{a\tau} (2C_{E})^{\frac1N} r^{\frac{N}{N-1}} < \frac{r}{4},
\end{split}
\]
where in the last inequality we are using the first bound in \eqref{eq:DefAtau}. Hence there exists $s_0 \in (r/8,3r/8)$ such that
\begin{equation}\label{eq:PropertiesBallsSzero}
    \begin{split}
    \sum_{i=1}^M  P(B_{s_0}(y_i), E^{(1)} )=
    \sum_{i=1}^M  P(B_{s_0}(y_i), E^{(1)} \cap B_r(x)) < \frac\tau2 \haus^N(B_r(x) \cap E)^{1-\frac1N},  \\
    P(E, \partial B_{s_0}(y_i) ) = 0 \qquad \forall\,i=1,\ldots,M.
    \end{split}
\end{equation}
Observe that any two different balls in $\{ B_{s_0}(y_i)\}_{i=1}^M$ are located at positive distance one from the other. Denote by $V\eqdef \bigcup_{i=1}^M B_{s_0}(y_i)$.

We now distinguish a few cases, eventually yielding the alternative in the statement. Suppose first that
\[
\haus^N(E \cap B_r(x) \setminus V ) \ge \left(1 - \gamma_N^{\frac{N}{N-1}} \right) \haus^N(E\cap B_r(x)).
\]
Then
\begin{equation}\label{eq:Alternative2}
    \exists\,i \st \haus^N(E \cap B_{s_0}(y_i)) < \gamma_N^{\frac{N}{N-1}}  \haus^N(E\cap B_r(x)).
\end{equation}
Then, recalling \eqref{eq:DefAtau}, by assumption we get
\[
\begin{split}
    \haus^N(E \cap B_{s_0}(y_i)) < 2C_{E} \gamma_N^{\frac{N}{N-1}} r^{\frac{N^2}{N-1}} 
    < 2C_{E} \gamma_N^{\frac{N}{N-1}} s_0^{\frac{N^2}{N-1}} 8^{\frac{N^2}{N-1}} \le 2C_{E} s_0^{\frac{N^2}{N-1}},
\end{split}
\]
by definition of $\gamma_N$ in \eqref{eq:DefTau}. Employing \cref{lem:CoareaDiminuisceRaggio} we get that there is $r' \in [s_0/2,s_0]\subset(r/16,3r/8)$ such that $B_{r'}(y_i) \in A_\tau(E)$. Hence $B_{r'}(x')$ with $x'=y_i$ satisfies the second alternative in the statement.

Next we suppose that
\[
\haus^N(E \cap B_r(x) \setminus V ) \le  \gamma_N^{\frac{N}{N-1}}  \haus^N(E\cap B_r(x)).
\]
In this case we consider
\[
x' \in {\rm argmin}\, \left\{ \sfd(w,x) \st w \in \bigcup_{i=1}^M \partial B_{\frac{r}{2}}(y_i) \right\}.
\]
We have that $\sfd(x',x)\le r/2$, for otherwise $B_{r/2}(x)$ would be disjoint from any $B_{r/2}(y_i)$ and thus the family $\{B_{r/2}(y_i) \st i=1,\ldots,M\}\cup\{B_{r/2}(x)\}$ would contradict the maximality of the family $\{B_{r/2}(y_i) \st i=1,\ldots,M\}$. Up to renaming, we can say that $x' \in \partial B_{\frac{r}{2}}(y_1)$; thus letting $s_1\eqdef \tfrac{r}{2}- s_0 \in (r/8,3r/8)$, we see that $B_{s_1}(x')\subset B_r(x)$ and $B_{s_1}(x')$ is disjoint from any ball $B_{s_0}(y_i)$ by the triangle inequality. Hence there exists $s_1'\in(r/8,s_1]$ such that
\[
\begin{split}
&P(E,\partial B_{s_1'}(x')) = 0, \\
&\haus^N(E \cap B_{s_1'}(x')) \le \haus^N(E \cap B_r(x) \setminus V ) \le \gamma_N^{\frac{N}{N-1}}  \haus^N(E\cap B_r(x)),    
\end{split}
\]
and thus we are essentially back to the case \eqref{eq:Alternative2} with $s_1'$, $x'$ in place of $s_0$, $y_i$. Therefore, arguing as before, one finds $r' \in [s_1'/2,s_1']\subset(r/16,3r/8)$ such that $B_{r'}(x')$ satisfies the second alternative in the statement.

Finally, it remains to deal with the case
\[
\gamma_N^{\frac{N}{N-1}}  \haus^N(E\cap B_r(x)) <\haus^N(E \cap B_r(x) \setminus V ) < \left(1 - \gamma_N^{\frac{N}{N-1}} \right) \haus^N(E\cap B_r(x)).
\]
In this case 
\[
\gamma_N^{\frac{N}{N-1}}  \haus^N(E\cap B_r(x)) \le \haus^N(E \cap V) \le \left(1 - \gamma_N^{\frac{N}{N-1}} \right) \haus^N(E\cap B_r(x)),
\]
as well. Using the elementary inequality
\[
(2-2^{1-\frac1N})(\min\{c,d\})^{1-\frac1N} \le c^{1-\frac1N} + d^{1-\frac1N} -(c+d)^{1-\frac1N}
\qquad
\forall\,c,d\ge0,
\]
and \eqref{eq:DefTau} we obtain the estimate
\begin{equation}\label{eq:Alternativa1}
    \begin{split}
        \haus^N(E \cap V)^{1-\frac1N} 
        &+\haus^N(E \cap B_r(x) \setminus V )^{1-\frac1N}
        -\haus^N(E\cap B_r(x))^{1-\frac1N}\\
        &\ge (2-2^{1-\frac1N})\min\left\{\haus^N(E \cap V)^{1-\frac1N}, \haus^N(E \cap B_r(x) \setminus V )^{1-\frac1N}\right\} \\
        &\ge (2-2^{1-\frac1N})\gamma_N \haus^N(E\cap B_r(x))^{1-\frac1N} \\
        &= \frac{8\tau}{C_N} \haus^N(E\cap B_r(x))^{1-\frac1N}.
    \end{split}
\end{equation}
Since $P(E,\partial B_r(x))=0$, then $0=\haus^{\rm cod\text{-}1}(\partial^e E \cap \partial B_r(x)) \ge \haus^{\rm cod\text{-}1}(\partial^e E \cap \partial^e B_r(x)) $. Then $P(E \cap B_r(x),\cdot) = P(E,\cdot)|_{B_r(x)} + P(B_r(x),\cdot)|_{E^{(1)}}$ by \cref{lem:per_inters_general}. Recalling also \eqref{eq:PropertiesBallsSzero}, that the balls $\{B_{s_0}(y_i)\}$ are pairwise located at positive distance and that each of them is located at positive distance from $\partial B_r(x)$, we can estimate
\begin{equation*}
    \begin{split}
        P(E,B_r(x)) 
        &= P(E\cap B_r(x)) - P(B_r(x), E^{(1)})  \\
        & = P(E \cap V) + P(E \cap B_r(x) \setminus V) 
        - 2\sum_{i=1}^M  P(B_{s_0}(y_i), E^{(1)})
        - P(B_r(x), E^{(1)})  \\
        &\ge C_N\left(1-\eps - \bar{C}r \right)\left[(\haus^N(E\cap V))^{1-1/N} +
        \haus^N(E \cap B_r(x) \setminus V )^{1-1/N}
        \right]
        + {} \\  &\phantom{=}{}
        -2\frac\tau2\haus^N(B_r(x)\cap E)^{1-1/N} - \tau \haus^N(B_r(x)\cap E)^{1-1/N} \\
        &\ge C_N\left(1-\eps - \bar{C}r \right)\haus^N(B_r(x)\cap E)^{1-1/N} 
        + {} \\  &\phantom{=}{}
        + \left[8\tau \left(1-\eps - \bar{C}r \right)-2\tau\right]\haus^N(B_r(x)\cap E)^{1-1/N} \\
        &\ge \left(C_N\left(1-\eps - \bar{C}r \right) + 4\tau \right)\haus^N(B_r(x)\cap E)^{1-1/N},
    \end{split}
\end{equation*}
where we used \eqref{eq:QuasiIsopApplicata} in the first inequality since $r<\rho(o,\eps)$, \eqref{eq:Alternativa1} in the second inequality, and the assumed upper bound on $r$ in the last inequality.

The alternative stated in the second part of the statement follows because $A_\tau(E)$ contains balls of arbitrarily small radii by \cref{rem:AtauNonempty}. So, for given $0<\lambda\le \min\left\{\rho(o,\eps), \tfrac{1}{\bar{C}}\left(\tfrac34 - \eps \right) \right\}$, there is $B_r(x) \in A_\tau(E)$ with $0<r< \lambda$, and then an inductive application of the first part of the statement yields the sought alternative.
\end{proof}

\subsection{Interior and exterior points}\label{sec:InteriorExterior}

Let $(\X,\dist,\haus^N)$ be an $\RCD(K,N)$ space with $N\ge 2$ natural number. Let $\Omega\subset \X$ be an open set. In this section we consider a functional
\[
G:\left\{\text{$\haus^N$-measurable sets in $\Omega$} \right\}/\sim \,\,\to (-\infty,+\infty],
\]
where $E\sim F$ if and only if $\haus^N(E\Delta F)=0$, where $E\Delta F$ denotes the symmetric difference between $E$ and $F$,
such that
\begin{equation}\label{eq:ConditionG}
    \begin{array}{lll}
    &G(\emptyset)<+\infty,&\\
    &\forall\widetilde\Omega\Subset \Omega \text{ bounded open }\,\,&\exists C_G>0, \sigma>1-\frac1N \st \\ 
    &   &G(E)\le G(F) + C_G\haus^N(E\Delta F)^\sigma ,
    \end{array}
\end{equation}
for any Borel sets $E,F \subset \Omega$ such that $E\Delta F\subset\widetilde\Omega$. Observe both $C_G$ and $\sigma$ may depend on $\widetilde\Omega$. Also \eqref{eq:ConditionG} implies that $G$ is finite on bounded subsets of $\Omega$. Notice that \eqref{eq:ConditionG} exactly coincides with the assumption \eqref{eq:ConditionGIntro} we gave in the discussion in the Introduction.

For such a function $G$, we define the \emph{quasi-perimeter $\mathscr P$ restricted to $\Omega$} by
\begin{equation*}
    {\mathscr P}(E,\Omega) \eqdef P(E,\Omega) + G(E\cap \Omega),
\end{equation*}
for any $\haus^N$-measurable set $E$ in $\Omega$. When $\Omega=\X$ we can also define the \emph{quasi-perimeter}
\begin{equation}\label{eqn:QuasiPerimeter}
    {\mathscr P}(E) \eqdef P(E) + G(E),
\end{equation}
for any $\haus^N$-measurable set $E$ in $\X$.

\begin{definition}\label{def:Minimizers}
Let $(\X,\dist,\haus^N)$, $\Omega\subset \X$, $G$, $\mathscr P$ be as above.

We say that a set of locally finite perimeter $E \subset \X$ is a \emph{volume constrained minimizer of $\mathscr P$ in $\Omega$} if for any $\haus^N$-measurable set $F$ such that there is a compact set $K\subset \Omega$ with $\haus^N((E\Delta F)\setminus K)=0$, and such that $\haus^N(F \cap K)=\haus^N(E \cap K)$ it occurs that
\[
\mathscr P(E,\Omega) \le \mathscr P(F,\Omega).
\]

If $\Omega=\X$, $\haus^N(E)<+\infty$, and $E$ satisfies that $\mathscr P(E) \le \mathscr P(F)$ for any $F$ with $\haus^N(F)=\haus^N(E)$, we say that $E$ is a \emph{volume constrained minimizer of $\mathscr P$}.
\end{definition}

From now on and for the rest of the section, let $\Omega$ be an open set in $\X$ and let $E \subset \X$ be a volume constrained minimizer of $\mathscr P$ in $\Omega$. Without loss of generality we assume $P(E,\Omega)>0$. In fact, if $P(E,\Omega)=0$, then applying \cref{prop:RelativeIsoperimetricInequality} we have that either $E\cap\Omega=\emptyset$, $E\cap\Omega=\Omega$, or $E\cap\Omega$ is the union of some connected components of $\Omega$, and thus a good description of the set $E$ is readily established.
Notice moreover that, from the locality of the perimeter, having $P(E,\Omega)>0$ implies $\min\{ \haus^N(E\cap \Omega), \haus^N(\Omega\setminus E)\}>0$.

Let $\tau,\eps$ be as in \eqref{eq:DefTau}, \eqref{eq:DefEpsilon}.
Let $o \in \Omega \cap E^{(0)} \cap \{\vartheta[\X,\dist,\haus^N]=1 \}$ be fixed. Then let $R_o$, $C_{E}$ be given by \cref{lem:VolumeDecayEstimate}. Without loss of generality we can assume that $R_o<R(\eps,o)$, where $R(\eps,o)$ (together with $\rho(o,\eps)$ and $\bar C$) is given by \cref{prop:AlmostEuclIsop}. Let $a,b$ be as in \eqref{eq:Defab} and define $A_\tau(E)$ as in \eqref{eq:DefAtau}. Therefore we are in position to apply \cref{prop:KeyAlternative} on the set $E$. The next result tells that the minimality assumption on $E$ implies that only the second alternative in \cref{prop:KeyAlternative} occurs.

\begin{lemma}\label{lem:ExistenceSequenceBalls}
If $\mathscr P$, $E$, $A_\tau(E)$ are as above, then there exists a sequence of balls $\{B_{r_i}(x_i)\}_{i \in \N} \subset A_\tau(E)$ such that $B_{r_{i+1}}(x_{i+1}) \subset B_{r_i}(x_i)$ and $r_{i+1}\in \left(\tfrac{r_i}{16},\tfrac38 r_i \right)$.
\end{lemma}

\begin{proof}
We can assume without loss of generality that $E$ has no exterior points in $B_{R_o}(o)$, i.e.,  $\haus^N(E\cap B_r(x))>0$ for any $r>0$ and $x \in B_{R_o}(o)$, otherwise the claim immediately follows.
We then prove that the occurrence of the first alternative in \cref{prop:KeyAlternative} leads to a contradiction with the minimality of $E$. So we assume by contradiction that for any $0<\lambda\le \min\left\{\rho(o,\eps), \tfrac{1}{\bar{C}}\left(\tfrac34 - \eps \right) \right\} $ there is $B_t(z) \in A_\tau(E)$ with $t\in(0,\lambda)$ and such that
\[
P(E,B_t(z)) \ge  \left(C_N\left(1-\eps - \bar{C}t \right) + 4\tau \right)\haus^N(B_t(z) \cap E)^{1-\frac1N}.
\]
We impose
\begin{equation}\label{eq:DefLambda}
\begin{split}
     &0< \lambda < \min \left\{ \rho(o,\eps), \frac{1}{\bar{C}}\left(\frac34 - \eps \right),
  \frac{\dist(o,\partial\Omega)}{2}, \left(\frac{a}{2C_{E}}\right)^{N-1} ,
  \left(\frac{V_\tau}{2C_{E}}\right)^{\frac{N-1}{N^2}},
  \frac{\tau}{4C_N\bar{C}}
   \right\}, \\
   & C_G2^\sigma(2C_{E}\lambda^{\frac{N^2}{N-1}})^{\sigma-\frac{N-1}{N}} \le \tau,
\end{split}
\end{equation}
where $C_G$ and $\sigma$ are as in \eqref{eq:ConditionG} and $V_\tau$ is given by \cref{cor:MJ-PalleVolumiPiccoli} taking $\eps_1=\tau$.
By absurd assumption, this gives a ball $B_t(z) \in A_\tau(E)$ as above. Observe that $B_t(z)\Subset \Omega$ by the choice of $\lambda$. We construct the comparison set
\[
F\eqdef (E\setminus B_t(z) ) \cup B,
\]
where $B$ is the ball centered at $z$ with measure $\haus^N(B) = \haus^N(E\cap B_t(z))$. Since
\[
\haus^N(B)=\haus^N(E \cap B_t(z)) \le 2C_{E} t^{\frac{N^2}{N-1}} < a t^N \le \haus^N(B_t(z)),
\]
by \eqref{eq:DefLambda}, we see that $B\Subset B_t(z)$ and $\haus^N(F\cap B_t(z))=\haus^N(E \cap B_t(z))$, and thus $\mathscr P(E,\Omega)\le \mathscr P(F,\Omega)$. Using \eqref{eq:DefLambda} we estimate
\begin{equation}\label{eq:EstimateSymmDifference}
\begin{split}
C_G\haus^N(F\Delta E)^\sigma
& \le C_G 2^\sigma \haus^N(E \cap B_t(z))^{\sigma-\frac{N-1}{N}} \haus^N(E \cap B_t(z))^{\frac{N-1}{N}} \\
&\le C_G 2^\sigma (2C_{E} t^{\frac{N^2}{N-1}})^{\sigma-\frac{N-1}{N}} \haus^N(E \cap B_t(z))^{\frac{N-1}{N}} \\
&\le \tau \haus^N(E \cap B_t(z))^{1-\frac1N}.
\end{split}    
\end{equation}
Hence we get
\[
\begin{split}
    \mathscr P(F,\Omega)
    &= P(E,\Omega) + P(B) + P(B_t(z),E^{(1)}) - P(E,B_t(z)) + G(F\cap \Omega) \\
    &\le P(E,\Omega) + P(B) + \tau\haus^N(B_t(z) \cap E)^{1-\frac1N} 
    + {} \\  &\phantom{=}{}
    - \left(C_N(1-\eps - \bar{C}t) + 4\tau \right)\haus^N(B_t(z) \cap E)^{1-\frac1N} + G(E\cap \Omega) + C_G\haus^N(F\Delta E)^\sigma \\
    &\le \mathscr P(E,\Omega) + P(B) + \left[\tau
    - C_N(1-\eps - \bar{C}t) - 4\tau +\tau
    \right]\haus^N(B_t(z) \cap E)^{1-\frac1N} \\
    &\le \mathscr P(E,\Omega) + \left[C_N+\tau - C_N(1-\eps - \bar{C}t) -2\tau \right]\haus^N(B_t(z) \cap E)^{1-\frac1N} \\
    &\le \mathscr P(E,\Omega) -\frac\tau2\haus^N(B_t(z) \cap E)^{1-\frac1N}
\end{split}
\]
where in the first equality we used that $P(E,\partial B_t(z))=0$ by \eqref{eq:DefAtau} and \cref{lem:per_inters_general}; in the first inequality we used \eqref{eq:DefAtau}, the absurd hypothesis, and \eqref{eq:ConditionG}; in the second inequality we used \eqref{eq:EstimateSymmDifference}; in the third inequality we used \cref{cor:MJ-PalleVolumiPiccoli}, as $\haus^N(B)=\haus^N(E\cap B_t(z))< V_\tau$ by \eqref{eq:DefLambda}; in the fourth inequality we used \eqref{eq:DefEpsilon} and \eqref{eq:DefLambda}. Since $B_t(z)\in A_\tau(E)$, then $z \in B_{R_o}(o)$ and we have $\haus^N(B_t(z) \cap E)>0$ by the assumption at the beginning of the proof. Hence $\mathscr P(F,\Omega)< \mathscr P(E,\Omega) \le \mathscr P(F,\Omega)$ gives the desired contradiction.
\end{proof}

\begin{remark}\label{rem:MinimoComplementare}
It is immediate to check that if $\Omega$, $G$ are as above and $E$ is a volume constrained minimizer of $\mathscr P$ in $\Omega$, then $\widetilde E\eqdef \Omega\setminus E$ is  a volume constrained minimizer of $\widetilde{\mathscr P}\eqdef P+ \widetilde G$ in $\Omega$, where $\widetilde G(F)\eqdef G(\Omega\setminus F)$ for any measurable set $F\subset \Omega$. Also $\widetilde G$ satisfies \eqref{eq:ConditionG} with the same constants $C_G$, $\sigma$ of $G$. For further details see also \cite[Lemma 4.2]{Xia05}.
\fr\end{remark}

Before proving the main theorem of this section, we need a last technical lemma. We stress that we are thinking about the balls as couples center-radius, and not as sets, so that the functions assigning to a ball its center and its radius are well-defined.

\begin{lemma}\label{lem:CurvaDiPalle}
Let $(\X,\dist,\meas)$ be a geodesic metric measure space. Let $0<r_{i+1}<r_i<\ldots<r_0$ be a sequence such that $\lim_i r_i=0$ and suppose that there exist points $x_i \in \X$ such that $B_{r_{i+1}}(x_{i+1})\subset B_{r_i}(x_i)$ for any $i$. Then there exists a one-parameter family of balls $\{B(s) \st s \in (0,r_0]\}$ such that
\begin{itemize}
    \item $B(r_i)=B_{r_i}(x_i)$ for any $i$;
    
    \item $B(s)\subset B_{r_i}(x_i)$ for any $s \le r_i$ for any $i$;
    
    \item the map $s\mapsto {\rm rad}(B(s)) \in \R$ associating $s$ with the radius of $B(s)$ is continuous and $r_{i+1}/2\le {\rm rad}(B(s)) \le r_i$ for $s \in [r_{i+1},r_i]$ for any $i$;
    
    \item the map $s \mapsto {\rm cent}(B(s)) \in \X$ associating $s$ with the center of $B(s)$ is continuous.
\end{itemize}
In particular $ s \mapsto \nchi_{B(s)}$ is continuous with respect to the $L^1(\meas)$ topology.
\end{lemma}

\begin{proof}
It suffices to define $B(s)$ for $s \in [r_1,r_0]$, then the construction can be iterated. Let $\gamma:[0,1]\to B_{r_0}(x_0)$ be a geodesic from $x_1$ to $x_0$, i.e., $\gamma(0)=x_1$, $\gamma(1)=x_0$, and $\dist(\gamma(t),\gamma(r))=|t-r|\dist(x_0,x_1)$. We define
\[
B(s) \eqdef \begin{cases}
B_{a+bs}(x_1) & s \in \left[ r_1, r_1 + \frac{r_0-r_1}{3}  \right], \\
B_{\frac{r_1}{2}}\left(
\gamma\left( 
\frac{3}{r_0-r_1}(s-r_1-\tfrac{r_0-r_1}{3})
\right)
\right) 
& s \in\left( r_1+ \frac{r_0-r_1}{3}, r_0 - \frac{r_0-r_1}{3}  \right) , \\
B_{c+ds}(x_0)  & s \in\left[ r_0 - \frac{r_0-r_1}{3}, r_0  \right],
\end{cases}
\]
where
\begin{equation*}
    \begin{array}{lll}
    & a=r_1+\frac32 \frac{r_1^2}{r_0-r_1} , & b=-\frac32\frac{r_1}{r_0-r_1}, \\
    & c=\frac{r_1}{2} - \frac{(r_0-r_1/2)(2r_0+r_1)}{r_0-r_1} , 
    & d= 3\frac{r_0-r_1/2}{r_0-r_1}.
    \end{array}
\end{equation*}
For $s \in \left[ r_1, r_1 + \frac{r_0-r_1}{3}  \right] $ (resp.\! $s \in\left[ r_0 - \frac{r_0-r_1}{3}, r_0  \right]$) the center is fixed and the radius linearly passes from $r_1$ to $r_1/2$ (resp.\! from $r_1/2$ to $r_0$). Hence obviously $B(s)\subset B_{r_0}(x_0)$ for $s \in \left[ r_1, r_1 + \frac{r_0-r_1}{3}  \right] \cup \left[ r_0 - \frac{r_0-r_1}{3}, r_0  \right]$. While for $s \in\left( r_1+ \frac{r_0-r_1}{3}, r_0 - \frac{r_0-r_1}{3}  \right)$ we have that if $\dist({\rm cent}(B(s)),x_1) < r_1/2$, then $B(s)\subset B_{r_1}(x_1)\subset B_{r_0}(x_0)$; then for such $s$ we can assume that $\dist({\rm cent}(B(s)),x_1) \ge r_1/2$, that is
\begin{equation}\label{eq:StimaCurvaDiPalle}
\begin{split}
\dist(x_0,x_1)\frac{3}{r_0-r_1}(s-r_1-\frac{r_0-r_1}{3}) &=
\dist\left(\gamma\left( 
\frac{3}{r_0-r_1}(s-r_1-\frac{r_0-r_1}{3})
\right), x_1\right) \\
&= \dist({\rm cent}(B(s)),x_1) \ge \frac{r_1}{2}.
\end{split}
\end{equation}
And for any $q \in B(s)$ we estimate
\[
\begin{split}
\dist(q,x_0)
&\le \dist(q,{\rm cent}(B(s))) + \dist({\rm cent}(B(s)), x_0) \\
&\le \frac{r_1}{2} + \dist(x_0,x_1)\left[ 1 - \frac{3}{r_0-r_1}(s-r_1-\tfrac{r_0-r_1}{3})   \right] \\
&\le \dist(x_0,x_1) \le r_0,
\end{split}
\]
that implies $B(s)\subset B_{r_0}(x_0)$, where in the third inequality we used \eqref{eq:StimaCurvaDiPalle}. The remaining claims in the statement clearly follow from the construction.
\end{proof}

We can finally prove the main theorem of the section.

\begin{theorem}\label{thm:InteriorAndExteriorPoints}
Let $(\X,\dist,\haus^N)$ be an $\RCD(K,N)$ space with $N\ge 2$ natural number, let $\Omega\subset \X$ be open, and let $\mathscr P=P+G$ be a quasi-perimeter. Let $E$ be a volume constrained minimizer of $\mathscr P$ in $\Omega$ with $P(E,\Omega)>0$.

Then $E\cap \Omega$ has both exterior and interior points.
\end{theorem}

\begin{proof}
It suffices to show that $E$ has exterior points in $\Omega$. Indeed, by \cref{rem:MinimoComplementare}, once we know that volume constrained minimizers have exterior points, we can apply this fact to the complement $\Omega\setminus E$, thus getting that $E$ also has interior points. So let us assume by contradiction that $\haus^N(E \cap B_r(x))>0$ for any $B_r(x)\subset \Omega$.

By hypotheses, we can take $o \in E^{(0)}\cap \Omega \cap \{\vartheta[\X,\dist,\haus^N]=1\}$ and define $A_\tau(E)$ together with all the parameters in \eqref{eq:DefAtau}, ending up in the hypotheses of \cref{lem:ExistenceSequenceBalls}. This yields a sequence of balls $\{B_{r_i}(x_i)\}_{i \in \N} \subset A_\tau(E)$ such that $B_{r_{i+1}}(x_{i+1}) \subset B_{r_i}(x_i)$ and $r_{i+1}\in \left(\tfrac{r_i}{16},\tfrac38 r_i \right)$. Also we can assume that $B_{r_0}(x_0)\Subset\Omega$ and $r_0<\rho(o,\eps)$, where $\rho(o,\eps)$ is given by \cref{prop:AlmostEuclIsop}. Given such a sequence, let $\{B(s) \st s \in (0,r_0]\}$ be the one-parameter family of balls given by \cref{lem:CurvaDiPalle}. Observe that $r_i/32 \le r_{i+1}/2\le {\rm rad}(B(s)) \le r_i$ for $s \in [r_{i+1},r_i]$ for any $i$. Also
\begin{equation}\label{eq:StimaVolumeBs}
\begin{split}
    0<\haus^N(E \cap B(s)) &\le \haus^N(E \cap B_{r_i}(x_i)) \le 2C_{E} r_i^{\frac{N^2}{N-1}} \le 2C_{E} 16^{\frac{N^2}{N-1}} r_{i+1}^{\frac{N^2}{N-1}} \\
    &\le 2C_{E} 16^{\frac{N^2}{N-1}} s^{\frac{N^2}{N-1}} \eqqcolon L_1  s^{\frac{N^2}{N-1}},
\end{split}
\end{equation}
where $i\in\N$ is the index such that $r_{i+1}\le s \le r_i$.

Let $w \in \Omega$ be the limit point $w\eqdef\lim_{s\to0}{\rm cent}(B(s))$ of the centers of $B(s)$. Let $\{\Omega_\ell \}$ be the family of connected components $\Omega_\ell$ of $\Omega\setminus \{w\}$. Since $\X$ is separable, such components are countably many at most. Also, since the space is geodesic, $\X$ is locally path-connected, and then $\Omega_\ell$ is open for any $\ell$. By the absurd assumption, we have that $\haus^N(E \cap \Omega_\ell)>0$ for any $\ell$.
Since $P(E,\Omega)>0$, there exists a connected component $\Omega_\ell$, in the following denoted by $\Omega'$, such that $P(E,\Omega')>0$. In particular we have that $\haus^N(\Omega'\setminus E)>0$. We now distinguish two cases, depending on whether $E$ has interior points in $\Omega'$.

\emph{Case 1.} 
Let us assume that $E$ has no interior points in $\Omega'$, i.e., $\haus^N(B_\rho(z)\setminus E)>0$ for any $B_\rho(z)\subset \Omega'$.
By \cref{rem:MinimoComplementare}, the same arguments at the beginning of the proof can be applied to the complement $\Omega\setminus E$ at some point $o' \in E^{(1)}\cap \Omega' \cap \{\vartheta[\X,\dist,\haus^N]=1\}$, which plays the role of $o$ above. Eventually, this yields a one-parameter family of balls $\{K(t) \st t \in (0,t_0]\}$ such that $K(t)\subset K(t_0)\Subset \Omega'\subset \Omega$ and
\begin{equation}\label{eq:StimaVolumeKt}
    0<\haus^N(K(t)\setminus E ) = \haus^N((\Omega \setminus E) \cap K(t) ) \le L_2 t^{\frac{N^2}{N-1}},
\end{equation}
for some $L_2>0$, where the first inequality follows since $E$ has no interior points in $\Omega'$. Moreover $t\mapsto {{\rm rad}(K(t))} \in \R$ and $t \mapsto \nchi_{K(t)} \in L^1(\haus^N)$ are continuous. Also, up to decreasing $r_0$ and $t_0$, we can assume that $B(r_0)$ and $K(t_0)$ are located at positive distance.

For $R=\max\{r_0,t_0\}$, let $C_{K,N,R}$ be given by \cref{thm:MainEst}. 
Fix $\eps_0\in (0, \min\{r_0,t_0\}]$ such that
\begin{equation}\label{eq:DefEpsZero}
 C_{K,N,R}L_2^{\frac1N}\eps_0^{\frac{1}{N-1}} \le \tau,
\end{equation}
and
\[
 \haus^N(E\cap B(\eps_0)) < \haus^N(K(t_0)\setminus E ).
\]
By continuity and since ${\rm rad}(K(t))\to0$ as $t\to0^+$, for any $s \in (0,\eps_0)$ there exists a minimum value $g(s) \in (0,t_0]$ such that
\begin{equation}\label{eq:EqualityVolumes}
    \haus^N(E \cap B(s)) = \haus^N(K(g(s))\setminus E).
\end{equation}
Let $C_G,\sigma$ be given by \eqref{eq:ConditionG} where $\widetilde\Omega\Subset \Omega$ is some bounded open set such that $B(r_0),K(t_0)\Subset \widetilde\Omega$.
From now on, we fix $s \in (0,\eps_0)$ such that
\begin{equation}\label{eq:DefS}
\begin{array}{lll}
    & B(s) \in A_\tau(E),
    & C_{K,N,R}L_1^{\frac1N}s^{\frac{N}{N-1}}\le \eps_0\tau,\\
    & C_N\bar{C}{\rm rad}(B(s)) \le \frac34\tau,
    & 2^\sigma C_G\haus^N(E\cap B(s)))^{\sigma-1+\frac1N} \le \tau.
\end{array}
\end{equation}
For such an $s$, we consider the comparison set
\[
F \eqdef (E \setminus B(s) ) \cup K(g(s)).
\]
Hence $E\Delta F \Subset \Omega$ and $\haus^N(E \cap ( B(r_0) \cup K(t_0)))=\haus^N(F \cap ( B(r_0) \cup K(t_0)))$ by \eqref{eq:EqualityVolumes}. Also, since $B(s) \in A_\tau(E)$, we have
\begin{equation}\label{eq:StimaProofPuntiIntExt}
\begin{split}
        P(F,\Omega) &= P(E \cup K(g(s)), \Omega) - P(E,B(s)) + P(B(s), E^{(1)}) \\
        &\le P(E,\Omega) + \frac{C_{K,N,R}}{g(s)}\haus^N(K(g(s))\setminus E)
        - P(E,B(s)) + P(B(s), E^{(1)}) \\
        &\le P(E,\Omega) + \frac{C_{K,N,R}}{g(s)}\haus^N(K(g(s))\setminus E)
        - P(E,B(s)) + \tau \haus^N(E\cap B(s))^{1-\frac1N}
\end{split}
\end{equation}
where in the equality we used \cref{lem:per_inters_general} and \eqref{eq:DefAtau}, in the first inequality we applied \eqref{eq:main_claim2NEW} locally in $\Omega$, and in the second inequality we used \eqref{eq:DefAtau}. Now if $g(s)\le\eps_0$, by \eqref{eq:EqualityVolumes}, \eqref{eq:StimaVolumeKt}, and \eqref{eq:DefEpsZero}, we have
\[
\begin{split}
    \frac{C_{K,N,R}}{g(s)}\haus^N(K(g(s))\setminus E)
    &\le\frac{C_{K,N,R}}{g(s)}\haus^N(E\cap B(s))^{1-\frac1N}L_2^{\frac1N}g(s)^{\frac{N}{N-1}} \\
    &\le C_{K,N,R}\haus^N(E\cap B(s))^{1-\frac1N}L_2^{\frac1N}\eps_0^{\frac{1}{N-1}} \\
    &\le \tau \haus^N(E\cap B(s))^{1-\frac1N},
\end{split}
\]
and if $g(s)>\eps_0$, by \eqref{eq:EqualityVolumes}, \eqref{eq:StimaVolumeBs}, and \eqref{eq:DefS}, we have
\[
\begin{split}
    \frac{C_{K,N,R}}{g(s)}\haus^N(K(g(s))\setminus E)
    &\le\frac{C_{K,N,R}}{g(s)}\haus^N(E\cap B(s))^{1-\frac1N}L_1^{\frac1N}s^{\frac{N}{N-1}} \\
    &\le C_{K,N,R}\haus^N(E\cap B(s))^{1-\frac1N}L_1^{\frac1N}s^{\frac{N}{N-1}}/\eps_0 \\
    &\le \tau \haus^N(E\cap B(s))^{1-\frac1N},
\end{split}
\]
as well. Plugging into \eqref{eq:StimaProofPuntiIntExt} we obtain
\begin{equation*}
\begin{split}
    P(F,\Omega) &\le  P(E,\Omega) + 2 \tau \haus^N(E\cap B(s))^{1-\frac1N} - P(E,B(s)).
\end{split}    
\end{equation*}
Since $B(s)\in A_\tau(E)$ and $r_0<\rho(o,\eps)$, the almost Euclidean isoperimetric inequality in \cref{prop:AlmostEuclIsop} reads as
\[
C_N(1-\eps-\bar{C}{\rm rad}(B(s))\haus^N(E\cap B(s)))^{1-\frac1N} \le P(E\cap B(s)) = P(E,B(s)) + P(B(s),E^{(1)}).
\]
Therefore
\[
\begin{split}
 P(F,\Omega) &\le P(E,\Omega) + \left[2\tau -C_N +\eps C_N + C_N\bar{C}{\rm rad}(B(s)) \right]\haus^N(E\cap B(s))^{1-\frac1N} + P(B(s),E^{(1)}) \\
 &\le  P(E,\Omega) + \left[3\tau -C_N +\eps C_N + C_N\bar{C}{\rm rad}(B(s)) \right]\haus^N(E\cap B(s))^{1-\frac1N} \\
 &\le P(E,\Omega) - 4\tau \haus^N(E\cap B(s)))^{1-\frac1N},
\end{split}
\]
where in the last equality we used \eqref{eq:DefTau} (which implies $C_N>8\tau$), \eqref{eq:DefEpsilon}, and \eqref{eq:DefS}. Finally this implies
\[
\begin{split}
    \mathscr P(F,\Omega) 
    &\le P(E,\Omega) - 4\tau \haus^N(E\cap B(s))^{1-\frac1N} + G(E\cap \Omega) + C_G\haus^N(E\Delta F)^\sigma \\
    &\le \mathscr P(E,\Omega)   + \left(2^\sigma C_G\haus^N(E\cap B(s))^{\sigma-1+\frac1N} - 4\tau\right) \haus^N(E\cap B(s))^{1-\frac1N} \\
    &\le \mathscr P(E,\Omega) - 3\tau\haus^N(E\cap B(s))^{1-\frac1N},
\end{split}
\]
where in the last inequality we used \eqref{eq:DefS}. As $E$ has no exterior points, this implies $\mathscr P(F,\Omega) < \mathscr P(E,\Omega)$, which contradicts the volume constrained minimality of $E$ in $\Omega$.

\emph{Case 2.} It remains to consider that case where there is a ball $B_\rho(x)\Subset\Omega'$ such that $\haus^N(B_\rho(x)\setminus E)=0$. Up to decrease $r_0$, we can assume that there is $y \in E^{(0)}\cap \Omega' \setminus \overline{B(r_0)}$.
We now argue similarly as in the proof of \cref{thm:VariazioniMaggi} in order to find a suitable interior point. Since $\Omega'$ is path-connected, there is a Lipschitz curve $\gamma:[0,L]\to \Omega'$ with finite length $L$ such that $\gamma(0)=x$ and $\gamma(L)=y$, and we parametrize $\gamma$ so that its metric derivative $|\gamma'|=1$ almost everywhere (see \cref{rem:LispchitzPath}). Let
\[
r\eqdef\frac12 \min\Big\{\rho,\inf\big\{ \dist(\gamma(t),\partial\Omega')\st t \in [0,L]\big\}\Big\}>0.
\]
Consider the family of balls
\[
\left\{B_r(\gamma(t_k)) \st
t_k = \min\left\{ k\frac{r}{2}, L\right\}, \quad 
k\in \N
\right\}.
\]
By construction $\dist(\gamma(t_{k+1}) ,\gamma(t_k))\le r/2$. Hence, if $k$ satisfies $\haus^N(B_r(\gamma(t_k))\setminus E) =0$, then $\gamma(t_{k+1})$ is an interior point. Since $y \in E^{(0)}$ there exists a first index $k_0$ such that
\[
\haus^N(B_r(\gamma(t_{k_0}))\setminus E) >0.
\]
As in the proof of \cref{thm:VariazioniMaggi}, the point $\gamma(t_{k_0})$ is an interior point and $0<t_{k_0}<L$. Finally, $\dist(w,B_r(\gamma(t_{k_0})))>0$ by definition of $r$ as $w \in \partial \Omega'$.

All in all, we found a point $z \in \Omega'$ and a radius $r>0$ such that
\begin{equation}\label{eq:DefZ}
    \haus^N(B_{\frac{r}{2}}(z)\setminus E)=0,
    \qquad
    \haus^N(B_r(z)\setminus E)>0,
    \qquad
    \dist(z,\partial \Omega')\ge2r.
\end{equation}
This time we consider the one-parameter family of balls $Q(t)\eqdef B_t(z)\Subset \Omega'$ for $t\in(0, r]$, that shall play the role of the family $K(t)$ of the previous case. Up to taking a smaller $r_0$, we additionally have that $Q(r)$ and $B(r_0)$ are located at positive distance.

We now consider $\widetilde R=\max\{r_0,r\}$ and $C_{K,N,\widetilde R}$ the constant given by \cref{thm:MainEst}. Fix $\eps_1 \in (0,\min\{r/2,r_0\})$ such that
\[
\haus^N(E \cap B(\eps_1)) < \haus^N( Q(r)\setminus E).
\]
By continuity and since $ \haus^N( Q(r/2)\setminus E)= 0$, we have that for any $s \in (0,\eps_1)$ there is a minimum value $h(s)\in (r/2,r]$ such that
\begin{equation}\label{eq:EqualityVolumes2}
    \haus^N(E \cap B(s)) = \haus^N(Q(h(s))\setminus E).
\end{equation}
Let $C_G,\sigma$ be given by \eqref{eq:ConditionG} where $\widetilde\Omega\Subset \Omega$ is some bounded open set such that $B(r_0),Q(r)\Subset \widetilde\Omega$.
From now on, we fix $\tilde{s} \in (0,\eps_1)$ such that
\begin{equation}\label{eq:DefS2}
\begin{array}{lll}
    & B(\tilde{s}) \in A_\tau(E),
    &  C_{K,N,\widetilde R}L_1^{\frac1N} \tilde s^{\frac{N}{N-1}}\le \eps_1\tau,\\
    &  C_N\bar{C}{\rm rad}(B(\tilde s)) \le \frac34\tau,
    &  2^\sigma C_G\haus^N(E\cap B(\tilde s))^{\sigma-1+\frac1N} \le \tau.
\end{array}
\end{equation}
For such an $\tilde{s}$, this time we consider the comparison set
\[
\widetilde F \eqdef (E \setminus B(\tilde{s}) ) \cup Q(h(\tilde{s})).
\]
Hence $E\Delta \widetilde F\Subset \Omega$ and $\haus^N(E \cap (B(r_0)\cup Q(r) ))=\haus^N(\widetilde F \cap (B(r_0)\cup Q(r) ))$ by \eqref{eq:EqualityVolumes2}. As in \eqref{eq:StimaProofPuntiIntExt} we estimate
\begin{equation}\label{eq:StimaProofPuntiIntExt2}
    \begin{split}
        P(\widetilde F, \Omega) \le P(E,\Omega) + \frac{C_{K,N,\widetilde R}}{h(\tilde s)} \haus^N( Q(h(\tilde s)) \setminus E ) - P(E,B(s)) + \tau \haus^N(E\cap B(s))^{1-\frac1N}.
    \end{split}
\end{equation}
This time, since $h(\tilde s)>r/2\ge \eps_1$, we have
\[
\begin{split}
\frac{C_{K,N,\widetilde R}}{h(\tilde s)} \haus^N( Q(h(\tilde s)) \setminus E ) 
&\le \frac{C_{K,N,\widetilde R}}{h(\tilde s)} \haus^N(E \cap B(\tilde s) )^{1-\frac1N}L_1^{\frac1N}\tilde s^{\frac{N}{N-1}} \\
&\le C_{K,N,\widetilde R}\haus^N(E \cap B(\tilde s) )^{1-\frac1N}L_1^{\frac1N}\tilde s^{\frac{N}{N-1}} / \eps_1 \\
&\le \tau \haus^N(E\cap B(\tilde s))^{1-\frac1N},
\end{split}
\]
by \eqref{eq:EqualityVolumes2}, \eqref{eq:StimaVolumeBs}, and \eqref{eq:DefS2}. Plugging into \eqref{eq:StimaProofPuntiIntExt2} we obtain
\[
 P(\widetilde F, \Omega) \le P(E,\Omega) + 2 \tau \haus^N(E\cap B(\tilde s))^{1-\frac1N} - P(E,B(\tilde s)) .
\]
Employing \cref{prop:AlmostEuclIsop} exactly as in Case 1 and \eqref{eq:DefS2} we infer
\[
 P(\widetilde F, \Omega) \le P(E,\Omega) -4 \tau \haus^N(E\cap B(\tilde s))^{1-\frac1N}.
\]
We then conclude as in Case 1 with the estimate
\[
\begin{split}
    \mathscr P(\widetilde F,\Omega) 
    &\le P(E,\Omega) - 4\tau \haus^N(E\cap B(\tilde s))^{1-\frac1N} + G(E\cap \Omega) + C_G\haus^N(E\Delta F)^\sigma \\
    &\le \mathscr P(E,\Omega)   + \left(2^\sigma C_G\haus^N(E\cap B(\tilde s))^{\sigma-1+\frac1N} - 4\tau\right) \haus^N(E\cap B(\tilde s))^{1-\frac1N} \\
    &\le \mathscr P(E,\Omega) - 3\tau\haus^N(E\cap B(\tilde s))^{1-\frac1N},
\end{split}
\]
where in the last equality we used \eqref{eq:DefS2}. Again, this contradicts the volume constrained minimality of $E$ in $\Omega$.
\end{proof}

\subsection{Isoperimetric inequality for small volumes in PI spaces}\label{sec:SmallPI}

In this section we prove that an isoperimetric inequality for small volumes holds on a subclass of PI spaces. Such class contains for example $\CD(K,N)$ spaces, with $K\in\mathbb R$, $N<+\infty$, that have a uniform lower bound on the volumes of unit balls. The proof of the forthcoming result is an adaption to the nonsmooth context of \cite[Chapter 3: Lemma 3.1, Lemma 3.2]{Heb00}.

Before stating the first result, we give a definition. Let $(\X,\dist,\meas)$ be a metric measure space. We say that a {\em weak local $(1,1)$ non-averaged Poincar\'{e} inequality} holds on $(\X,\dist,\meas)$ if there exists $\lambda\geq 1$ such that for every $R>0$ there exists $C_P>0$ such that the following holds. For every pair of functions $(f,g)$, where $f\in L^1_{\mathrm{loc}}(\X,\meas)$ and $g$ is an upper gradient of $f$, we have 
$$
\int_{B_r(x)}|f-\overline f(x)|\de\meas \leq C_Pr\int_{B_{\lambda r}(x)}g\de\meas,\quad \forall x\in \X\;\forall r<R,
$$
where $\overline f(x):=\fint_{B_r(x)}f\de\meas$.

Notice that, taking into account \cref{rem:BGBella}, a metric measure space $(\X,\dist,\meas)$ which satisfies a {\em weak local $(1,1)$ non-averaged Poincar\'{e} inequality} and on which $\meas$ is uniformly locally doubling, a weak local $(1,1)$-Poincar\'{e} inequality holds; and thus it is a PI space.

\begin{lemma}\label{lem:EnHancedPoincare}
Let $(\X,\dist,\meas)$ be a metric measure space such that $\meas$ is uniformly locally doubling and such that it satisfies a weak local $(1,1)$ non-averaged Poincar\'{e} inequality in the sense above. Let $R>0$. Then there exists a constant $C_1$ such that the following holds. 

For every $u\in L^1_{\mathrm{loc}}(\X)$, every upper gradient $g$ of $u$,  and every $r\in (0,R]$, we have that 
\begin{equation}\label{eqn:EnHancedPoincare}
\int_X |u-\overline u_r|\de\meas\leq C_1r\int_X g\de\meas,
\end{equation}
where, for $x\in \X$ and $r>0$, $\overline u_r$ is a function defined as follows
$$
\overline u_r(x):=\fint_{B_r(x)}u\de\meas.
$$
\end{lemma}
\begin{proof}
Let us fix $u,g$ as in the statement, and let us fix $r\in (0,R]$.

From the hypotheses we know that there exist $\lambda\geq 1$ and $C_P$ such that 
$$
\int_{B_s(x)}|f-\overline f(x)|\de\meas \leq C_Ps\int_{B_{\lambda s}(x)}g\de\meas, \quad \forall x\in \X\;\forall s\leq\lambda R.
$$
From an application of item (ii) of \cref{lem:CoveringLemma} we get the existence of countably many points $\{x_i\}_{i\in\mathbb N}$ such that
\begin{equation}\label{eq:CoveringIsop}
    \begin{split}
        B_{\frac r2}(x_i) \cap B_{\frac r2}(x_j) = \emptyset& \qquad \forall\,i\neq j, \\
        \bigcup_{i\in\mathbb N} B_r(x_i) =M& , \\
    \end{split}
    \end{equation}
and, for some constants $C_1,C_2$, the following holds
    \begin{equation}\label{eq:StimaCoveringOverlapIsop}
    \begin{split}
        \sharp \left\{ \text{balls $B_{\lambda r}(x_i) \st z \in B_{\lambda r}(x_i)$} \right\} \le \max\left\{ 1,C_1(8\lambda)^{C_2} \right\}=:A_1, \\
        \sharp \left\{ \text{balls $B_{\lambda^2r}(x_i) \st z \in B_{\lambda^2r}(x_i)$} \right\} \le \max\left\{ 1,C_1(8\lambda^2)^{C_2} \right\}=: A_2.
    \end{split}
    \end{equation}
We will just sketch the the proof since it closely follows \cite[Lemma 3.1]{Heb00}. By the triangle inequality and the second equality in \eqref{eq:CoveringIsop} we have
\begin{equation}\label{eqn:Finish}
    \begin{split}
        &\int_\X|u-\overline u_r|\de\meas\\
        \leq\,&\sum_{i\in\mathbb N}\left(\int_{B_r(x_i)}|u-\overline u_r(x_i)|\de\meas+\int_{B_r(x_i)}|\overline u_r(x_i)-\overline u_{\lambda r}(x_i)|\de\meas+\int_{B_r(x_i)}|\overline u_{\lambda r}(x_i)-\overline u_r|\de\meas\right).
    \end{split}
\end{equation}
By using the first estimate in \eqref{eq:StimaCoveringOverlapIsop} and the weak local $(1,1)$ non-averaged Poincar\'{e} inequality we get
$$
\sum_{i\in\mathbb N}\int_{B_r(x_i)}|u-\overline u_r(x_i)|\de\meas \leq A_1 C_P r \int_\X g\de\meas.
$$
By using the second estimate in \eqref{eq:StimaCoveringOverlapIsop} and the weak local $(1,1)$ non-averaged Poincar\'{e} inequality we moreover obtain
$$
\sum_{i\in\mathbb N}\int_{B_r(x_i)}|\overline u_r(x_i)-\overline u_{\lambda r}(x_i)|\de\meas \leq A_2 C_P \lambda r \int_\X g\de\meas.
$$
We notice that as a consequence of the fact that $\meas$ is uniformly locally doubling we have that there exists some constant $\vartheta>0$ such that $\meas(B_{2r}(x))/\meas(B_r(x))\leq \vartheta$ for every $x\in \X$ and every $r\leq R$. Exploiting the latter inequality, arguing as in \cite[Lemma 3.1]{Heb00}, and noticing that without loss of generality we may assume $\lambda\geq 2$ we conclude that 
$$
\sum_{i\in\mathbb N}\int_{B_r(x_i)}|\overline u_{\lambda r}(x_i)-\overline u_r|\de\meas \leq A_2 \lambda r \vartheta C_P\int_\X g\de\meas.
$$
The latter three inequalities together with \eqref{eqn:Finish} conclude the proof of the Lemma with $C_1:=C_P(A_1+A_2\lambda+A_2\lambda\vartheta)$.
\end{proof}

Before proving the main result of this section we introduce a definition. Given a real number $s>0$, we say that a metric measure space $(\X,\dist,\meas)$ is {\em uniformly lower $s$-Ahlfors regular} if there exists a radius $R>0$ and a constant $C$ such that $\meas(B_r(x))\geq Cr^s$ for every $x\in \X$ and $r\in (0,R]$.

\begin{proposition}\label{prop:IsopVolumiPiccoli}
Let $(\X,\dist,\meas)$ be a metric measure space, and let $s>1$. Assume that $(\X,\dist,\meas)$ is uniformly lower $s$-Ahlfors regular, that $\meas$ is uniformly locally doubling, and that $(\X,\dist,\meas)$ satisfies a weak local $(1,1)$ non-averaged Poincar\'{e} inequality in the sense above. Then there exist $C$ and $v$ such that the following holds. 

For every set of finite perimeter $E$, the following implication holds
$$
\meas(E)\leq v \quad\Rightarrow \quad \meas(E)^{s/(s-1)}\leq CP(E).
$$
\end{proposition}

\begin{proof}
From the hypotheses we have that there exist $R>0$ and $C_2>0$ such that
\begin{equation}\label{eqn:EstimateBrX}
\meas(B_r(x))\geq C_2r^s,
\end{equation}
for every $r\in (0,R]$ and every $x\in \X$.
We claim that the assertion holds with the following constants
\begin{equation}\label{eqn:ChoiceOfConstants}
C:=\frac{4\cdot 16^{1/s}C_1}{C_2^{1/s}}, \qquad v:=\frac{C_2R^s}{16},
\end{equation}
where $C_1$ is the constant given by \cref{lem:EnHancedPoincare} associated to $R$.

Let us fix $E\subset \X$ with $\meas(E)\leq v$. By definition of perimeter, and since $\nchi_E\in L^1(\X,\meas)$, we can take $\{f_i\}_{i\in\mathbb N}\in \mathrm{Lip}_{\mathrm{loc}}(\X,\dist)$ such that 
$$
f_i\to \nchi_E \quad \text{in $L^1(\X,\meas)$}, \quad \text{and} \quad \int_\X \lip f_i\,\d\mm\to P(E).
$$
Convergence in $L^1$ implies that $\meas(\{|f_i|\ge1/2\})\to \meas(E)$. Also, for any $x$ such that $|f_i(x)|\ge1/2$, either $|\overline f_{i,r}(x)|\geq 1/4$ or $|f_i(x)-\overline f_{i,r}(x)|\geq 1/4$, for any $i$ and $r>0$. Hence
\begin{equation}\label{eqn:STIMA1}
\begin{split}
\meas(\{x\in \X: |f_i(x)|\geq 1/2\})&\leq \meas(\{x\in \X:|f_i(x)-\overline f_{i,r}(x)|\geq 1/4\})
 {} \\  &\phantom{=}{}
+\meas(\{x\in \X: |\overline f_{i,r}(x)|\geq 1/4\}),
\end{split}
\end{equation}
where, for every $x\in \X$,
$$
\overline f_{i,r}(x):=\fint_{B_{r}(x)}f_i\de \meas.
$$
For $i$ large enough, since $f_i\to \nchi_{E}$ in $L^1$, we get that, for every $x\in \X$ and every $r>0$, we have
\begin{equation}
|\overline f_{i,r}(x)|\leq \frac{2\meas(E)}{\meas(B_r(x))}.
\end{equation}
Now fix $r:=(16\mm(E)/C_2)^{1/s}\leq R$, cf.\ \eqref{eqn:ChoiceOfConstants}. We have that for every $x\in \X$ the following holds
$$
\frac{2\meas(E)}{\meas(B_r(x))}\leq \frac{2\meas(E)}{C_2r^s}= \frac18 , 
$$
where in the inequality we used \eqref{eqn:EstimateBrX}, and in the equality we used the definition of $r$. Hence, if $r=(16\mm(E)/C_2)^{1/s}$ and $i$ is large enough we obtain that 
\begin{equation}\label{eqn:STIMA2}
|\overline f_{i,r}(x)|\leq \frac18,
\end{equation}
for every $x\in \X$.
 
Moreover, choosing $r$ as above, by using Markov inequality and \eqref{eqn:EnHancedPoincare}, which we can do since $\lip f_i$ is an upper gradient of $f_i$, we have 
\begin{equation}\label{eqn:STIMA3}
\begin{split}
\meas(\{x\in \X:|f_i(x)-\overline f_{i,r}(x)|\geq 1/4\})&\leq 4\int_\X|f_i-\overline f_{i,r}|\de \meas\leq  4C_1r\int_\X\lip f_i\de\meas \\
&= \frac{4\cdot 16^{1/s}C_1\mm(E)^{1/s}}{C_2^{1/s}}\int_\X\lip f_i\de\meas.
\end{split}
\end{equation}
Hence, letting $i\to+\infty$ in \eqref{eqn:STIMA1}, taking into account \eqref{eqn:STIMA2}, \eqref{eqn:STIMA3}, and the fact that $\meas(\{|f_i|\geq 1/2\})\to\meas(E)$, we conclude that 
$$
\meas(E)\leq \frac{4\cdot 16^{1/s}C_1\mm(E)^{1/s}}{C_2^{1/s}}P(E),
$$
from which the conclusion follows with the constant chosen in \eqref{eqn:ChoiceOfConstants}.
\end{proof}
\begin{remark}[\cref{prop:IsopVolumiPiccoli} can be applied to $\CD$ spaces with uniform bounds on the volumes of unit balls]\label{rem:IsoperimetricaSiPuoApplicareARCD}

As a consequence of Bishop--Gromov comparison theorem, see \cref{rem:PerimeterMMS2}, we get that all the $\CD(K,N)$ spaces $(\X,\dist,\meas)$, and thus also the $\RCD(K,N)$ spaces, with $K\in\mathbb R$ and $N<+\infty$, are uniformly locally doubling. More precisely, for every $R>0$, there exists $C$ such that 
\begin{equation}\label{eqn:EssentiallyNregular}
\frac{\meas(B_{r_2}(x))}{\meas(B_{r_1}(x))} \leq \frac{v(N,K/(N-1),r_2)}{v(N,K/(N-1),r_1)} \leq C\left(\frac{r_2}{r_1}\right)^N,\quad \forall x\in \X\;\forall r_1\leq r_2\leq R.
\end{equation}

Moreover, as a result of the work by Rajala in \cite{Rajala12}, all the $\CD(K,N)$ spaces, and thus also the $\RCD(K,N)$ spaces, with $K\in\mathbb R$ and $N\leq +\infty$ satisfy a weak local $(1,1)$ non-averaged Poincar\'{e} inequality in the sense above. In particular, by a careful inspection of the proof of \cref{lem:EnHancedPoincare}, and by the fact that the constant in the non-averaged Poincar\'{e} inequality for $\CD(K,N)$ spaces only depends on $K$ and $R$ (\cite{Rajala12}); if we apply \cref{lem:EnHancedPoincare} with $R=1$ to $\CD(K,N)$ spaces with $K\in\mathbb R$ and $N<+\infty$ the constant $C_1$ only depends on $K,N$.

In addition, as a consequence of \eqref{eqn:EssentiallyNregular}, we get that if $(\X,\dist,\meas)$ is a $\CD(K,N)$ space with $K\in\mathbb R$, $N<+\infty$, and such that there exists $v_0$ satisfying $\meas(B_1(x))\geq v_0$ for all $x\in \X$, we conclude that $(\X,\dist,\meas)$ is uniformly lower $N$-Ahlfors regular with the constant $C_2$ as in \eqref{eqn:EstimateBrX} only depending on $K,N,v_0$. 

In conclusion, the result in \cref{prop:IsopVolumiPiccoli}, with the choice $s=N$, can, and will, be applied on the class of $\CD(K,N)$ (and thus also on $\RCD(K,N)$ spaces) with $K\in\mathbb R$, $1< N<+\infty$, and with a uniform lower bound on the volumes of unit balls. Moreover, the constants $C,v$ in the statement of \cref{prop:IsopVolumiPiccoli} (cf. \eqref{eqn:ChoiceOfConstants} taking $R=1$ in there) only depend on $K,N,v_0$.

Finally notice that, for instance, all the examples discussed in the setting of \cite{AmbrosioAhlfors} fall in the hypotheses of \cref{prop:IsopVolumiPiccoli}.
\fr\end{remark}

\subsection{$(\Lambda,r_0,\sigma)$-minimality, $(K,r_0)$-quasi minimality and density estimates}\label{sec:LambdaMin}

Following the exposition in \cite[Chapter 21]{MaggiBook}, we introduce the following notions.

\begin{definition}[Quasi-perimeter minimizers and quasi minimal sets]\label{def:QuasiMinimi}
Let $(\X,\dist,\meas)$ be a metric measure space. Let $\Omega\subset \X$ be open, and let $E$ be a set of locally finite perimeter in $\X$. 

Given $\Lambda\geq 0$, $r_0>0$, and $\sigma>0$, we say that a set $E$ is a {\em $(\Lambda,r_0,\sigma)$-perimeter minimizer in $\Omega$} if for every $F\subset \X$ such that $E\Delta F\Subset B_r(x)\cap\Omega$, for some $x$ and $r<r_0$, we have
\begin{equation}\label{eq:DefLamRSigPerimeterMinimizer}
P(E,B_r(x))\leq P(F,B_r(x))+\Lambda\meas(E\Delta F)^\sigma.
\end{equation}

Given $K\geq 1$, $r_0>0$, we say that $E$ is a {\em $(K,r_0)$-quasi minimal set in $\Omega$} if for every $F\subset \X$ such that $E\Delta F\Subset B_r(x)\cap\Omega$, for some $x$ and $r<r_0$, we have
\begin{equation}\label{eq:DefKRQuasiMinimal}
P(E,B_r(x))\leq K P(F,B_r(x)).
\end{equation}
\end{definition}
Notice that in the above definition, since $E\Delta F\Subset B_r(x)\cap\Omega$, it is equivalent to require the localized inequalities
\begin{equation*}
    P(E,B_r(x)\cap \Omega)\leq P(F,B_r(x)\cap \Omega)+\Lambda\meas(E\Delta F)^\sigma,
\end{equation*}
in place of \eqref{eq:DefLamRSigPerimeterMinimizer}, and
\begin{equation*}
    P(E,B_r(x)\cap \Omega)\leq K P(F,B_r(x)\cap \Omega)
\end{equation*}
in place of \eqref{eq:DefKRQuasiMinimal}.

We are now committed to link the previous two notions in the following Remark.

\begin{remark}[$(\Lambda,r_0,\sigma)$-perimeter minimizer and Isoperimetric inequality for small volumes imply $(K,r_0')$-quasi minimal set]\label{rem:FromLambdaMintoKquasiMin}
    Let us assume $(\X,\dist,\meas)$ is a metric measure space on which there exist $v,\beta,C>0$ for which the following isoperimetric inequality for small volumes holds 
    $$
    \meas(E)\leq v \quad\Rightarrow\quad CP(E)\geq \meas(E)^\beta.
    $$
    Moreover let us assume that $f(r):=\sup_{x\in \X}\meas(B_r(x))<+\infty$, and that $f(r)\to 0$ as $r\to 0^+$.
    
    Then, on such a metric measure space $\X$ every  $(\Lambda,r_0,\sigma)$-perimeter minimizer in an open set $\Omega\subset \X$, with $\sigma\geq\beta$, is a $(K,r_0')$-quasi minimal set in $\Omega$ for some $K,r_0'$ depending on $\Lambda,r_0,\sigma,v,\beta,C,f(\cdot)$.
    
    Indeed, first of all choose $r_0'<r_0$ such that whenever $E\Delta F\Subset B_r(x)\cap \Omega$ for some $r<r_0'$ and some $F\subset \X$, we have that $\meas(E\Delta F)\leq v$. This can be done by taking $r_0'$ such that $f(r)\leq v$ on $(0,r_0']$.
    
    Then, let us take $F\subset \X$ such that $E\Delta F\Subset B_r(x)\cap\Omega$ with $r<r_0'$. Hence we have
    \begin{equation}\label{eqn:Iniziamo}
        \begin{split} 
        P(E,B_r(x))&\leq P(F,B_r(x))+\Lambda\meas(E\Delta F)^{\beta}\meas(E\Delta F)^{\sigma-\beta} \\
        &\leq P(F,B_r(x))+\Lambda C f(r)^{\sigma-\beta}P(E\Delta F,B_r(x)) \\
        &\leq P(F,B_r(x))+\Lambda C f(r)^{\sigma-\beta}\left(P(E,B_r(x))+P(F,B_r(x))\right)
        \end{split}
    \end{equation}
where in the last inequality we are using the subadditivity of the perimeter, i.e., for every $B_r(x)\subset \X$ we have
$$
P(E\Delta F,B_r(x))\leq P(E\cup F,B_r(x))+P(E\cap F,B_r(x))\leq P(E,B_r(x))+P(F,B_r(x)).
$$

Hence, from \eqref{eqn:Iniziamo} we get that 
\begin{equation}\label{eqn:Iniziamo2}
\left(1-\Lambda C f(r)^{\sigma-\beta}\right)P(E,B_r(x))\leq \left(1+\Lambda C f(r)^{\sigma-\beta}\right)P(F,B_r(x)).
\end{equation}

Hence, taking $r_0'$ smaller if needed, in such a way that $1-\Lambda C f(r_0')^{\sigma-\beta} >0$, the previous inequality tells us that $E$ is a $(K,r_0')$-quasi minimal set, with 
$$
K=\frac{1+\Lambda C f(r_0')^{\sigma-\beta}}{1-\Lambda C f(r_0')^{\sigma-\beta}}.
$$
\fr\end{remark}

From now on we shall consider $\meas=\haus^N$. Let us fix $\Omega\subset \X$ an arbitrary open set. Let us fix $\sigma>1-1/N$, and $C_G>0$ two constants. We consider
\[
G:\left\{\text{$\haus^N$-measurable sets in $\Omega$} \right\}/\sim \,\,\to (-\infty,+\infty],
\]
where $E\sim F$ if and only if $\haus^N(E\Delta F)=0$, where $E\Delta F$ denotes the symmetric difference between $E$ and $F$,
such that
\begin{equation}\label{eq:ConditionG2}
    \begin{split}
    G(\emptyset)&<+\infty,\\
    G(E)\le G(F) &+ C_G\haus^N(E\Delta F)^\sigma ,
    \end{split}
\end{equation}
for any Borel sets $E,F \subset \Omega$ such that $E\Delta F\Subset\Omega$. Notice that, in contrast with \eqref{eq:ConditionG}, we are now asking that the constants $C_G,\sigma$ are uniform with respect to the set $\widetilde\Omega$ in \eqref{eq:ConditionG}.

\begin{theorem}\label{thm:FromMinToQuasiMin}
    Let $(\X,\dist,\haus^N)$ be an $\RCD(K,N)$ space with $N\ge 2$ natural number, let $\Omega\subset \X$ be open, and let $\mathscr P=P+G$ be the quasi-perimeter restricted to $\Omega$ associated to $G$ as in \eqref{eq:ConditionG2}. Let $E\subset \X$ be a volume constrained minimizer of $\mathscr P$ in $\Omega$, see \cref{def:Minimizers}, and assume $P(E,\Omega)>0$.
    
    Hence there exist $\Lambda\geq 0$, and $r_0>0$ such that $E$ is a $(\Lambda,r_0,\min\{1,\sigma\})$-perimeter minimizer in $\Omega$, where $\sigma$ is the exponent provided by \eqref{eq:ConditionG2}. 
    
    In particular, if there exists $v_0>0$ such that $\mathcal{H}^N(B_1(x))\geq v_0$ for every $x\in \X$, there exist $K'\geq 1$, $r_0'>0$ such that $E$ is a $(K',r_0')$-quasi minimal set in $\Omega$.
\end{theorem}

\begin{proof}
Let $\partial^eE$ be the essential boundary of $E$. Since the perimeter measure $P(E,\cdot)$ is concentrated on $\partial^e E$, and since by hypothesis we have $P(E,\Omega)>0$, we can choose two distinct points $x_1,x_2\in \Omega\cap\partial^eE$. Hence, by exploiting \cref{prop:RelativeIsoperimetricInequality}, we can find a radius $s>0$ such that the balls $B_1\eqdef B_s(x_1)$ and $B_2\eqdef B_s(x_2)$ satisfy
$$
\min\{P(E,B_1),P(E,B_2)\}>0,
$$ 
$\dist(x_1,x_2)>5s$, and $B_1,B_2\Subset \Omega$. 
Since a fortiori $E$ is a volume constrained minimizer of $\mathscr{P}$ in $B_1,B_2$, by applying \cref{thm:InteriorAndExteriorPoints} we have that $E\cap B_i$ contains both exterior and interior points for $i=1,2$.

In the notation of \cref{thm:VariazioniMaggi} we define 
$$
\eta:=\min\{\eta_1(E,B_1),\eta_2(E,B_1),\eta_1(E,B_2),\eta_2(E,B_2)\},
$$
and $C':=\max\{C_1(E,B_1),C_2(E,B_1),C_1(E,B_2),C_2(E,B_2)\}$. By Bishop--Gromov volume comparison, together with the fact that the density of $\haus^N$ is $\leq 1$ everywhere, there exists $R>0$ such that $\sup_{x\in \X}\mathcal{H}^N(B_r(x))<\eta$ for every $r\in (0,R]$. Let us take $r_0:=\min\{s,R\}$. Let us show that $E$ is a $(\Lambda,r_0,\min\{1,\sigma\})$-perimeter minimizer in $\Omega$ for some choice of $\Lambda$ that will be clear through the proof. 

Take $F$ such that $E\Delta F\Subset B_r(x)\cap\Omega$ for some $r<r_0$. By the choice of $r_0$, the ball $B_r(x)$ can intersect at most one $B_i$ as a consequence of the triangle inequality. Let us assume without loss of generality that $B_r(x)$ does not intersect $B_1$. Moreover, by the choice of $r_0$, we have $\tau:=\mathcal{H}^N(E\cap B_r(x))-\mathcal{H}^N(F\cap B_r(x))\leq \mathcal{H}^N(E\Delta F)< \eta$. Then, from \cref{thm:VariazioniMaggi}, there exists $F_1$ such that $E\Delta F_1\Subset B_1$ and 
\begin{equation}\label{eqn:EstimateLeiLei}
\mathcal{H}^N(F_1\cap B_1)=\mathcal{H}^N(E\cap B_1)+\tau,\qquad P(F_1,B_1)\leq C'|\tau|+P(E,B_1).
\end{equation}
Let us then consider the competitor $\widetilde F:=(F_1\cap B_1)\cup (F\cap B_r(x)) \cup (E\setminus (B_r(x)\cup B_1))$. Since $\widetilde F\Delta E \Subset B_r(x)\cup B_1 \subset \Omega$, and $\mathcal{H}^N(\widetilde F\cap (B_r(x)\cup B_1))=\mathcal{H}^N(E\cap (B_r(x)\cup B_1))$, by hypotheses we get that
\begin{equation}\label{eqn:LAFINE}
    \begin{split}
        P(E,\Omega)+G(E\cap \Omega) &\leq P(\widetilde F,\Omega)+G(\widetilde F\cap  \Omega)= P(\widetilde F,B_1)+P(\widetilde F,\Omega\setminus B_1)+G(\widetilde F\cap\Omega) \\
        &\leq P(E,B_1)+C'|\tau|+P(F,\Omega\setminus B_1)+G(\widetilde F\cap\Omega)  \\
        &\leq P(F,\Omega)+C'|\tau|+G(E\cap\Omega)+C_G\mathcal{H}^N(\widetilde F\Delta E)^\sigma \\
        &\leq P(F,\Omega)+C'\mathcal{H}^N(E\Delta F)+G(E\cap\Omega)+2^\sigma C_G\mathcal{H}^N(E\Delta F)^\sigma \\
        &\leq P(F,\Omega)+G(E\cap\Omega)+\Lambda\mathcal{H}^N(E\Delta F)^{\min\{1,\sigma\}},
    \end{split}
\end{equation}
where in the second inequality we are using the locality of the perimeter, together with the fact that $\mathcal{H}^N((\widetilde F\Delta F_1)\cap B_1)=0$, $\mathcal{H}^N((\widetilde F\Delta F)\cap (\Omega\setminus B_1))=0$, and the estimate \eqref{eqn:EstimateLeiLei}; in the third inequality we are using the hypothesis on $G$ and again the locality of the perimeter together with $\mathcal{H}^N((E\Delta F)\cap B_1)=0$; in the fourth inequality we are using that $\mathcal{H}^N(\widetilde F\Delta E)\leq 2\mathcal{H}^N(E\Delta F)$ from how we constructed $\widetilde F$ and the fact that $\mathcal{H}^N(F_1\Delta E)=|\tau|$ due to the fact that the variations made in \cref{thm:VariazioniMaggi} either contain or are contained in $E$; and the last inequality holds for some choice of $\Lambda$ depending on $C,C',\sigma,\eta$ since $\mathcal{H}^N(E\Delta F)\leq \eta$. Hence from \eqref{eqn:LAFINE} we get the first part of the statement. 

The last part of the statement follows from the first one and \cref{rem:FromLambdaMintoKquasiMin} since the hypotheses of the Remark are met due to the validity of the isoperimetric inequality for small volumes established in \cref{prop:IsopVolumiPiccoli}, cf. \cref{rem:IsoperimetricaSiPuoApplicareARCD}, and the Bishop--Gromov comparison for volumes.
\end{proof}

In the following Remark we discuss a Bishop--Gromov comparison for perimeters in metric measure spaces $(\X,\dist,\mathcal{H}^N)$ that are $\RCD(K,N)$ spaces.
\begin{remark}\label{rem:ImprovedEstimatePerimeter}
On every $(\X,\dist,\mathcal{H}^N)$ that is an $\RCD(K,N)$ space we have that 
\begin{equation}\label{eqn:ComparisonPerimeter}
P(B_r(x))\leq s(N,K/(N-1),r), \quad \text{for all $x\in \X$ and every $r>0$}.
\end{equation}
Indeed, by \cite[Corollary 2.14]{DePhilippisGigli18} we have that 
$$
\lim_{r\to 0^+}\frac{\mathcal{H}^N(B_r(x))}{\omega_Nr^N}= \lim_{r\to 0^+}\frac{\mathcal{H}^N(B_r(x))}{v(N,K/(N-1),r)}\leq 1,
$$
for all $x\in \X$. Hence, by Bishop--Gromov monotonicity and the coarea formula one deduces 
$$
\frac{P(B_r(x))}{s(N,K/(N-1),r)}\leq \frac{\mathcal{H}^N(B_r(x))}{v(N,K/(N-1),r)}\leq 1,
$$
for every $x\in \X$. 
\fr\end{remark}

The following proposition can be obtained combining the observation in \cref{rem:FromLambdaMintoKquasiMin} and the result in \cite[Theorem 4.2]{KKLS13}. Nevertheless, we give here a more direct proof which is heavily inspired by \cite[Theorem 21.11]{MaggiBook}, and uses the Bishop--Gromov comparison results. First we recall the definition of Ahlfors regular set (with respect to a measure).

\begin{definition}[Ahlfors regular set]\label{def:Ahlfors}
Let $(\X,\dist)$ be a metric space and $\mu$ be a Borel measure on it. Let $\Omega\subset \X$ be an open set. We say that a set $S\subset \X$ is {\em $k$-Ahlfors regular in $\Omega$ with respect to $\mu$} if there exist constants $C\geq 1$ and $r_0>0$ such that
$$
C^{-1}r^k\leq \mu(B_r(x))\leq Cr^k,
$$
for every $x\in S$ and every $r<r_0$ such that $B_r(x)\Subset \Omega$. We say that $S\subset \X$ is {\em locally $k$-Ahlfors regular in $\Omega$ with respect to $\mu$} if for any ball $B\subset \Omega$ it holds that $S\cap B$ is $k$-Ahlfors regular in $B$ with respect to $\mu$.

We say that $S\subset \X$ is {\em $k$-Ahlfors regular in $\Omega$} if there exist constants $C\geq 1$ and $r_0>0$ such that 
$$
C^{-1}r^k\leq \mathcal{H}^k(S\cap B_r(x))\leq Cr^k,
$$
for every $x\in S$ and every $r<r_0$ such that $B_r(x)\Subset\Omega$. We say that $S\subset \X$ is {\em locally $k$-Ahlfors regular in $\Omega$} if for any ball $B\subset \Omega$ it holds that $S\cap B$ is $k$-Ahlfors regular in $B$.
\end{definition}

\begin{proposition}\label{prop:FromQuasiMinToOpen}
Let $(\X,\dist,\haus^N)$ be an $\RCD(K,N)$ space with $N\ge 2$ natural number, such that $\mathcal{H}^N(B_1(x))\geq v_0$ for every $x\in \X$. Let $\Omega\subset \X$ be open, and let $E$ be a $(\Lambda,r_0,\sigma)$-perimeter minimizer in $\Omega$ for some $\Lambda\geq 0$, $r_0>0$, and $\sigma>1-1/N$. Then there exist $0\leq C_1:=C_1(N,K,v_0)\leq 1$, $C_2:=C_2(N,K,v_0)\geq 1$, and $r_0':=r_0'(\Lambda,r_0,\sigma,N,K,v_0)$ such that 
\begin{equation}\label{eqn:ConclusionQuasiMin}
    \begin{split}
        C_1&\leq \frac{\mathcal{H}^N(E\cap B_r(x))}{\mathcal{H}^N(B_r(x))}\leq 1-C_1, \\
        C_2^{-1}&\leq \frac{P(E,B_r(x))}{r^{N-1}}\leq C_2,
    \end{split}
\end{equation}
whenever $x\in\Omega\cap\partial E^{(1)}$, and $B_r(x)\subset \Omega$ with $r<r_0'$. 

In particular, $E^{(1)}\cap \Omega$ is open, $\partial^e E \cap \Omega = \partial E^{(1)} \cap \Omega$, and $\partial E^{(1)}$ is $(N-1)$-Ahlfors regular in $\Omega$ with respect to the perimeter measure $P(E,\cdot)$.
\end{proposition}

\begin{proof}
Without loss of generality we can assume that $\haus^N(E\Delta\Omega)>0$.
The representative of $E$ given by $E^{(1)}$ satisfies 
$$
\min\big\{\mathcal{H}^N(B_r(x)\cap E^{(1)}),\mathcal{H}^N(B_r(x)\setminus E^{(1)})\big\}>0,
$$
for every $x\in \Omega\cap \partial E^{(1)}$.
Let us denote by $C$ the constant for which the isoperimetric inequality for small volumes holds, see \cref{prop:IsopVolumiPiccoli}, and \cref{rem:IsoperimetricaSiPuoApplicareARCD}. Let us moreover notice that $\sup_{x\in \X}\mathcal{H}^N(B_r(x))\leq v(N,K/(N-1),r)=:V(r)$ for every $r>0$ from Bishop--Gromov comparison theorem and the fact that the density of $\mathcal{H}^N$ is everywhere $\leq 1$. Hence, from \eqref{eqn:Iniziamo2}, we have that 
\begin{equation}\label{eqn:ControlQuasi}
\left(1-\Lambda C V(r)^{\sigma-(N-1)/N}\right)P(E,B_r(x))\leq \left(1+\Lambda C V(r)^{\sigma-(N-1)/N}\right)P(F,B_r(x)),
\end{equation}
whenever $E\Delta F\Subset B_r(x)\cap\Omega$, and $r<\min\{r_0,r_1,1\}=:r_0'$, where $r_1$ is small enough such that $V(r_1)\leq v$, where $v$ is the volume for which \cref{prop:IsopVolumiPiccoli}, see \cref{rem:IsoperimetricaSiPuoApplicareARCD}, holds; and such that $\Lambda CV(r_1)^{\sigma-(N-1)/N}\leq 1/2$. For the sake of simplicity let us call $k(r):=\Lambda C V(r)^{\sigma-(N-1)/N}$. 

Let us fix $x\in \Omega\cap \partial E^{(1)}$ and take $r<\min\{r_0',d(x,\partial\Omega)\}:=d$. Hence the function $m:(0,d)\to\mathbb R$ given by $m(r):=\mathcal H^N(B_r(x)\cap E^{(1)})$ is such that $0<m(r)<v(N,K/(N-1),d)$ for every $r\in (0,d)$ and $m'(r)=P(B_r(x),E^{(1)})$ for $\mathcal{H}^N$-a.e. $r\in (0,d)$ due to the coarea formula. For almost every $r\in (0,d)$ we have $P(E,\partial B_r(x))=0$, and $m$ is differentiable at $r$. Take such an $r$, and take $s\in (r,d)$. Hence, using \eqref{eqn:ControlQuasi} evaluated at the radius $s$ with the competitor $F:=E^{(1)}\setminus B_r(x)$, exploiting \cref{lem:per_inters_general}, and taking $s\to r^+$ we get 
\begin{equation}\label{eqn:CiSiamoEh}
(1-k(r))P(E,B_r(x))\leq (1+k(r))P(B_r(x),E^{(1)})\leq (1+k(r))s(N,K/(N-1),r),
\end{equation}
    where in the last inequality we are using the inequality in \eqref{eqn:ComparisonPerimeter}. Hence, since $k(r)\leq 1/2$ and $r<1$, from the inequality in \eqref{eqn:CiSiamoEh} we get that there exists a constant $C_2'$, depending only on $K$ and $N$, such that 
\begin{equation}\label{eqn:Finale1}
P(E,B_r(x))\leq C_2'r^{N-1}.
\end{equation}

Adding $(1-k(r))P(B_r(x),E^{(1)})$ to the first inequality of \eqref{eqn:CiSiamoEh}, taking into account \cref{lem:per_inters_general} and the fact that $k(r)\leq 1/2$, we conclude that 
$$
(2C)^{-1}m(r)^{(N-1)/N}\leq (1-k(r))P(E\cap B_r(x))\leq 2P(B_r(x),E^{(1)})=2m'(r),
$$
where the first inequality comes from the fact that $V(r)\leq V(r_1)\leq v$, and then we can use the isoperimetric inequality for small volumes proved in \cref{prop:IsopVolumiPiccoli}, see \cref{rem:IsoperimetricaSiPuoApplicareARCD}, on the set $E\cap B_r(x)$. Integrating the previous inequality, which we can do since $m(r)$ is positive and bounded on $(0,d)$, we get that there exists a constant $C_1'$, only depending on the constant $C$, such that
\begin{equation}\label{eqn:ThePrevious1}
\mathcal{H}^N(B_r(x)\cap E)\geq C_1'r^N.
\end{equation}
Since $\mathcal{H}^N(B_r(x))\leq V(r)\leq C_3r^{N}$ for some constant $C_3$ only depending on $N,K$, since $r<1$, the previous inequality directly implies
\begin{equation}\label{eqn:ThePrevious}
\frac{\mathcal{H}^N(B_r(x)\cap E)}{\mathcal{H}^N(B_r(x))}\geq C_1,
\end{equation}
for some constant $C_1\leq 1$ depending on $C_3$ and $C_1'$. Applying \eqref{eqn:ThePrevious} to $\Omega\setminus E$, which is still a $(\Lambda,r_0,\sigma)$-perimeter minimizer in $\Omega$, and noticing that the constant $C_1$ in fact only depends on $K,N,v_0$, we get the first of \eqref{eqn:ConclusionQuasiMin}.

Finally, from the relative isoperimetric inequality in \cref{prop:RelativeIsoperimetricInequality}, \eqref{eqn:ThePrevious1} and the analogue of \eqref{eqn:ThePrevious1} applied to $\Omega\setminus E$, we get, for some constant $C_{\mathrm{RI}}$ only depending on $N,K,v_0$, that 
\begin{equation}\label{eqn:IsopIsop}
P(E,B_r(x))\geq C_{\mathrm{RI}}\min\{\mathcal{H}^N(B_r(x)\cap E)^{(N-1)/N},\mathcal{H}^N(B_r(x)\setminus E)^{(N-1)/N}\}\geq C_2''r^{N-1}.
\end{equation}
Taking into account \eqref{eqn:Finale1} and \eqref{eqn:IsopIsop} we thus have the second bound in \eqref{eqn:ConclusionQuasiMin} for some constant $C_2$ only depending on $N,K,v_0$.

In order to conclude the proof, observe that for any $y \in E^{(1)}\cap \Omega \subset \overline{E^{(1)}} \cap \Omega$, \eqref{eqn:ConclusionQuasiMin} implies that $y\not\in \partial E^{(1)}$, for otherwise the density of $E$ at $y$ would be different from $1$. Hence $E^{(1)}\cap\Omega$ has interior and, in fact, $E^{(1)}\cap\Omega={\rm int}(E^{(1)})\cap\Omega$ is open. The same argument works on the set of density zero points $E^{(0)}\cap \Omega$, which turns out to be open. Therefore
\[
\Omega=(E^{(1)}\cap \Omega  )\sqcup 
(E^{(0)}\cap \Omega )\sqcup 
(\partial^eE \cap \Omega)
= (E^{(1)}\cap \Omega  )\sqcup 
(E^{(0)}\cap \Omega )\sqcup 
(\partial E^{(1)} \cap \Omega),
\]
and $\partial^e E\cap \Omega =\partial E^{(1)} \cap \Omega $. Finally, by \eqref{eqn:ConclusionQuasiMin} it follows that $\partial E^{(1)}$ is $(N-1)$-Ahlfors regular in $\Omega$ with respect to the measure $P(E,\cdot)$.
\end{proof}

We are ready to prove \cref{thm:MainIntro2}, after the next final observation.
\begin{remark}[Locality]\label{rem:Locality}
The results in \cref{thm:FromMinToQuasiMin} and \cref{prop:FromQuasiMinToOpen} have obvious local counterparts. Namely, if $\overline{B} \subset \X$ is a closed ball in an $\RCD(K,N)$ space $(\X,\dist,\haus^N)$, by compactness there holds a lower bound on the measure of unit balls centered at points in $\overline{B}$. Also, a relative isoperimetric inequality holds by \cref{prop:RelativeIsoperimetricInequality}, which gives a local counterpart of \cref{prop:IsopVolumiPiccoli}. This is enough to deduce local versions of \cref{thm:FromMinToQuasiMin} and \cref{prop:FromQuasiMinToOpen} holding on $\overline{B}$ with constants depending on the local lower bound on the measure of unit balls.
\fr\end{remark}

\begin{proof}[Proof of \cref{thm:MainIntro2}]
By \cref{thm:InteriorAndExteriorPoints}, $E\cap \Omega$ has both interior and exterior points. Applying locally  \cref{thm:FromMinToQuasiMin} and \cref{prop:FromQuasiMinToOpen}, taking into account \cref{rem:Locality}, it follows that $E^{(1)}\cap B \cap \Omega$ is open, $\partial^e E \cap B \cap \Omega = \partial E^{(1)} \cap B \cap \Omega $, and $ \partial E^{(1)} \cap B$ is $(N-1)$-Ahlfors in $B\cap \Omega$ with respect to $P(E,\cdot)$, for any ball $B$. Therefore the first part of the statement follows by the previous discussion, and taking into account that whenever $E$ is a set of finite perimeter in an $\RCD(K,N)$ space $(\X,\dist,\mathcal{H}^N)$, then $P(E,\cdot)=\mathcal{H}^{N-1}\llcorner \partial^eE$. Indeed, this follows by putting together the representation given in \cite[Theorem 5.3]{Ambrosio02} and the recent one contained in \cite[Corollary 4.2]{BPS19}.

If also \eqref{eq:ConditionG2} holds and there holds a uniform positive lower bound on the measure of unit balls, then the second part of the statement directly follows by applying \cref{prop:FromQuasiMinToOpen}.
\end{proof}

\subsection{Bounded representatives}\label{sec:Boundedness}

We first prove that volume constrained minimizers of quasi-perimeters are bounded, and then we conclude the proof of the main results of the paper putting together the results in the previous sections.

\begin{theorem}\label{thm:Boundedness}
Let $(\X,\dist,\mathcal{H}^N)$ be an $\RCD(K,N)$ space such that there exists $v_0$ for which $\mathcal{H}^N(B_1(x))\geq v_0$ for all $x\in \X$. Let $G$ satisfy the bound in \eqref{eq:ConditionG2Intro}, and let $\mathscr{P}=P+G$ be the quasi-perimeter associated to it.
    
Let $E$ be a volume constrained minimizer for $\mathscr{P}$, with $P(E)>0$, see \cref{def:Minimizers}. Hence there exists a bounded set $\widetilde E$ such that $\mathcal{H}^N(\widetilde E\Delta E)=0$.
\end{theorem}

\begin{proof}
Let $E$ be as in the statement. Let $\bar x\in \X$ be an interior point, which always exists by \cref{thm:InteriorAndExteriorPoints}. Let $R:=\sup\{s\in[0,+\infty):\mathcal{H}^N(E\cap B_s(\bar x))=\mathcal{H}^N(B_s(\bar x))\}$, and notice that $0<R<+\infty$ because $\bar x$ is an interior point and $P(E)>0$.

Let, for every $r>0$,
\[
V(r)\eqdef \mathcal{H}^N(E \setminus B_r(\bar x)), 
\qquad
A(r)\eqdef P(E, \X\setminus B_r(\bar x)).
\]
Since $\mathcal{H}^N(E)<+\infty$, there exists $r_0>0$ such that for any $r\ge r_0$ the volume $V(r)$ is sufficiently small to apply the isoperimetric inequality for small volumes proved in \cref{prop:IsopVolumiPiccoli}, cf. \cref{rem:IsoperimetricaSiPuoApplicareARCD}, on the set $E\setminus B_r(\bar x)$. In particular, for almost every $r\ge r_0$ we can write, by using the coarea formula and \cref{cor:per_inters_ball},
\[
| V'(r) | + A(r) = P(B_r(\bar x),E^{(1)}) + P(E, \X\setminus B_r(\bar x)) = P(E \setminus B_r(\bar x)) \ge C^{-1} V(r)^{\frac{N-1}{N}},
\]
where $C$ is the constant in \cref{prop:IsopVolumiPiccoli}, cf. \cref{rem:IsoperimetricaSiPuoApplicareARCD}.
We want to prove that
\begin{equation}\label{eq:GoalBddIsopRegion}
    A(r) \le | V'(r) |  + C_1 V(r)+C_2V(r)^\sigma ,
\end{equation}
for some constants $C_1,C_2$, and for almost every $r$ sufficiently big, where $\sigma>1-1/N$ is the coefficient in \eqref{eq:ConditionG2Intro}. Combining the previous two inequalities, in this way we would get, for almost every sufficiently big radii $r$, that
\[
C^{-1}V(r)^{\frac{N-1}{N}} \le  C_1 V(r)  + C_2V(r)^\sigma + 2| V'(r) | \le \frac{C^{-1}}{2}V(r)^{\frac{N-1}{N}} - 2 V'(r),
\]
because $| V'(r) | = - V'(r) $ and $ C_1 V(r)+C_2V(r)^\sigma\le \tfrac{C^{-1}}{2}V(r)^{\frac{N-1}{N}} $ for almost every sufficiently big radius $r$ because $\sigma>1-1/N$. Hence ODE comparison implies that $V(r)$ vanishes at some $r=\overline{r}<+\infty$, i.e., $E$ has a bounded representative, that is the sought claim.

So we are left to prove \eqref{eq:GoalBddIsopRegion}. From \eqref{eq:main_claim2NEW}, we have that there exist $\eps_0>0$ and $C_1>0$ such that for any $\eps\in (0,\eps_0)$ there is a radius $R<\widetilde R<R+1$ such that
\begin{equation}\label{eq:LocalModificationsSetsY}
    \mathcal{H}^N(E\cup B_{\widetilde R}(\bar x))=\mathcal{H}^N(E)+\eps,
    \qquad
    P(E\cup B_{\widetilde R}(\bar x)) \le P(E)+ C_1\eps.
\end{equation}
Notice that one can choose $\varepsilon_0:=\mathcal{H}^N(B_{R+1}(\bar x)\setminus E)$, and $C_1:=C_{K,N,R+1}/R$, where the constant $C_{K,N,R+1}$ is the one given in the estimate \eqref{eq:main_claim2NEW}. Moreover, notice that given $\varepsilon\in(0,\varepsilon_0)$, the choice of $R<\widetilde R<R+1$ is done in such a way that $\mathcal{H}^N(B_{\widetilde R}(\bar x)\setminus E)=\varepsilon$. 

Now consider any $r>R+1$ such that $V(r)<\eps_0$, and set $\eps:=V(r)$. Then there is $R<\widetilde R<R+1$ satisfying \eqref{eq:LocalModificationsSetsY}. Define $\widetilde{F}= (E\cup B_{\widetilde R}(\bar x))\cap B_r(\bar x)$, so that
\begin{equation}
\begin{split}
\mathcal{H}^N(\widetilde{F}) &= \mathcal{H}^N(E\cup B_{\widetilde R}(\bar x)) - \mathcal{H}^N(E\cup B_{\widetilde R}(\bar x)\setminus B_r(\bar x)) \\&= \mathcal{H}^N(E\cup B_{\widetilde R}(\bar x)) - \mathcal{H}^N(E\setminus B_r(\bar x)) = \mathcal{H}^N(E) + \eps - \eps = \mathcal{H}^N(E).
\end{split}
\end{equation}
By using the fact that $r>R+1>\widetilde R$, \cref{lem:per_inters_general}, and in particular its consequence in \cref{cor:per_inters_ball}, we get that for almost every choice of $r>R+1$ with $V(r)<\varepsilon_0$ we can perform the previous choice of $\widetilde F$ such that it holds that
\begin{equation}\label{eqn:Firstin}
\begin{split}
P(\widetilde{F}) =\,& P(E\cup B_{\widetilde R}(\bar x),B_r(\bar x)) + P(B_r(\bar x),E^{(1)}\cup B_{\widetilde R}(\bar x)) \\
=\,&P(E\cup B_{\widetilde R}(\bar x)) - P(E\cup B_{\widetilde R}(\bar x),\X\setminus B_r(\bar x)) - P(E\cup B_{\widetilde R}(\bar x),\partial B_r(\bar x)) \\
&{}+ P(B_r(\bar x),E^{(1)}\cup B_{\widetilde R}(\bar x)) \\
\leq\,& P(E\cup B_{\widetilde R}(\bar x)) - P(E\cup B_{\widetilde R}(\bar x),\X\setminus B_r(\bar x)) + P(B_r(\bar x),E^{(1)}\cup B_{\widetilde R}(\bar x)) \\
\leq\,& P(E)+C_1\varepsilon -A(r)+|V'(r)|,
\end{split}
\end{equation}
where in the last inequality we are using \eqref{eq:LocalModificationsSetsY}, and the facts that $P(B_r(\bar x),E^{(1)}\cup B_{\widetilde R}(\bar x))=P(B_r(\bar x),E^{(1)})$, and $P(E\cup B_{\widetilde R}(\bar x),\X\setminus B_r(\bar x))=P(E,\X\setminus B_r(\bar x))$, because $r>R+1>\widetilde R$. Notice moreover that by the hypothesis on $G$ we have, for some constants $C_G>0$ and $\sigma>1-1/N$, that the following holds
\begin{equation}\label{eqn:Secondin}
\begin{split}
G(\widetilde F)&\leq G(E)+C_G\mathcal{H}^N(E\Delta \widetilde F)^\sigma=G(E)+C_G\mathcal{H}^N(B_{\widetilde R}(\bar x)\setminus E\cup E\setminus B_r(\bar x))^\sigma  \\
&= G(E)+C_G(2\varepsilon)^\sigma,
\end{split}
\end{equation}
where in the last equality we are using that $\mathcal{H}^N(B_{\widetilde R}(\bar x)\setminus E)=\varepsilon$, and $\mathcal{H}^N(E\setminus B_r(\bar x))=V(r)=\varepsilon$.
Finally, since $E$ is a volume constrained minimizer we estimate
\[
\begin{split}
P(E)+G(E)
& \le P(\widetilde{F}) +G(\widetilde F) \le P(E) + C_1\eps - A(r) + |V'(r)| +G(\widetilde F)   \\
&\le P(E)+G(E)+C_1\varepsilon+ C_G(2\varepsilon)^\sigma -A(r)+|V'(r)|,
\end{split}
\]
where in the first inequality we are using \eqref{eqn:Firstin}, and in the second inequality we are using \eqref{eqn:Secondin}.
Hence we obtained that for almost every $r$ sufficiently big we have $A(r) \le |V'(r)| +  C_1 V(r)+C_G2^\sigma V(r)^\sigma$. Hence we see that \eqref{eq:GoalBddIsopRegion} holds for almost every $r$ sufficiently big, and with the choice $C_2=C_G2^\sigma$. Therefore, the proof is concluded.
\end{proof}

\begin{remark}
We stress that the previous technique to prove \cref{thm:Boundedness} is very likely to be adapted, under the same hypotheses on $\X$, for quasi-perimeters $\mathscr{P}$ restricted to open sets $\Omega$, where $\Omega$ satisfies the volume noncollapsedness condition $\inf_{x\in\Omega}\mathcal{H}^N(B_1(x)\cap\Omega)\geq v_0>0$, and where the volume constrained minimizer $E$ in $\Omega$ is such that $\mathcal{H}^N(E\cap\Omega)<+\infty$. Since this level of generality is out of the scope of this paper, we will not treat this case.
\fr\end{remark}

\begin{proof}[Proof of \cref{cor:FINALEIntro}]
The proof is a direct consequence of \cref{thm:Boundedness},  \cref{thm:FromMinToQuasiMin}, and \cref{prop:FromQuasiMinToOpen}.
\end{proof}

\begin{proof}[Proof of \cref{cor:Manifold}]
By assumptions, $E$ is a volume constrained minimizer in any ball $B\subset \Omega\setminus \partial M$. Also, for any $x \in \Omega \setminus\partial M$ there exists $r>0$ such that $B_r(x)$ has smooth boundary, it is diffeomorphic to a ball in $\R^n$, and $\partial B_r(x)$ is convex. Hence, by \cite[Theorem 1.1, Corollary 2.5, Corollary 2.6]{Han20}, it follows that $(\overline{B}_r(x),\dist_g,\vol)$ is an $\RCD(k,N)$ for some $k\in \R$ depending on $x$ and $r$, where $\dist_g$ is the geodesic distance, and $\vol$ the Riemannian volume measure on $M$. Therefore \cref{thm:MainIntro2} applies for $E$ on any such $B_r(x)$. Moreover, in any such ball we can apply \cref{thm:FromMinToQuasiMin}, which implies that $E$ is a $(\Lambda,r_0,\min\{1,\sigma\})$-perimeter minimizer. Since $\min\{1,\sigma\}>1-1/N$, we are in position to apply the classical regularity theory on $\Lambda$-perimeter minimizers developed in \cite[Theorem 1]{Tama82}, suitably adapted to smooth Riemannian manifolds. This completes the proof of the first part of the statement.

Assuming now either i) or ii) in the statement, by \cite{CorderoMcCann, SturmVonRenesse}, and \cite[Theorem 1.1, Corollary 2.5, Corollary 2.6]{Han20} we have that $(M,\dist,\mathcal{H}^N)$ is an $\RCD(K,N)$ space. Therefore the second part of the statement immediately follows from \cref{thm:MainIntro2} and \cref{cor:FINALEIntro}.
\end{proof}

\printbibliography

\end{document}